\documentclass[12pt,reqno]{amsart}
\usepackage{amsmath}
\usepackage{amssymb}
\usepackage{amstext}
\usepackage{a4wide}
\usepackage{graphicx}
\allowdisplaybreaks \numberwithin{equation}{section}
\usepackage{color}
\usepackage{cases}
\usepackage{hyperref}

\numberwithin{equation}{section}

\newtheorem{theorem}{Theorem}[section]
\newtheorem{proposition}[theorem]{Proposition}
\newtheorem{corollary}[theorem]{Corollary}
\newtheorem{lemma}[theorem]{Lemma}

\theoremstyle{definition}

\newtheorem{innercustomthm}{Theorem}
\newenvironment{customthm}[1]
{\renewcommand\theinnercustomthm{#1}\innercustomthm}
{\endinnercustomthm}

\theoremstyle{remark}
\newtheorem{remark}[theorem]{Remark}

\begin{document}

\title[Clustered helical vortices for 3D incompressible Euler equation]
{Clustered helical vortices for 3D incompressible Euler equation in infinite cylinders}

\author{Daomin Cao, Jie Wan}
	
\address{Institute of Applied Mathematics, Chinese Academy of Sciences, Beijing 100190, and University of Chinese Academy of Sciences, Beijing 100049,  P.R. China}
\email{dmcao@amt.ac.cn}
\address{School of Mathematics and Statistics, Beijing Institute of Technology, Beijing 100081,  P.R. China}
\email{wanjie@bit.edu.cn}

%\thanks{This work is partially supported by ARC}

\begin{abstract}
In this article, we first consider  solutions to % helical vortices with small cross-section to the 3D incompressible Euler equations in infinite pipes. By considering
a semilinear elliptic problem in divergence form
\begin{equation*}
\begin{cases}
-\varepsilon^2\text{div}(K(x)\nabla u)=  (u-q|\ln\varepsilon|)^{p}_+,\ \ &x\in \Omega,\\
u=0,\ \ &x\in\partial \Omega
\end{cases}
\end{equation*}
%Using the finite-dimensional reduction method,
for small values of $ \varepsilon $. We prove that  there exists a family of clustered solutions which have arbitrary many bubbles and collapse  into given maximum points  of  $ q^2\sqrt{\det K} $ as $ \varepsilon\to0. $ Then as an application, we construct clustered traveling-rotating helical vortex solutions to Euler equations in infinite cylinders, such that the support set of corresponding vortices consists of  several helical tubes concentrating near a single helix.
\vspace{0.3cm}


\vspace{0.3cm}
\end{abstract}

\maketitle

\section{Introduction and main results}
The motion of the incompressible ideal flow is governed by the following Euler
equations
\begin{equation}\label{Euler eq}
\begin{cases}
\partial_t\mathbf{v}+(\mathbf{v}\cdot \nabla)\mathbf{v}=-\nabla P,\ \ &D\times (0,T),\\
\nabla\cdot \mathbf{v}=0,\ \ &D\times (0,T),\\
\mathbf{v}\cdot \mathbf{n}=0,\ \ &\partial D\times (0,T),
\end{cases}
\end{equation}
where $ D\subseteq \mathbb{R}^3 $ is a domain with $ C^\infty $ boundary, $ \mathbf{v}=(v_1,v_2,v_3) $ is the velocity field, $ P$ is the scalar pressure, $ \mathbf{n} $ is the outward unit normal to $ \partial D $. % The third equation means that the net flux of velocity across the boundary is zero.
For velocity field $\mathbf{v}$, the corresponding vorticity field is   $ \mathbf{w}=\nabla\times \mathbf{v} $.  Then $ \mathbf{w} $ satisfies the vorticity equation (see \cite{MB})
\begin{equation}\label{Euler eq2}
\begin{split}
\partial_t\mathbf{w}+(\mathbf{v}\cdot \nabla)\mathbf{w}=(\mathbf{w}\cdot \nabla)\mathbf{v}.
\end{split}
\end{equation}

In this paper, we are concerned with concentrated clustered helical vortex solutions to Euler equation \eqref{Euler eq2}. The research of solutions to 3D Euler equations with helical symmetry has attracted great attention in the past decades, see \cite{AS,CW,DDMW,DDMW2,Du,ET,Jiu} and reference therein. Let us first define helical solutions, see \cite{CW2,ET}.  For fixed $ k>0 $, let $ \mathcal{G}_k=\{H_\rho:\mathbb{R}^3\to\mathbb{R}^3 \} $ be the helical transformation group, where
\begin{equation*}
H_\rho(x_1,x_2,x_3)^t=(x_1\cos\rho+x_2\sin\rho, -x_1\sin\rho+x_2\cos\rho, x_3+k\rho)^t.
\end{equation*}
Here $ A^t $ is the transposition of a matrix $ A $. 
Let
$ R_\rho=\begin{pmatrix}
\cos\rho & \sin\rho & 0 \\
-\sin\rho &\cos\rho & 0 \\
0 & 0 & 1
\end{pmatrix} $
be the rotation with respect to $ x_3 $-axis. Define the field of tangents of symmetry lines of $ \mathcal{G}_k $
\begin{equation*}
\overrightarrow{\zeta}= (x_2, -x_1, k)^t.
\end{equation*}

Helical solutions must define on helical domains. A domain $ D\in \mathbb{R}^3 $ is called  a  helical~domain, if $ H_\rho(D)=D $ for any $\rho\in\mathbb{R}$.
Let $ \Omega=D\cap\{x\mid x_3=0\} $ be the section of $ D $ over $ x_1Ox_2 $ plane. Then $ D  $ can be generated by $ \Omega $ by  letting $ D=\cup_{\rho\in\mathbb{R}}H_\rho(\Omega) $. In the following, we always assume that  $ \Omega $ is a simply-connected bounded domain with $ C^\infty $ boundary.

Helical solutions to \eqref{Euler eq} is then defined as follows. We say that  ($\mathbf{v}, P$) is   a  helical solution to \eqref{Euler eq} with pitch $ k $, if ($\mathbf{v}, P$) satisfies \eqref{Euler eq} and the  vector field $ \mathbf{v} $ and scalar function $ P $ satisfies for every $ \rho\in\mathbb{R}, x\in D  $
\begin{equation}\label{helical func}
P(H_{\rho}(x))=P(x);\ \ \mathbf{v}(H_{\rho}(x))=R_\rho \mathbf{v}(x).
\end{equation}
Moreover, we also impose  $ \mathbf{v} $ to satisfy  the following non-swirl condition:
\begin{equation}\label{ortho}
\mathbf{v}\cdot \overrightarrow{\zeta}=0.
\end{equation}
Under assumptions \eqref{helical func}  and \eqref{ortho}, it can be proved  that $ \mathbf{w} $ satisfies (see \cite{ET})
\begin{equation}\label{w formula}
\mathbf{w}=\frac{\omega}{k}\overrightarrow{\zeta},
\end{equation}
where $ \omega:=w_3=\partial_{x_1}v_2-\partial_{x_2}v_1 $, the third component of vorticity field $ \mathbf{w} $, is a helical function. Moreover, $ \omega $ satisfies the 2D vorticity equations
\begin{equation}\label{vor str eq}
\begin{cases}
\partial_t \omega+\nabla \omega\cdot \nabla^\perp\varphi=0, &\Omega\times (0,T),\\
\omega=\mathcal{L}_{K_H}\varphi, &\Omega\times (0,T),\\
\varphi|_{\partial \Omega}=0,
\end{cases}
\end{equation}
%\begin{equation}\label{vor str eq2}
%\begin{cases}
%\partial_t \omega+\nabla \omega\cdot \nabla^\perp\varphi=0, &\Omega\times (0,T),\\
%\omega=\mathcal{L}_H\varphi, &\Omega\times (0,T),\\
%\nabla^\perp\varphi\cdot \nu|_{\partial \Omega}=v_n|_{\partial \Omega},
%\end{cases}
%\end{equation}
where $ \varphi $ is the stream function, $ \perp $ denotes the clockwise rotation through $ \frac{\pi}{2}  $, $ \mathcal{L}_{K_H}\varphi=-\text{div}(K_H(x_1,x_2)\nabla\varphi) $ is   a second-order elliptic operator of divergence type with the coefficient matrix
\begin{equation}\label{coef matrix}
K_H(x_1,x_2)=\frac{1}{k^2+x_1^2+x_2^2}
\begin{pmatrix}
k^2+x_2^2 & -x_1x_2 \\
-x_1x_2 &  k^2+x_1^2
\end{pmatrix},
\end{equation}
%and $ \nu=(n_1,n_2) $ is the two-dimensional vector of the first two component of $ \mathbf{n} $ on $ \partial \Omega $, which is an outward normal to $ \partial \Omega $.
see \cite{CW2} for more details. For  helical solution pairs ($\mathbf{v}, P$) to \eqref{Euler eq}, it suffices to solve solutions $ \omega $ to \eqref{vor str eq}.
%Note also that when $ v_n\equiv 0 $, which corresponds to the impermeable boundary condition, the third equation of \eqref{vor str eq2} implies that $ \varphi$  is   0 on $ \partial \Omega $. Thus we get the  vorticity equations under the impermeable boundary condition
Note that since $ \mathcal{L}_{K_H} $ is a uniformly elliptic operator and has the same $ L^p $ estimates as $ -\Delta $, many references get similar well-posedness and stability results of solutions to \eqref{vor str eq} as those to 2D Euler equations. \cite{ET} proved the global well-posedness of $ L^1\cap L^\infty $ weak solutions to  \eqref{vor str eq}, which coincides with the classical Yudovich's result \cite{Y} in 2D Euler flows. \cite{Ben} considered nonlinear stability of stationary smooth Euler flows with helical symmetry  by using the direct method of Lyapunov. For more results of the existence and regularity of helical solutions to Euler equations, see \cite{AS, BLN, Du, Jiu}.

The problem of concentrated helical vortex solutions to 3D Euler equations, meanwhile,  has been widely concerned in recent years, see \cite{CW,CW1,DDMW2,GM} and reference therein. It is also called the vortex filament conjecture (see \cite{JS2}) to  3D Euler equations with helical symmetry, that is,  constructing ``true'' helical solutions to Euler equations  such that the corresponding vorticity  concentrates near a helix. The research of this problem  can be traced back to
Helmholtz \cite{He}
%, who  first studied the motion of the traveling vortex rings  whose vorticities are supported in toroidal regions with a small  cross-section. Then
and then many articles proved the existence of vortex solutions to Euler equations concentrating near a straight line and a circle, see \cite{CLW,CPY,DDMW,DV,Fr1,FB,SV} for example. %This problem  is called the vortex desingularization problem to 2D Euler equation and 3D axisymmetric Euler equation.
For concentrated solutions to Euler equations with helical symmetry,  D$\acute{\text{a}}$vila et al.  \cite{DDMW2} constructed rotational-invariant smooth Euler flows with helical symmetry in the whole space. For $ \alpha \in\mathbb{R} $, consider  rotating-invariant solutions to \eqref{vor str eq} being of the form
\begin{equation}\label{104}
\omega(x,t)=w\left (\bar{R}_{-\alpha|\ln\varepsilon| t}(x)\right );\ \ \varphi(x,t)=u\left (\bar{R}_{-\alpha|\ln\varepsilon| t}(x)\right ),
\end{equation}
where $ x=(x_1,x_2)\in \mathbb{R}^2$ and $ \bar{R}_{\alpha t}=\begin{pmatrix}
\cos\alpha t & \sin\alpha t \\
-\sin\alpha t &\cos\alpha t
\end{pmatrix} $. Taking \eqref{104} into \eqref{vor str eq}, we get
\begin{equation}\label{rot eq-1}
\begin{cases}
\nabla w\cdot \nabla^\perp \left( u-\frac{\alpha}{2}|x|^2|\ln\varepsilon|\right) =0,\\
w=\mathcal{L}_{K_H}u.
\end{cases}
\end{equation}
So formally if
\begin{equation}\label{105-1}
\mathcal{L}_{K_H} u=w=f_\varepsilon\left( u-\frac{\alpha}{2}|x|^2|\ln\varepsilon|\right) \ \ \text{in}\ \mathbb{R}^2
\end{equation}
for some function $ f_\varepsilon $, then \eqref{rot eq-1} automatically holds. By taking  $ f_\varepsilon(t)=\varepsilon^2e^t $ and using the Lyapunov-Schmidt reduction method, the authors proved the existence of solutions to \eqref{105-1} concentrating near several distinct points in the distributional sense as $ \varepsilon\to0 $.
Note that by the choice of  $ f_\varepsilon $, the support set of vorticity is still the whole plane. Recently, \cite{CW} considered
rotational-invariant  concentrated solutions with small cross-section to \eqref{vor str eq} in an infinite cylinder $ B_{R^*}(0)\times\mathbb{R} $. Similar to the deduction of \eqref{105-1}, it suffices to solve a semilinear elliptic equations in divergence form
\begin{equation}\label{rot eq2}
\begin{split}
-\text{div}(K(x)\nabla u)=f_\varepsilon\left(u-q |\ln\varepsilon|\right)\ \ \text{in}\  \Omega;\ \
u(x)=0\ \ \text{on}\  \partial \Omega
\end{split}
\end{equation}
for some function $ f_\varepsilon  $, where  $ K $ is a positive-definite matrix and the function $ q>0 $. Denote   $ \det K $ the determinant of $ K $. By choosing $ f_\varepsilon(t)=\frac{1}{\varepsilon^2}t^p_+ $ for $ p>1 $ and constructing asymptotic expansion of Green's function $ G_K $ of the elliptic operator $ -\text{div}\cdot(K(x)\nabla  ) $ being of the form
\begin{equation*}
G_K(x, y)=\frac{\sqrt{\det K(x)}^{-1}+\sqrt{\det K(y)}^{-1}}{2}\Gamma\left (\frac{T_x+T_y}{2}(x-y)\right )+S_K(x,y),
\end{equation*}
where $ \Gamma(x)=-\frac{1}{2\pi}\ln|x| $, $ (T_x)^{-1}(T_x)^{-t}=K(x) $ and $ S_K(x,y)\in C^{0,\gamma}(\Omega\times\Omega) $ for $ \gamma\in(0,1) $, \cite{CW} proved the following results:
\begin{customthm}{A}[\cite{CW}]\label{thm1A}
For any given $ l $ distinct strict local minimum (maximum) points $ x_{0,j}(j=1,\cdots, l)$ of $ q^2\sqrt{\det(K)} $ in $ \Omega $, there exists $ \varepsilon_0>0 $  such that for every $ \varepsilon\in(0,\varepsilon_0] $, \eqref{rot eq2} has a solution $ u_\varepsilon $ satisfying
\begin{enumerate}
	\item  Define $ \bar{A}_{\varepsilon,i}=\left\{u_\varepsilon>q\ln\frac{1}{\varepsilon}\right\}\cap B_{\bar{\rho}}(x_{0,i})  $, where $ \bar{\rho}>0 $ is small. Then there exist $ (z_{1,\varepsilon}, \cdots, z_{l,\varepsilon}) $ and $ R_1,R_2>0 $ independent of $ \varepsilon $ satisfying
	\begin{equation*}
	\lim_{\varepsilon\to 0}(z_{1,\varepsilon}, \cdots, z_{l,\varepsilon})=(x_{0,1}, \cdots, x_{0,l});\ \ 	 B_{R_1\varepsilon}(z_{i,\varepsilon})\subseteq \bar{A}_{\varepsilon,i}\subseteq B_{R_2\varepsilon}(z_{i,\varepsilon}).
	\end{equation*}	
	\item
	\begin{equation*}
	\lim_{\varepsilon\to 0}\frac{1}{\varepsilon^2}\int_{B_{\bar{\rho}}(x_{0,i})}\left( u_\varepsilon-q|\ln\varepsilon|\right)^{p}_+dx=2\pi q\sqrt{\det(K)}(x_{0,i}).
	\end{equation*}
\end{enumerate}
\end{customthm}
By choosing $ K=K_H $, $ q=\frac{\alpha}{2}|x|^2+\beta $ for some constants $ \alpha,\beta $ in Theorem \ref{thm1A}, the authors constructed  multiple traveling-rotating helical vortices  in $ B_{R^*}(0)\times \mathbb{R} $ with polygonal symmetry. Note that the concentrating locations $ x_{0,j} $ are $l $  distinct  points, which constitute the vertices of a regular polygon. More results can be seen in \cite{CW1, CW2}.

Existing results indicate that there exist concentrated helical  vortex solutions concentrating near several $ distinct $ helices in $D$. So here comes a natural question, are there helical vortex solutions to \eqref{Euler eq}, whose support sets consist of several helical tubes and collapse into a single helix as parameter $ \varepsilon\to 0 $? We call this kind of solutions the $ clustered $ helical solutions. From the deduction of \eqref{rot eq2}, the question becomes whether there exists a family of solutions to \eqref{rot eq2}, such that solutions consist of several bubbles which collapse into a single point as $ \varepsilon\to 0 $. Note that when $ K(x)\equiv Id $,  \eqref{rot eq2} becomes vorticity equations of  2D Euler equations. In this case, classical results (see \cite{CGPY,CPY}) indicate that limiting locations of concentrated solutions must be critical points of the Kirchhoff-Routh function, which are $ l $ distinct points in $\Omega$. So clustered solutions to 2D Euler equations do not exist. As for vortex rings to 3D Euler equations, \cite{ALW} constructed smooth clustered solutions shrinking to a circle. Very recently, by choosing proper $ f_\varepsilon $ in \eqref{105-1}, \cite{GM}   constructed  smooth clustered  solutions  to \eqref{105-1} shrinking to a single point in $ \mathbb{R}^2 $, which correspond to clustered helical solutions in $ \mathbb{R}^3 $. However, because of the choice of the vortex profile $ f(t)=e^t $, it is not sure that the support sets of  vortex solutions constructed in \cite{ALW,GM} are included in a vortex tube with small cross-section.

Our goal in this paper is to construct clustered helical  solutions to Euler equation \eqref{Euler eq} with small cross-section in helical domains, such that the support of vortices consists of several helical tubes   collapsing  into a single helix as $ \varepsilon\to 0 $. From the deduction of \eqref{rot eq2}, it suffices to construct clustered solutions to a semilinear elliptic equations in divergence form \eqref{rot eq2}.  We prove that, suppose that $ x_0 $ is a strict local maximizer of $ q^2\sqrt{\det K} $ in $\Omega$, then for any positive integer $ m $ there exists a family of clustered solutions concentrating near $ m $ points $ (z_{1,\varepsilon},z_{2,\varepsilon},\cdots,z_{m,\varepsilon}) $, which satisfy  $ \lim_{\varepsilon\to 0}z_{i,\varepsilon} =x_0 $ for $ i=1,\cdots,m $. The key of proof is to get $ C^1-$asymptotic expansion of Green's function $ G_K $ of the operator $ -\text{div}(K(x)\nabla  ) $ (see Lemma \ref{Green expansion2}), the $ C^1- $dependence of the error term $ \omega_{\delta,Z} $ with respect to $ Z $ (see Proposition \ref{differ of w}) and the existence of critical points of energy $ K_\delta(Z) $ (see Proposition \ref{order of main term}).
Therefore in our construction, solutions consist of several bubbles  concentrating near  a single point rather than $ m $ distinct points, which is quite different from known results in \cite{CW,CW1,DDMW}.

Now we begin to show our main results. Let us consider  clustered solutions to a semilinear elliptic equation in divergence form
\begin{equation}\label{eq1-1}
\begin{cases}
-\varepsilon^2\text{div}(K(x)\nabla u)= (u-q|\ln\varepsilon|)^{p}_+,\ \ &x\in \Omega,\\
u=0,\ \ &x\in\partial \Omega,
\end{cases}
\end{equation}
where $ \Omega\subset \mathbb{R}^2 $ is a smooth bounded domain, $ \varepsilon\in(0,1) $ and $ p>1 $. $ K=(K_{i,j})_{2\times2} $ is a positive-definite smooth matrix satisfying
\begin{enumerate}
	\item[($\mathcal{K}$1).] $ -\text{div}(K(x)\nabla \cdot) $ is a uniformly elliptic operator, that is, there exist  $ \Lambda_1,\Lambda_2>0 $ such that $$ \Lambda_1|\zeta|^2\le (K(x)\zeta|\zeta) \le \Lambda_2|\zeta|^2,\ \ \ \ \forall\ x\in \Omega, \ \zeta\in \mathbb{R}^2.$$
\end{enumerate}
$ q  $ is a function defined in $ \overline{\Omega} $ satisfying
\begin{enumerate}
	\item[(Q1).]  $ q \in C^{\infty}(\overline{\Omega}) $ and $ q(x)>0 $ for any $ x\in\overline{\Omega}. $
\end{enumerate}
Our first result is as follows.
\begin{theorem}\label{thm1}
Suppose that ($\mathcal{K}$1) and (Q1) hold. Let $ x_0 $ be a strict local maximum point of $ q^2\sqrt{\det(K)} $ in $ \Omega $, i.e., there exists $ \bar{\rho}>0 $ small such that
\begin{equation*}
q^2\sqrt{\det(K)}(y)< q^2\sqrt{\det(K)}(x_{0})\ \ \ \ \forall y\in B_{\bar{\rho}}(x_{0})\backslash\{x_{0}\}.
\end{equation*} Then, for any $ m\in \mathbb{N}^* $  there exists $ \varepsilon_0>0 $, such that for every $ \varepsilon\in(0,\varepsilon_0] $, \eqref{eq1-1} has a family of clustered solutions $ u_\varepsilon $ with
\begin{equation*}
\frac{1}{\varepsilon^2}\int_{\Omega}\left( u_\varepsilon-q|\ln\varepsilon|\right)^{p}_+dx\to 2\pi m q\sqrt{\det K}(x_{0})  \ \ \text{as}\ \varepsilon\to0.
\end{equation*}
Moreover, there exist $ (z_{1,\varepsilon}, \cdots, z_{m,\varepsilon})\in \Omega^{(m)} $ such that
\begin{equation*}
|z_{i,\varepsilon}-z_{j,\varepsilon}|\geq |\ln\varepsilon|^{-m^2-1},\ \ \forall i\neq j; \ \ \left \{u_\varepsilon>q|\ln\varepsilon|\right \}\subseteq \cup_{i=1}^m B_{|\ln\varepsilon|^{-m^2-2}}(z_{i,\varepsilon})
\end{equation*}
and
\begin{equation*}
\lim_{\varepsilon\to 0}(z_{1,\varepsilon}, \cdots, z_{m,\varepsilon})=(x_{0}, \cdots, x_{0}).
\end{equation*}	
Define the  set $ A_{\varepsilon,i}=\left\{u_\varepsilon>q|\ln\varepsilon|\right\}\cap B_{|\ln\varepsilon|^{-m^2-2}}(z_{i,\varepsilon})  $. Then there exist constants $ R_1,R_2>0 $ independent of $ \varepsilon $  such that
\begin{equation*}
B_{R_1\varepsilon}(z_{i,\varepsilon})\subseteq A_{\varepsilon,i}\subseteq B_{R_2\varepsilon}(z_{i,\varepsilon}).
\end{equation*}
\end{theorem}

\begin{remark}
It is quite surprising that accumulation of bubbles can occur for system \eqref{eq1-1}. When $ K\equiv Id $, this phenomenon does not exist, see \cite{CGPY}. The only known result for such phenomena is due to \cite{GM}. Note that the construction of clustered solutions in \cite{GM} depends on the choice of $ f_\varepsilon $ and the accurate expression of $ K_H $ in \eqref{coef matrix}. In contrast to \cite{GM}, we get clustered solutions to equations \eqref{eq1-1} with any positive-definite matrix $ K $ by using $ C^1 $-asymptotic estimates of Green's function $ G_K $. Another interesting phenomenon  is the multiplicity of solutions to \eqref{eq1-1}. Indeed, Theorem \ref{thm1} shows that there exists solutions of \eqref{eq1-1} with arbitrarily
many bubbles   at given local maximum points of $ q^2\sqrt{\det K} $. So the quantity
\begin{equation*}
\frac{1}{\varepsilon^2}\int_{\Omega}\left( u_\varepsilon-q|\ln\varepsilon|\right)^{p}_+dx
\end{equation*}
can tend to $ +\infty $ as $ \varepsilon\to0 $. These results show a striking difference with the classical results in \cite{CW,CW1,DDMW2}.
\end{remark}

\begin{remark}
In \cite{WYZ}, the authors considered an anisotropic Emden-Fowler equation
\begin{equation*}
\begin{cases}
\text{div}(a(x)\nabla u)+\varepsilon^2a(x)e^u=0& \text{in}\ \Omega,\\
 u=0& \text{on}\ \partial\Omega,
 \end{cases}
\end{equation*}
%By using the expansion of Green's function of $ \nabla(a(x)\nabla ) $, the authors prove the existence of clustering solutions.
where $ a $ is a smooth positive function in $ \Omega $. For any given maximum points $ x^* $ of $ a $,  the authors constructed clustered 
solutions concentrating near $ x^*  $. Note that when choosing $ K(x)=a(x)Id $ in \eqref{eq1-1},   results  in Theorem \ref{thm1} 
coincides with those in \cite{WYZ}. %Thus to some extent, Theorem \ref{thm1} extends results in \cite{WYZ}.
\end{remark}
Our strategy of proof for Theorem \ref{thm1} is as follows. Set $ \delta=\varepsilon|\ln\varepsilon|^{-\frac{p-1}{2}} $ and $ u=|\ln\varepsilon|v $, then \eqref{eq1-1} becomes
\begin{equation}\label{111}
\begin{cases}
-\delta^2\text{div}(K(x)\nabla v)=  (v-q)^{p}_+,\ \ &x\in \Omega,\\
v=0,\ \ &x\in\partial \Omega.
\end{cases}
\end{equation}
To get solutions of \eqref{111}, we first give a $ C^1 $-expansion of Green's function for the elliptic operator $ -\text{div}(K(x)\nabla) $, see Lemma \ref{Green expansion2}. Then we construct approximate solutions   $ \sum_{j=1}^m(V_{\delta, z_j, \hat{q}_j}+H_{\delta, z_j, \hat{q}_j})+\omega_{\delta,Z} $, where $ V_{\delta, z_j, \hat{q}_j} $, $ H_{\delta, z_j, \hat{q}_j} $ and $ \omega_{\delta,Z}$ are the main term,  projection term and  error term respectively with $Z=(z_1,\cdots,z_m)$. Note that the admissible class $ \Lambda_{\varepsilon, m} $ must be chosen properly. Under the choice of $ H_{\delta, z_j, \hat{q}_j} $, we get the  equation \eqref{3-03} for $ \omega_\delta $. We also choose  $ \hat{q}_j $ properly  to ensure that $ V_{\delta,Z}-q $ is  close to  $V_{\delta, z_i, \hat{q}_{\delta,i}}-\hat{q}_{\delta,i} $, see \eqref{203}.   Using the non-degeneracy of solutions to \eqref{eq5}, we get the existence and uniqueness of $ \omega_\delta $, see Lemma \ref{coercive esti} and Proposition \ref{exist and uniq of w}. To prove that the energy functional is $ C^1 $ with respect to the variable $ Z $, we need to prove the differentiability of $ \omega_{\delta,Z}  $ about $ Z $, which is shown in Proposition \ref{differ of w}. Finally it suffices to solve a finite dimension problem. By calculating the main term of the energy $ K_\delta(Z) $ and choosing test functions as  the vertices of a $ m $–sided regular polygon,  we get the existence of critical points of $  K_\delta(Z)  $, which correspond to a family of clustered solutions to \eqref{111}.

As an application of Theorem \ref{thm1}, for any given $R^*>0$, we get clustered helical  rotational-invariant solutions with small cross-section to Euler equations \eqref{Euler eq} in the infinite cylinder $ B_{R^*}(0)\times\mathbb{R} $.  Let $ \alpha\in\mathbb{R} $.  We look for rotating-invariant solution pairs $ (\omega,\varphi) $ to \eqref{vor str eq} with angular velocity $ \alpha|\ln\varepsilon| $, that is, $ (\omega,\varphi) $ satisfies  \eqref{104}. From the deduction of \eqref{105-1}, if $ u $ solves
\begin{equation*}\label{rot eq3}
\begin{cases}
-\text{div}\cdot(K_H(x)\nabla u)=f_\varepsilon\left(u-\frac{\alpha}{2}|x|^2 |\ln\varepsilon|\right),\  &x\in B_{R^*}(0),\\
u(x)=0,\ &x\in \partial B_{R^*}(0)
\end{cases}
\end{equation*}
for some function $ f_\varepsilon  $, then $ \omega(x,t)=w\left (\bar{R}_{-\alpha|\ln\varepsilon| t}(x)\right )$ and  $ \varphi(x,t)=u\left (\bar{R}_{-\alpha|\ln\varepsilon| t}(x)\right )  $ satisfy vorticity equations \eqref{vor str eq}, which corresponds helical solutions to Euler equations \eqref{Euler eq}. By choosing $ f_\varepsilon(t)=\frac{1}{\varepsilon^2}(t-\beta|\ln\varepsilon|)^p_+ $ for some $ \beta\in\mathbb{R} $, we get
\begin{theorem}\label{thm2}
Let $ R^*$ and $k$ be two given positive numbers. Suppose that $\alpha, \beta\in \mathbb{R}  $  are two numbers such that  $ \min_{x\in B_{R^*}(0)}\left( \frac{\alpha|x|^2}{2}+\beta\right)>0  $ and that $ \left( \frac{\alpha|x|^2}{2}+\beta\right)^2\sqrt{\det K_H} $ has a strict local maximum point $ x_0\in B_{R^*}(0)$ up to a rotation, i.e.,   $ |x_0| $ is a strict local maximum  point of $ \left( \frac{\alpha|x|^2}{2}+\beta\right)^2\sqrt{\det K_H} $ in $ [0,{R^*}) $.   Then for any $ m\in \mathbb{N}^* $  there exists $ \varepsilon_0>0 $ such that for $ \varepsilon\in(0,\varepsilon_0] $, \eqref{Euler eq} has a family of clustered helical Euler flows $ (\mathbf{v}_\varepsilon, P_\varepsilon)(x,t)\in C^1(B_{R^*}(0)\times \mathbb{R}) $. Moreover, the associated vorticity-stream function pair $ (\omega_\varepsilon,\varphi_\varepsilon) $ is a  rotational-invariant solution to \eqref{vor str eq} with the following properties:
	\begin{enumerate}
		\item The angular velocity is $ \alpha|\ln\varepsilon| $ and the circulations satisfy as $\varepsilon\to 0$
		\begin{equation*}
	\int_{B_{R^*}(0)} \omega_\varepsilon dx\to \frac{\pi km(\alpha|x_0|^2+2\beta
		)}{\sqrt{k^2+|x_0|^2}}.
		\end{equation*}
		\item There exist $ (z_{1,\varepsilon}, \cdots, z_{m,\varepsilon})\in B_{R^*}(0)^{(m)} $ such that
		\begin{equation*}
		|z_{i,\varepsilon}-z_{j,\varepsilon}|\geq |\ln\varepsilon|^{-m^2-1},\ \ \forall i\neq j; \ \ \text{supp}(\omega_\varepsilon)\subseteq \cup_{i=1}^m B_{|\ln\varepsilon|^{-m^2-2}}(z_{i,\varepsilon})
		\end{equation*}
		and
		\begin{equation*}
		\lim_{\varepsilon\to 0}(z_{1,\varepsilon}, \cdots, z_{m,\varepsilon})=(x_{0}, \cdots, x_{0}).
		\end{equation*}	\\
		\item  There exist constants $ R_1,R_2>0 $ independent of $ \varepsilon $  such that
		\begin{equation*}
		B_{R_1\varepsilon}(z_{i,\varepsilon})\subseteq \text{supp}(\omega_\varepsilon)\cap B_{|\ln\varepsilon|^{-m^2-2}}(z_{i,\varepsilon})\subseteq B_{R_2\varepsilon}(z_{i,\varepsilon}).
		\end{equation*}
	\end{enumerate}
\end{theorem}

A consequence
%, by choosing proper constants $ \alpha$ and  $\beta$ in 
of Theorem \ref{thm2} is the existence of rotational-invariant clustered helical vorticity solutions to 3D incompressible Euler equation in infinite cylinders, whose support sets consist of several helical tubes and collapse into  $ x_3- $axis as parameter $ \varepsilon\to 0 $.
\begin{corollary}\label{cor3}
Let $ R^*$ and $k$ be two given positive numbers. Let  $ \alpha$ and  $\beta$ be constants such that $ \alpha\leq 0 $ and $ \min_{x\in B_{R^*}(0)}\left( \frac{\alpha|x|^2}{2}+\beta\right)>0 $. Then for any $ m\in \mathbb{N}^* $  there exists $ \varepsilon_0>0 $ such that for $ \varepsilon\in(0,\varepsilon_0] $, \eqref{Euler eq} has a family of clustered helical Euler flows $ (\mathbf{v}_\varepsilon, P_\varepsilon)(x,t)\in C^1(B_{R^*}(0)\times \mathbb{R}) $. Moreover, the associated vorticity field $ \mathbf{w}_\varepsilon=\nabla\times \mathbf{v}_\varepsilon  $ is a  rotational-invariant solution to \eqref{Euler eq2} with angular velocity $ \alpha|\ln\varepsilon| $ whose support set consists of $ m $ helical tubes with pitch $ k $ and collapses into  $ x_3- $axis as  $ \varepsilon\to 0 $, and the circulation of $ \mathbf{w}_\varepsilon $ tends to $ 2\pi m\beta $ as $ \varepsilon\to0 $. 
\end{corollary}

%We now give some examples of the choice of $ \alpha$ and  $\beta$ so that the assumptions of Theorem \ref{thm2} are satisfied. 
The idea of  proof is as follows. We choose constants $ \alpha$ and  $\beta$ so that $ \alpha<0 $ and $ \min_{x\in B_{R^*}(0)}\left( \frac{\alpha|x|^2}{2}+\beta\right)>0 $ in Theorem \ref{thm2}. Direct computations show that $ (0,0) $ is a strict local maximum point of $ \left( \frac{\alpha|x|^2}{2}+\beta\right)^2\sqrt{\det K_H} $ up to a rotation. From Theorem \ref{thm2}, there exist clustered helical solutions concentrating near $ x_3- $axis. This phenomenon is not found in any existing literatures.
\begin{remark}
Indeed, it is also possible to construct clustered steady helical solutions to Euler equations \eqref{Euler eq} in general  helical domains, see \cite{Ben,CW2}. Moreover, it is interesting whether there exist clustered solutions to \eqref{vor str eq} with different vortex profiles, such as vortex patch solutions.
\end{remark}

\begin{remark}
In \cite{MR}, Martel and Rapha$\ddot{\text{e}}$l considered the  existence of clustered solutions for the mass critical two dimensional nonlinear Schr\"odinger equation
\begin{equation}\label{NLS}
i\partial_t u+\Delta u+|u|^2u=0,\ \ t\in\mathbb{R},\ \ x\in\mathbb{R}^2.
\end{equation}
Given any integer $ K\geq 2 $, the authors constructed a global (for $ t > 0 $) $ K $-solitary wave solution $ u(t) $ of \eqref{NLS} that decomposes asymptotically into a sum of
solitary waves centered at the vertices of a $ K $-sided regular polygon and concentrating at a logarithmic rate as $ t\to\infty$ so that the solution blows up in infinite time with the rate $ ||\nabla u||_{L^2}\sim |\ln t| $ as $ t\to \infty. $ Moreover, such solution concentrates $ K $ bubbles at a point $ x_0\in\mathbb{R}^2 $. In contrast to \cite{MR}, Theorem \ref{thm2} constructed clustered helical rotational-invariant solutions to 3D incompressible Euler equations that decomposes asymptotically into a sum of
bubbles collapsing to a point $ x_0  $ as $ \varepsilon\to0 $, rather than $ t\to \infty. $ It is interesting whether one can construct clustered helical solutions to 3D incompressible Euler equations which blow up in infinite time and finite time. To our knowledge, it is also unknown.
\end{remark}

The paper is organized as follows.  To construct clustered solutions to \eqref{111}, we first give the $ C^1 $-asymptotic expansion of Green's function $ G_K $ in section 2. We also choose the admissible class $ \Lambda_{\varepsilon, m} $ and approximate solutions properly and give some basic estimates for approximate solutions. In section 3, by using non-degeneracy of solutions to limiting equations \eqref{limit eq} we get coercive estimates of the linearized operator $ Q_\delta L_\delta $. The existence of the error term $ \omega_{\delta,Z} $ and  the differentiability of $ \omega_{\delta,Z}  $ with respect to $ Z $ are proved in section  4.
In sections 5 and 6, we calculate the order of the energy  $ K_\delta(Z) $ and show the existence of maximizers of $ K_\delta $ in $ \Lambda_{\varepsilon, m} $, which complete the proof of Theorem \ref{thm1}. The proof of  Theorem \ref{thm2} is given in section 7.

\section{Approximate solutions}

The purpose of this section is to give $ C^1 $ estimates of Green's function $ G_K $ and construct approximate solutions to \eqref{111}. %the following elliptic equations in divergence form
%\begin{equation}\label{111}
%\begin{cases}
%-\delta^2\text{div}(K(x)\nabla w)=  (w-q)^{p}_+,\ \ &\text{in}\  \Omega,\\
%w=0,\ \ &\text{on}\ \partial \Omega.
%\end{cases}
%\end{equation}
%and give some basic estimates for the approximate  solutions. %Note that $ -\text{div}(K(x)\nabla \cdot) $ is a uniformly elliptic operator.
%Throughout this paper, we denote $ C,C_1,C_2\cdots $ positive constants independent of $ \varepsilon $, whose values may change from line to line.

The expansion of Green's function $ G_K $ of the operator $ -\text{div}(K(x)\nabla \cdot) $ with 0-Dirichlet condition  plays an essential role in our analysis.
%We first give some properties of the Green's function $ G_K $ of the operator $ -\text{div}(K(x)\nabla \cdot) $ with Dirichlet condition.
Let $ G_K(x,y) $ be the Green's function of $ -\text{div}(K(x)\nabla \cdot) $ with 0-Dirichlet condition in $ \Omega $, that is, solutions of the following linear elliptic problem:
\begin{equation}\label{diver form}
\begin{cases}
-\text{div}(K(x)\nabla u)= f,\ \ & \Omega,\\
u=0,\ \ & \partial \Omega
\end{cases}
\end{equation}
can be expressed by $ u(x)=\int_{\Omega}G_K(x,y)f(y)dy $ for $ x\in \Omega. $
%Since $ K $ is a $ C^\infty $ positive definite matrix with all eigenvalues having uniformly positive lower and upper bounds, by the Cholesky decomposition \textcolor{red}{(see \cite{G})} one can find the unique $ T\in C^\infty(\overline{\Omega}) $ such that for any $ x\in \Omega $, $ T(x) $ is invertible and
%Let $ T $ be the unique positive-definite matrix-valued function satisfying
%\begin{equation}\label{T_z choice}
%(T(x)^{-1})(T(x)^{-1})^t=K(x),\ \ \ \ x\in \Omega.
%\end{equation}
%For simplicity, we denote $ T_{x}=T(x) $. Define $ \Gamma(x)=-\frac{1}{2\pi}\ln|x|$ the fundamental solutions of $ -\Delta $ in $ \mathbb{R}^2 $.

%From lemma 3.1 in \cite{CW2}  (see also Lemma \ref{A1} in Appendix), we know that there exists a function $ S_K\in C^{0,\gamma}_{loc} (\Omega\times \Omega) $ for any $ \gamma\in(0,1) $, such that for every $f \in L^{p}(\Omega)$  with $ p>2 $ and every $ x\in \Omega $, the solution of \eqref{diver form} can be expressed by
%\begin{equation*}
%\begin{split}
%u(x)=\int_{\Omega}\left (\frac{\sqrt{\det K(x)}^{-1}+\sqrt{\det K(y)}^{-1}}{2}\Gamma\left (\frac{T_x+T_y}{2}(x-y)\right )+S_K(x,y)\right)f(y)dy.
%\end{split}
%\end{equation*}
%Moreover, $ S_K(x,y)=S_K(y,x) $  for any  $ x,y \in \Omega. $ In other words, the Green's function $ G_K $ of the operator $ -\text{div}(K(x)\nabla  ) $ with Dirichlet condition has the decomposition
%\begin{equation}\label{2-11}
%G_K(x,y)=\frac{\sqrt{\det K(x)}^{-1}+\sqrt{\det K(y)}^{-1}}{2}\Gamma\left (\frac{T_x+T_y}{2}(x-y)\right )+S_K(x,y).
%\end{equation}
In \cite{CW} we have obtained $ C^0$-asymptotic expansion of $ G_K(x,y)$. We need to use $ C^1 $-asymptotic expansion of $ G_K(x,y) $ in the present paper.
\begin{lemma}[lemma 3.1, \cite{CW}]\label{Green expansion}
For $ y\in\Omega $, let $ T_y=T(y) $ be the unique positive-definite matrix satisfying
$$ (T_y)^{-1}(T_y)^{-t}=K(y). $$ 
Then there exists a function $\bar{S}_K\in C^{0,\gamma}_{loc} (\Omega\times \Omega) $ for any $ \gamma\in(0,1) $, such that
\begin{equation*}
G_K(x,y)=\sqrt{\det K(y)}^{-1}\Gamma\left (T_y(x-y)\right )+\bar{S}_K(x,y),\ \ \forall\ x,y\in\Omega.
\end{equation*}
\end{lemma}
Now for $ i,j=1,2 $, we denote $ T_{ij}=(T_y)_{ij}=(T(y))_{ij} $ the component of row $ i $, column $ j $ of the matrix $ T_y $. The following lemma gives  the $ C^1 $-asymptotic expansion of $ G_K(x,y) $.%, which improves results in Lemma \ref{Green expansion}.
\begin{lemma}\label{Green expansion2}
%There exists a function $\bar{S}_K\in C^{0,\gamma}_{loc} (\Omega\times \Omega) $ for any $ \gamma\in(0,1) $, such that
 Then there holds
\begin{equation*}
\bar{S}_K(x,y)=-F_{1,y}(x)-F_{2,y}(x)+\bar{H}_1(x,y) \ \ \ \ \forall\ x,y\in\Omega,
\end{equation*}
where
\begin{equation}\label{exp of F_{1,y}}
\begin{split}
F_{1,y}(x)=-\frac{1}{4\pi}\sqrt{\det K(y)}^{-1}\sum_{i,j,m=1}^2T_{mj}\partial_{x_i}K_{ij}(y)\left( T_y(x-y)\right)_m\ln|T_y(x-y)|,
\end{split}
\end{equation}
\begin{equation*}
\begin{split}
F_{2,y}(x)=&\frac{1}{\pi}\sqrt{\det K(y)}^{-1}\sum_{i,j,\alpha=1}^2\partial_{x_\alpha}K_{ij}(y)\cdot\\
&\bigg\{T^{-1}_{\alpha1}T_{1j}T_{1i}\left( -\frac{1}{8}\frac{\left( T_y\left (x-y\right )\right) _1^3}{|T_y\left (x-y\right )|^2}+\frac{1}{8}\left( T_y\left (x-y\right )\right)_1\ln|T_y\left (x-y\right )|\right) \\
+&T^{-1}_{\alpha1}T_{1j}T_{2i}\left( -\frac{1}{8}\frac{\left( T_y\left (x-y\right )\right) _1^2\left( T_y\left (x-y\right )\right) _2}{|T_y\left (x-y\right )|^2}+\frac{1}{8}\left( T_y\left (x-y\right )\right)_2\ln|T_y\left (x-y\right )|\right) \\
+&T^{-1}_{\alpha1}T_{2j}T_{1i}\left( -\frac{1}{8}\frac{\left( T_y\left (x-y\right )\right) _1^2\left( T_y\left (x-y\right )\right) _2}{|T_y\left (x-y\right )|^2}+\frac{1}{8}\left( T_y\left (x-y\right )\right)_2\ln|T_y\left (x-y\right )|\right)\\
+&T^{-1}_{\alpha1}T_{2j}T_{2i}\left( -\frac{1}{8}\frac{\left( T_y\left (x-y\right )\right) _1\left( T_y\left (x-y\right )\right) _2^2}{|T_y\left (x-y\right )|^2}-\frac{1}{8}\left( T_y\left (x-y\right )\right)_1\ln|T_y\left (x-y\right )|\right)
\end{split}
\end{equation*}
\begin{equation}\label{exp of F_{2,y}}
\begin{split}
+&T^{-1}_{\alpha2}T_{1j}T_{1i}\left( -\frac{1}{8}\frac{\left( T_y\left (x-y\right )\right) _1^2\left( T_y\left (x-y\right )\right) _2}{|T_y\left (x-y\right )|^2}-\frac{1}{8}\left( T_y\left (x-y\right )\right)_2\ln|T_y\left (x-y\right )|\right)\\
+&T^{-1}_{\alpha2}T_{1j}T_{2i}\left( -\frac{1}{8}\frac{\left( T_y\left (x-y\right )\right) _1\left( T_y\left (x-y\right )\right) _2^2}{|T_y\left (x-y\right )|^2}+\frac{1}{8}\left( T_y\left (x-y\right )\right)_1\ln|T_y\left (x-y\right )|\right)\\
+&T^{-1}_{\alpha2}T_{2j}T_{1i}\left( -\frac{1}{8}\frac{\left( T_y\left (x-y\right )\right) _1\left( T_y\left (x-y\right )\right) _2^2}{|T_y\left (x-y\right )|^2}+\frac{1}{8}\left( T_y\left (x-y\right )\right)_1\ln|T_y\left (x-y\right )|\right)\\
+&T^{-1}_{\alpha2}T_{2j}T_{2i}\left( -\frac{1}{8}\frac{\left( T_y\left (x-y\right )\right) _2^3 }{|T_y\left (x-y\right )|^2}+\frac{1}{8}\left( T_y\left (x-y\right )\right)_2\ln|T_y\left (x-y\right )|\right)\bigg\},
\end{split}
\end{equation}
 and $ x\to \bar{H}_1(x,y)\in C^{1,\gamma}(\overline{\Omega})$ for all $ y\in \Omega $, $ \gamma\in (0, 1) $.  Moreover, the function $(x, y)\to \bar{H}_1(x,y)\in C^1(\Omega\times\Omega),$ and in particular the corresponding Robin function $x\to \bar{S}_K(x,x)\in C^1(\Omega)$.
\end{lemma}
\begin{remark}
It follows from Lemma \ref{Green expansion2}  that Green's function $ G_K $ has an expansion
\begin{equation}\label{exp of G_K}
G_K(x,y)=\sqrt{\det K(y)}^{-1}\Gamma\left (T_y(x-y)\right )-F_{1,y}(x)-F_{2,y}(x)+\bar{H}_1(x,y),\ \ \forall\ x,y\in\Omega.
\end{equation}

Before proving Lemma \ref{Green expansion2} let us now give some examples to explain results in Lemma \ref{Green expansion2}.\\

\noindent$ Example $ 1. If $ K(x)= Id $, then \eqref{diver form} is the standard Laplacian problem. In this case, one computes directly that $ F_{1,y}=F_{2,y}\equiv 0 $. From \eqref{exp of G_K},  Green's function has an expansion
\begin{equation*}
G_1(x,y)=\Gamma(x-y)+S_1(x,y),\ \ \ \ \forall\ x,y\in \Omega.
\end{equation*}
Thus we have $ S_1(x,y)=-H(x,y) $, where $ H(x,y) $ is the regular part of Green's function of $ -\Delta $ in $ \Omega $ with zero-Dirichlet data, which coincides with classical results in \cite{GT}.\\

\noindent$ Example $ 2. If  $ K(x)=\frac{1}{b(x)}Id $, where $ b\in C^2(\overline{\Omega}) $ and $ \inf_{\Omega}b>0 $, then $ \det K=\frac{1}{b^2} $ and $ T=\sqrt{b}Id $. By \eqref{exp of F_{1,y}} and \eqref{exp of F_{2,y}} it is not hard to get that $ F_{1,y}(x)= \frac{\nabla b(y)\cdot (x-y)}{4\pi}\ln|x-y|+F^*_1(x,y) $ and $ F_{2,y}(x)=F^*_2(x,y) $ for $ x,y\in\Omega $, where $ F^*_1,F^*_2\in C^{1}(\Omega\times\Omega) $. From  \eqref{exp of G_K},  Green's function has an expansion
\begin{equation*}
G_b(x,y)=b(y)\Gamma\left (x-y\right )-\frac{\nabla b(y)\cdot (x-y)}{4\pi}\ln|x-y|+S_b(x,y) \ \ \ \ \forall\ x,y\in \Omega,
\end{equation*}
where $ S_b\in C^1(\Omega\times\Omega) $, which coincides with results in \cite{De2,WYZ}.
\end{remark}

We now turn to the proof of Lemma \ref{Green expansion2}.
\begin{proof}
Let  $ y\in\Omega $ be fixed. In the following,  we always denote $ T_{ij}=(T_y)_{ij} $ the component of row $ i $, column $ j $ of the matrix $ T_y $ for $ i,j=1,2 $. From Lemma \ref{Green expansion}, the regular part $ \bar{S}_K(x,y) $ satisfies
\begin{equation}\label{equa of S_K}
\begin{cases}
-\text{div}\left (K(x)\nabla \bar{S}_K(x,y)\right )=\text{div}\left (\left (K(x)-K(y)\right )\nabla \left( \sqrt{\det K(y)}^{-1}\Gamma\left (T_y(x-y)\right )\right)  \right )\ &\text{in}\ \Omega,\\
\bar{S}_K(x,y)=-\sqrt{\det K(y)}^{-1}\Gamma\left (T_y(x-y)\right )\ &\text{on}\ \partial\Omega.
\end{cases}
\end{equation}
This implies that
\begin{equation}\label{3-1}
\begin{split}
-\text{div}\left (K(x)\nabla \bar{S}_K(x,y)\right )=&\sum_{i,j=1}^2\partial_{x_i}K_{ij}(x)\partial_{x_j}\left(\sqrt{\det K(y)}^{-1}\Gamma\left (T_y(x-y)\right ) \right) \\
&+\sum_{i,j=1}^2\left( \left( K_{ij}(x)-K_{ij}(y)\right) \partial_{x_ix_j}\left(\sqrt{\det K(y)}^{-1}\Gamma\left (T_y(x-y)\right ) \right)\right)\\
=&:A_1+A_2.
\end{split}
\end{equation}

As for $ A_1 $, for $ x\in\mathbb{R}^2 $, we denote $ J(x)=-\frac{1}{8\pi}|x|^2\ln|x| $. Then $ \Delta \left( J(x-y)\right)=\Gamma(x-y)-\frac{1}{2\pi}.  $
Using transformation of coordinates, one computes directly that
\begin{equation*}
\text{div}\left(K(y)\nabla J(T_y(x-y)) \right)= \Gamma(T_y(x-y))-\frac{1}{2\pi},
\end{equation*}
from which we deduce,
\begin{equation*}
\text{div}\left(K(y)\nabla   \partial_{x_j}\left( J(T_y(x-y))\right)  \right)= \partial_{x_j} \left( \Gamma(T_y(x-y))\right) .
\end{equation*}
We define for $ x\in\Omega $
\begin{equation}\label{F_{1,y}}
\begin{split}
F_{1,y}(x)=\sum_{i,j=1}^2\partial_{x_i}K_{ij}(y)\cdot\sqrt{\det K(y)}^{-1} \partial_{x_j}\left( J\left (T_y(x-y)\right )\right) ,
\end{split}
\end{equation}
then one has
\begin{equation}\label{3-2}
\text{div}\left(K(y)\nabla F_{1,y}(x)  \right)= \sum_{i,j=1}^2\partial_{x_i}K_{ij}(y)\partial_{x_j}\left(\sqrt{\det K(y)}^{-1}\Gamma\left (T_y(x-y)\right ) \right).
\end{equation}

As for $ A_2 $, using Taylor's expansion we obtain
\begin{equation}\label{3-3}
\begin{split}
\sum_{i,j=1}^2\left( K_{ij}(x)-K_{ij}(y)\right) \partial_{x_ix_j}&\left(\sqrt{\det K(y)}^{-1}\Gamma\left (T_y(x-y)\right ) \right)\\
=\sum_{\alpha,i,j=1}^2\sqrt{\det K(y)}^{-1}&\partial_{x_\alpha}K_{ij}(y) (x-y)_\alpha\cdot\partial_{x_ix_j}\left( \Gamma\left (T_y(x-y)\right ) \right)+\phi_y(x),
\end{split}
\end{equation}
where $ \phi_y(\cdot)\in L^p(\Omega) $ for all $ p>1. $
Since
\begin{equation*}
\partial_{x_ix_j}\Gamma(x)=-\frac{1}{2\pi}\left(\frac{\delta_{i,j}}{|x|^2}-\frac{2x_ix_j}{|x|^4} \right),
\end{equation*}
where $ \delta_{i,j}=1 $ for $ i=j $ and $ \delta_{i,j}=0 $ for $ i\neq j $, we have
\begin{equation}\label{3-6}
\begin{split}
\partial_{x_ix_j}\left( \Gamma(T_y(x-y))\right) =-\frac{1}{2\pi}\sum_{m,n=1}^2T_{mj}T_{ni}\left(\frac{\delta_{m,n}}{|T_y(x-y)|^2}-\frac{2\left( T_y(x-y)\right)_m\left( T_y(x-y)\right)_n}{|T_y(x-y)|^4} \right).
\end{split}
\end{equation}
%Here $ T_{mj}=T_{mj}(y) $. Define $ z=T_y(x-y) $.
Taking \eqref{3-6} into \eqref{3-3}, we get
\begin{equation}\label{3-7}
\begin{split}
\sum_{i,j=1}^2\left( K_{ij}(x)-K_{ij}(y)\right) \partial_{x_ix_j}&\left(\sqrt{\det K(y)}^{-1}\Gamma\left (T_y(x-y)\right ) \right)\\
=\sum_{\alpha,\beta,i,j,m,n=1}^2\sqrt{\det K(y)}^{-1}&\partial_{x_\alpha}K_{ij}(y)T^{-1}_{\alpha\beta} \left (T_y(x-y)\right )_\beta\cdot  \\ -\frac{1}{2\pi}T_{mj}T_{ni}&\left(\frac{\delta_{m,n}}{|T_y(x-y)|^2}-\frac{2\left( T_y(x-y)\right)_m\left( T_y(x-y)\right)_n}{|T_y(x-y)|^4} \right)  +\phi_y(x).
\end{split}
\end{equation}

Note that
\begin{equation}\label{3-4}
\frac{x^p}{|x|^4}=-\frac{1}{8}\Delta\left( \frac{x^p}{|x|^2}\right)+\frac{1}{8}\frac{\Delta x^p}{|x|^2}\ \ \ \ \text{for}\ |p|=3,
\end{equation}
where $p=(p_1,p_2)$ is the multi-index and $ x^p=x_1^{p_1}x_2^{p_2} $. From \eqref{3-4}, it is not hard to check that for $ 1\leq m\neq n\leq 2 $
\begin{equation*}
\begin{cases}
\frac{x_m}{|x|^2}&=\Delta\left( \frac{1}{2}x_m\ln|x|\right),\\
\frac{x_m^2x_n}{|x|^4}&=\Delta\left( -\frac{1}{8}\frac{x_m^2x_n}{|x|^2}+\frac{1}{8}x_n\ln|x|\right),\\
\frac{x_m^3}{|x|^4}&=\Delta\left( -\frac{1}{8}\frac{x_m^3}{|x|^2}+\frac{3}{8}x_m\ln|x|\right),
\end{cases}
\end{equation*}
which implies that
\begin{equation}\label{3-5}
\begin{cases}
\frac{\left (T_y\left (x-y\right )\right )_m}{|T_y\left (x-y\right )|^2}&=\text{div}\left(K(y) \nabla\left( \frac{1}{2}\left (T_y\left (x-y\right )\right )_m\ln|T_y\left (x-y\right )|\right)\right) ,\\
\frac{\left (T_y\left (x-y\right )\right )_m^2\left (T_y\left (x-y\right )\right )_n}{|T_y\left (x-y\right )|^4}&=\text{div}\left(K(y) \nabla \left( -\frac{1}{8}\frac{\left (T_y\left (x-y\right )\right )_m^2\left (T_y\left (x-y\right )\right )_n}{|T_y\left (x-y\right )|^2}+\frac{1}{8}\left (T_y\left (x-y\right )\right )_n\ln|T_y\left (x-y\right )|\right)\right) ,\\
\frac{\left (T_y\left (x-y\right )\right )_m^3}{|T_y\left (x-y\right )|^4}&=\text{div}\left(K(y) \nabla\left( -\frac{1}{8}\frac{\left (T_y\left (x-y\right )\right )_m^3}{|T_y\left (x-y\right )|^2}+\frac{3}{8}\left (T_y\left (x-y\right )\right )_m\ln|T_y\left (x-y\right )|\right)\right).
\end{cases}
\end{equation}
We define for $ x\in \Omega $
\begin{equation*}
\begin{split}
F_{2,y}(x)=-&\frac{1}{2\pi}\sqrt{\det K(y)}^{-1}\sum_{\alpha,\beta,i,j,m,n=1}^2 \partial_{x_\alpha}K_{ij}(y)T^{-1}_{\alpha\beta} T_{mj}T_{ni}\cdot\frac{1}{2}\left (T_y\left (x-y\right )\right )_\beta\ln|T_y\left (x-y\right )|\delta_{m,n}\\
+&\frac{1}{\pi}\sqrt{\det K(y)}^{-1} \sum_{\alpha,i,j=1}^2\partial_{x_\alpha}K_{ij}(y)\cdot \\
&\bigg[T^{-1}_{\alpha1}T_{1j}T_{1i}\left( -\frac{1}{8}\frac{\left( T_y\left (x-y\right )\right) _1^3}{|T_y\left (x-y\right )|^2}+\frac{3}{8}\left( T_y\left (x-y\right )\right)_1\ln|T_y\left (x-y\right )|\right) \\
+&T^{-1}_{\alpha1}T_{1j}T_{2i}\left( -\frac{1}{8}\frac{\left( T_y\left (x-y\right )\right) _1^2\left( T_y\left (x-y\right )\right) _2}{|T_y\left (x-y\right )|^2}+\frac{1}{8}\left( T_y\left (x-y\right )\right)_2\ln|T_y\left (x-y\right )|\right) \\
+&T^{-1}_{\alpha1}T_{2j}T_{1i}\left( -\frac{1}{8}\frac{\left( T_y\left (x-y\right )\right) _1^2\left( T_y\left (x-y\right )\right) _2}{|T_y\left (x-y\right )|^2}+\frac{1}{8}\left( T_y\left (x-y\right )\right)_2\ln|T_y\left (x-y\right )|\right)\\
+&T^{-1}_{\alpha1}T_{2j}T_{2i}\left( -\frac{1}{8}\frac{\left( T_y\left (x-y\right )\right) _1\left( T_y\left (x-y\right )\right) _2^2}{|T_y\left (x-y\right )|^2}+\frac{1}{8}\left( T_y\left (x-y\right )\right)_1\ln|T_y\left (x-y\right )|\right)
\end{split}
\end{equation*}
\begin{equation}\label{F_{2,y}}
\begin{split}
+&T^{-1}_{\alpha2}T_{1j}T_{1i}\left( -\frac{1}{8}\frac{\left( T_y\left (x-y\right )\right) _1^2\left( T_y\left (x-y\right )\right) _2}{|T_y\left (x-y\right )|^2}+\frac{1}{8}\left( T_y\left (x-y\right )\right)_2\ln|T_y\left (x-y\right )|\right)\\
+&T^{-1}_{\alpha2}T_{1j}T_{2i}\left( -\frac{1}{8}\frac{\left( T_y\left (x-y\right )\right) _1\left( T_y\left (x-y\right )\right) _2^2}{|T_y\left (x-y\right )|^2}+\frac{1}{8}\left( T_y\left (x-y\right )\right)_1\ln|T_y\left (x-y\right )|\right)\\
+&T^{-1}_{\alpha2}T_{2j}T_{1i}\left( -\frac{1}{8}\frac{\left( T_y\left (x-y\right )\right) _1\left( T_y\left (x-y\right )\right) _2^2}{|T_y\left (x-y\right )|^2}+\frac{1}{8}\left( T_y\left (x-y\right )\right)_1\ln|T_y\left (x-y\right )|\right)\\
+&T^{-1}_{\alpha2}T_{2j}T_{2i}\left( -\frac{1}{8}\frac{\left( T_y\left (x-y\right )\right) _2^3 }{|T_y\left (x-y\right )|^2}+\frac{3}{8}\left( T_y\left (x-y\right )\right)_2\ln|T_y\left (x-y\right )|\right)\bigg].
\end{split}
\end{equation}
Combining \eqref{F_{2,y}} with \eqref{3-7} and \eqref{3-5}, we get
\begin{equation}\label{3-8}
\begin{split}
\text{div}\left(K(y) \nabla F_{2,y}(x)  \right)=&\sum_{\alpha,\beta,i,j,m,n=1}^2\sqrt{\det K(y)}^{-1}\partial_{x_\alpha}K_{ij}(y)T^{-1}_{\alpha\beta} \left (T_y(x-y)\right )_\beta\cdot \\
&-\frac{1}{2\pi}T_{mj}T_{ni}\left(\frac{\delta_{m,n}}{|T_y(x-y)|^2}-\frac{2\left( T_y(x-y)\right)_m\left( T_y(x-y)\right)_n}{|T_y(x-y)|^4} \right)\\
=&
\sum_{i,j=1}^2\left( K_{ij}(x)-K_{ij}(y)\right) \partial_{x_ix_j}\left(\sqrt{\det K(y)}^{-1}\Gamma\left (T_y(x-y)\right ) \right)-\phi_y(x).
\end{split}
\end{equation}

Now we define $ \bar{H}_{1,y}(x)=\bar{S}_K(x,y)+F_{1,y}(x)+F_{2,y}(x) $. Taking \eqref{3-2} and \eqref{3-8} into \eqref{3-1}, we obtain
\begin{equation}\label{3-9}
\begin{split}
-\text{div}&\left (K(x)\nabla \bar{H}_{1,y}(x)\right )\\
=&-\text{div}\left (\left( K(x)-K(y)\right) \nabla\left( F_{1,y}(x)+F_{2,y}(x)\right) \right) \\
&+\sum_{i,j=1}^2\left( \partial_{x_i}K_{ij}(x)-\partial_{x_i}K_{ij}(y)\right) \partial_{x_j}\left(\sqrt{\det K(y)}^{-1}\Gamma\left (T_y(x-y)\right ) \right)+\phi_y(x).
\end{split}
\end{equation}
We can verify that for all $ y\in\Omega $, the right-hand side of \eqref{3-9} belongs to $ L^p(\Omega) $ for all $ p>1. $ Note also that
\begin{equation*}
\bar{H}_{1,y}(x)=-\sqrt{\det K(y)}^{-1}\Gamma\left (T_y(x-y)\right )+F_{1,y}(x)+F_{2,y}(x)\ \ \ \ x\in\partial\Omega.
\end{equation*}
For $ x,y\in\Omega $, we define $ \bar{H}_1(x,y)=\bar{H}_{1,y}(x) $. Applying the elliptic theory, we obtain that $ x\to \bar{H}_1(x,y) $ is in $ C^{1,\gamma}(\overline{\Omega}) $, for all $ \gamma\in(0,1) $. Furthermore,
by the continuity of the right-hand side of \eqref{3-9} and the boundary condition with respect to $ y $ in $ L^p(\Omega) $   and $ C^2(\partial\Omega) $, respectively, we can get $ \bar{H}_1(x,y)=\bar{H}_{1,y}(x)\in C(\Omega, C^{1,\gamma}(\overline{\Omega})) $ and thus $ \nabla_x\bar{H}_1(x,y)\in C(\Omega\times \Omega) $.

Similarly,  taking $ \nabla_y $ to both sides of \eqref{3-9}, we can check that $ \nabla_y\bar{H}_{1,y}(x)\in C(\Omega, C^{0,\gamma}(\overline{\Omega})) $, which implies that $ \nabla_y\bar{H}_1(x,y)\in C(\Omega\times \Omega) $, then $ \bar{H}_1 $ is a $ C^1 $ function over $ \Omega\times \Omega $. From \eqref{F_{1,y}} and \eqref{F_{2,y}}, we can prove that \eqref{exp of F_{1,y}} and \eqref{exp of F_{2,y}} hold. Finally, $ \bar{S}_K(x, x) =  \bar{H}_1(x,x) $ is clearly in $ C^1(\Omega) $.

\end{proof}

The argument that the function $x\to \bar{S}_K(x,x)\in C^1(\Omega)$ will play an important role for us to get the $ C^1-$dependance
of clustered solutions for the finite-dimensional variational reduction, see sections 4 and 5 below.

Let $ m> 1 $ be an integer. Now we give approximate solutions of \eqref{111} and define the admissible class $ \Lambda_{\varepsilon, m} $ for the parameter $ Z=(z_1,\cdots,z_m) $. For any $ \hat{x}\in \Omega, \hat{q}>0 $, we define
\begin{equation}\label{eq2-2}
V_{\delta, \hat{x}, \hat{q}}(x)=\begin{cases}
\hat{q}+\delta^{\frac{2}{p-1}}s_\delta^{-\frac{2}{p-1}}\phi\left(\frac{|T_{\hat{x}}(x-\hat{x})|}{s_\delta}\right),\ \ &|T_{\hat{x}}(x-\hat{x})|\le s_\delta,\\
\hat{q}\ln |T_{\hat{x}}(x-\hat{x})|/\ln s_\delta,\ \  &  |T_{\hat{x}}(x-\hat{x})|>s_\delta,
\end{cases}
\end{equation}
where $ \phi\in H^1_0(B_1(0)) $ satisfies (see, e.g., \cite{CLW})
\begin{equation*}
-\Delta\phi=\phi^p, \ \ \phi>0\ \ \text{in}\ B_1(0),
\end{equation*}
and $ s_\delta $ satisfies
\begin{equation}\label{201}
\delta^{\frac{2}{p-1}}s_\delta^{-\frac{2}{p-1}}\phi'(1)=\hat{q}/\ln s_\delta.
\end{equation}
Clearly, $ V_{\delta, \hat{x}, \hat{q}}\in C^1 $ is a solution of
\begin{equation}\label{eq2}
\begin{cases}
-\delta^2\text{div}(K(\hat{x})\nabla v)=  (v-\hat{q})^{p}_+,\ \ & \text{in}\  \mathbb{R}^2,\\
v=\hat{q},\ \ &\text{on} \ \{x\mid |T_{\hat{x}}(x-\hat{x})|=s_\delta\},
\end{cases}
\end{equation}
and  for $ \delta $ sufficiently small, \eqref{201} is uniquely solvable with
\begin{equation*}
\frac{s_\delta}{\delta|\ln\delta|^{\frac{p-1}{2}}}\to \left( \frac{|\phi'(1)|}{\hat{q}}\right) ^{\frac{p-1}{2}}\ \ \ \ \text{as}\ \delta\to0.
\end{equation*}
The Pohozaev identity implies
\begin{equation}\label{PI}
\int_{B_1(0)}\phi^{p+1}=\frac{\pi(p+1)}{2}|\phi'(1)|^2,\ \ \int_{B_1(0)}\phi^{p}= 2\pi|\phi'(1)|.
\end{equation}

Since $ V_{\delta, \hat{x}, \hat{q}} $ is not 0 on $ \partial \Omega $, we need to make a projection on $ H^1_0(\Omega). $ Note that the operator $ \text{div}(K(\hat{x})\nabla \cdot) $ in \eqref{eq2} is different from $ \text{div}(K(x)\nabla \cdot) $ appeared in \eqref{111}, we introduce a projection term $ H_{\delta, \hat{x}, \hat{q}} $ defined by
\begin{equation}\label{eq3}
\begin{cases}
-\text{div}(K(x)\nabla H_{\delta, \hat{x}, \hat{q}})=\text{div}((K(x)-K(\hat{x}))\nabla V_{\delta, \hat{x}, \hat{q}}),\ \ &\Omega,\\
H_{\delta, \hat{x}, \hat{q}}=-V_{\delta, \hat{x}, \hat{q}},\ \ &\partial\Omega.
\end{cases}
\end{equation}
Then  $ H_{\delta, \hat{x}, \hat{q}}\in W^{2,p}(\Omega)\subset  C^{1,\alpha}(\overline{\Omega}) $ for any $ p>1, \alpha\in(0,1) $. From lemma 3.2 in \cite{CW}, we have the following estimate of the difference between $ H_{\delta, \hat{x}, \hat{q}} $ and  $ -\frac{2\pi\hat{q}\sqrt{\det K(\hat{x})}}{\ln s_\delta}\bar{S}_K(\cdot,\hat{x}) $.

\begin{lemma}[lemma 3.2, \cite{CW}]\label{H estimate}
Define $ \zeta_{\delta, \hat{x}, \hat{q}}(x)=H_{\delta, \hat{x}, \hat{q}}(x)+\frac{2\pi\hat{q}\sqrt{\det K(\hat{x})}}{\ln s_\delta}\bar{S}_K(x,\hat{x}) $ for $ x\in\Omega $. Then for any $ p\in(1,2) $, there exists a constant $ C>0 $ independent of $ \delta $ such that
\begin{equation*}
||\zeta_{\delta, \hat{x}, \hat{q}}||_{C^{0,2-\frac{2}{p}}(\Omega)}\leq C\frac{s_\delta^{\frac{2}{p}-1}}{|\ln s_\delta|}.
\end{equation*}
\end{lemma}

Using \eqref{eq2-2}, the definition of $ H_{\delta, \hat{x}, \hat{q}} $ in \eqref{eq3} and the classical $ L^p $-theory of elliptic equations, one computes directly that
\begin{equation}\label{3-001}
||H_{\delta, \hat{x}, \hat{q}}||_{W^{2,p}(\Omega)}\leq
\begin{cases}
\frac{C}{\varepsilon^{1-\frac{2}{p}}|\ln\varepsilon|},\ \ &p>2,\\
C,\ \ &p=2,\\
\frac{C}{|\ln\varepsilon|},\ \ &1\leq p<2.
\end{cases}
\end{equation}
%\begin{equation}\label{3-002}
%||\frac{\partial H_{\delta, \hat{x}, \hat{q}}}{\partial \hat{x}_h}||_{W^{1,p}(\Omega)}\leq
%\begin{cases}
%\frac{C}{\varepsilon^{1-\frac{2}{p}}|\ln\varepsilon|},\ \ &p>2,\\
%C,\ \ &p=2,\\
%\frac{C}{|\ln\varepsilon|},\ \ &1\leq p<2.
%\end{cases}
%\end{equation}

Let $ Z=(z_1,\cdots,z_m)\in\Omega^{(m)} $. Since $ x_{0} $ is a strict local maximum point  of $ q^2\sqrt{\det(K)} $ in $ \Omega $,   we define the admissible set for $ Z $ as follows:%. . Hence we can choose $ \bar{\rho}>0 $ sufficiently small such that $ B_{\bar{\rho}}(x_{0})\Subset\Omega   $ and
%\begin{equation*}
%\max_{\partial B_{\bar{\rho}}(x_{0})}q^2\sqrt{\det(K)}< \max_{  B_{\bar{\rho}}(x_{0})}q^2\sqrt{\det(K)}=q^2\sqrt{\det(K)}(x_0).
%\end{equation*}
%The admissible class  we choose is as follows
\begin{equation}\label{admis set}
\Lambda_{\varepsilon, m}=\{Z=(z_1,\cdots, z_m)\in\Omega^{(m)}\mid z_i\in B_{\bar{\rho}}(x_{0}), \min_{i\neq j}|z_i-z_j|\geq |\ln\varepsilon|^{-M}, \  \forall 1\leq i\neq j\leq m\},
\end{equation}
where $ M=m^2+1. $ Clearly by \eqref{admis set},
\begin{equation}\label{3-003}
G_K(z_i,z_j)\leq C|\ln\frac{1}{|z_i-z_j|}|\leq C\ln|\ln\varepsilon|,\ \ \ \ Z\in \Lambda_{\varepsilon, m}.
\end{equation}

In the following, we will construct solutions of \eqref{111} being of the form
\begin{equation}\label{solu config}
v_\delta= V_{\delta, Z}+\omega_\delta=\sum_{j=1}^mV_{\delta, Z,j}+\omega_{\delta,Z}=\sum_{j=1}^m(V_{\delta, z_j, \hat{q}_j}+H_{\delta, z_j, \hat{q}_j})+\omega_{\delta,Z},
\end{equation}
where $ Z=(z_1,\cdots,z_m)\in \Lambda_{\varepsilon, m} $, $ \sum_{j=1}^mV_{\delta, Z,j} $ is the main term and $ \omega_{\delta,Z}$ is an error term. The choice of $ \hat{q}_j $ will be made later on. From \eqref{111}, one computes directly that
\begin{equation*}
\begin{split}
0=&\sum_{j=1}^m-\delta^2\text{div}(K(x)\nabla (V_{\delta, z_j, \hat{q}_j}+H_{\delta, z_j, \hat{q}_j}))-\delta^2\text{div}(K(x)\nabla \omega_{\delta,Z})
-\left (\sum_{j=1}^mV_{\delta, Z,j}+\omega_{\delta,Z}-q\right )^{p}_+\\
=&-\sum_{j=1}^m\delta^2\text{div}(K(z_j)\nabla V_{\delta, z_j, \hat{q}_j})-\sum_{j=1}^m\delta^2\text{div}((K(x)-K(z_j))\nabla V_{\delta, z_j, \hat{q}_j})\\
&-\sum_{j=1}^m\delta^2\text{div}(K(x)\nabla H_{\delta, z_j, \hat{q}_j})+\left (-\delta^2\text{div}(K(x)\nabla \omega_{\delta,Z})-p\left (\sum_{j=1}^mV_{\delta, Z,j}-q\right )^{p-1}_+\omega_{\delta,Z} \right )\\
&-\left (\left (\sum_{j=1}^mV_{\delta, Z,j}+\omega_{\delta,Z}-q\right )^{p}_+-p\left (\sum_{j=1}^mV_{\delta, Z,j}-q\right )^{p-1}_+\omega_{\delta,Z}\right )\\
%=&\sum_{j=1}^m(V_{\delta, z_j, \hat{q}_j}-\hat{q}_j)^{p}_+-\left (\sum_{j=1}^mV_{\delta, Z,j}-q\right )^{p}_+\\
%&+\left (-\delta^2\text{div}(K(x)\nabla \omega_\delta)-p\left (\sum_{j=1}^mV_{\delta, Z,j}-q\right )^{p-1}_+\omega_\delta \right )\\
%&-\left (\left (\sum_{j=1}^mV_{\delta, Z,j}+\omega_\delta-q\right )^{p}_+-\left (\sum_{j=1}^mV_{\delta, Z,j}-q\right )^{p}_+-p\left (\sum_{j=1}^mV_{\delta, Z,j}-q\right )^{p-1}_+\omega_\delta\right )\\
=&-l_\delta+L_\delta \omega_{\delta,Z}- R_\delta(\omega_{\delta,Z}),
\end{split}
\end{equation*}
where
\begin{equation*}
l_\delta:=\left (\sum_{j=1}^mV_{\delta, Z,j}-q\right )^{p}_+-\sum_{j=1}^m(V_{\delta, z_j, \hat{q}_j}-\hat{q}_j)^{p}_+,
\end{equation*}
$ L_\delta $ is the linearized operator of \eqref{111} at $ \sum_{j=1}^mV_{\delta, Z,j} $ defined by
\begin{equation*}
L_\delta \omega:=-\delta^2\text{div}(K(x)\nabla \omega)- p\left (\sum_{j=1}^mV_{\delta, Z,j}-q\right )^{p-1}_+\omega,
\end{equation*}
and $ R_\delta(\omega_{\delta,Z}) $ is the high-order error term defined by
\begin{equation*}
R_\delta(\omega_{\delta,Z}):=\left (\sum_{j=1}^mV_{\delta, Z,j}+\omega_{\delta,Z}-q\right )^{p}_+-\left (\sum_{j=1}^mV_{\delta, Z,j}-q\right )^{p}_+-p\left (\sum_{j=1}^mV_{\delta, Z,j}-q\right )^{p-1}_+\omega_{\delta,Z}.
\end{equation*}
Thus it suffices to solve the following equation
\begin{equation}\label{3-03}
L_\delta \omega=l_\delta+ R_\delta(\omega).
\end{equation}

%To make the  norm of $ \omega_\delta $ as small as possible, we choose  $ \hat{q}_j $ suitably close to $ q(z_j) $.
Now we choose parameters $ \hat{q}_j $ suitably to make the error term $ \omega_{\delta,Z}$ as small as possible. For any $ Z\in \Lambda_{\varepsilon, m}, $ let $ \hat{q}_i=\hat{q}_{\delta,i}(Z) $, $ i=1,\cdots,m $ be such that
\begin{equation}\label{q_i choice}
\hat{q}_i=q(z_i)+\frac{2\pi\hat{q}_i\sqrt{\det K(z_i)}}{\ln s_{\delta,i}}\bar{S}_K(z_i, z_i)+\sum_{j\neq i}\frac{2\pi\hat{q}_j\sqrt{\det K(z_j)}}{\ln s_{\delta,j}}G_K(z_i, z_j),
\end{equation}
where $ s_{\delta,i} $ satisfies for $ i=1,\cdots,m $
\begin{equation*}
\delta^{\frac{2}{p-1}}s_{\delta,i}^{-\frac{2}{p-1}}\phi'(1)=\hat{q}_i/\ln s_{\delta,i}.
\end{equation*}

From the Poincar$\acute{\text{e}}$ -- Miranda Theorem (see \cite{Ku}), for any $ \delta $ sufficiently small   there exists $ \hat{q}_{\delta,i}(Z) $ satisfying \eqref{q_i choice}. Moreover, by Lemma \ref{Green expansion2} one computes directly that %for $ i=1,\cdots,m $
%\begin{equation}\label{1999}
%\hat{q}_i=q(z_i)+O\left (\frac{\ln|\ln\varepsilon|}{|\ln\varepsilon|}\right ),
%\end{equation}
%\begin{equation}\label{1999-1}
%\frac{\partial \hat{q}_i}{\partial z_{j,h}}=\frac{\partial q}{\partial x_{h}}(z_i)\delta_{i,j}+O\left (|\ln\varepsilon|^{M}\right ),
%\end{equation}
\begin{equation}\label{2000}
\hat{q}_i=q(z_i)+O\left (\frac{\ln|\ln\varepsilon|}{|\ln\varepsilon|}\right );\ \ \frac{1}{\ln\frac{1}{s_{\delta,i}}}=\frac{1}{\ln\frac{1}{\varepsilon}}+O\left( \frac{\ln|\ln\varepsilon|}{|\ln\varepsilon|^2}\right).
\end{equation}

By the choice of $ \Lambda_{\varepsilon, m} $ in \eqref{admis set} and  $ \hat{q}_{\delta,j} $ in \eqref{q_i choice}, we claim that for any $ Z\in \Lambda_{\varepsilon, m} $, $ \gamma\in(0,1), L>1 $ and $ x\in B_{Ls_{\delta,i}}(z_i) $
\begin{equation}\label{203}
\begin{split}
\sum_{j=1}^mV_{\delta, Z,j}(x)-q(x)=V_{\delta, z_i, \hat{q}_{\delta,i}}(x)-\hat{q}_{\delta,i}+O\left( \frac{\varepsilon^\gamma}{|\ln\varepsilon|}\right).
\end{split}
\end{equation}
Indeed, for $ x\in B_{Ls_{\delta,i}}(z_i) $
\begin{equation*}
\begin{split}
V_{\delta,Z,i}(x)-q(x)
=&V_{\delta, z_i, \hat{q}_{\delta,i}}(x)+H_{\delta, z_i, \hat{q}_{\delta,i}}(x)-q(x)\\
=&V_{\delta, z_i, \hat{q}_{\delta,i}}(x)-q(z_i)-\frac{2\pi\hat{q}_{\delta,i}\sqrt{\det K(z_i)}}{\ln s_{\delta,i}}\bar{S}_K(x, z_i)+O(s_{\delta,i})+O\left( \frac{s_{\delta,i}^{\gamma}}{|\ln s_{\delta,i}|}\right) \\
=&V_{\delta, z_i, \hat{q}_{\delta,i}}(x)-q(z_i)-\frac{2\pi\hat{q}_{\delta,i}\sqrt{\det K(z_i)}}{\ln s_{\delta,i}}\bar{S}_K(z_i, z_i)+ O\left( \frac{\varepsilon^\gamma}{|\ln\varepsilon|}\right).
\end{split}
\end{equation*}
For any $ j\neq i $, since $ |\ln\varepsilon|^{-M}>2Ls_{\delta,i} $ for $ \varepsilon $ sufficiently small, one has
\begin{equation*}
\begin{split}
V_{\delta,Z,j}(x)=&V_{\delta, z_j, \hat{q}_{\delta,j}}(x)+H_{\delta, z_j, \hat{q}_{\delta,j}}(x)\\
=&\frac{\hat{q}_{\delta,j}}{\ln s_{\delta,j}}\ln |T_{z_j}(x-z_j)|-\frac{2\pi\hat{q}_{\delta,j}\sqrt{\det K(z_j)}}{\ln s_{\delta,j}}\bar{S}_K(x, z_j)+O\left(  \frac{s_{\delta,j}^{\gamma}}{|\ln s_{\delta,j}|}\right)\\
=&-\frac{2\pi\hat{q}_{\delta,j}\sqrt{\det K(z_j)}}{\ln s_{\delta,j}}G_K(x, z_j)+O\left(  \frac{s_{\delta,j}^{\gamma}}{|\ln s_{\delta,j}|}\right) \\
=&-\frac{2\pi\hat{q}_{\delta,j}\sqrt{\det K(z_j)}}{\ln s_{\delta,j}}G_K(z_i, z_j)+O\left( \frac{\varepsilon^\gamma}{|\ln\varepsilon|}\right),
\end{split}
\end{equation*}
where we have used  Lemma \ref{H estimate} and the fact that  for $ x\in B_{Ls_{\delta,i}}(z_i) $
\begin{equation*}
G_K(x, z_j)=G_K(z_i, z_j)+O(|\nabla_{z_i}G_K(z_i, z_j)(x-z_j)|)=G_K(z_i, z_j)+O(\varepsilon|\ln\varepsilon|^M).
\end{equation*}
Adding up the above inequalities and using \eqref{q_i choice}, we get \eqref{203}.

Using the definition of $ V_{\delta, z_i, \hat{q}_{\delta,i}} $, we obtain

\begin{equation}\label{200}
\begin{split}
\frac{\partial V_{\delta, z_i, \hat{q}_{\delta,i}}(x)}{\partial x_h}=\begin{cases}
\frac{1}{s_{\delta,i}}(\frac{\delta}{s_{\delta,i}})^{\frac{2}{p-1}}\phi'(\frac{|T_{z_i}(x-z_i)|}{s_{\delta,i}})\frac{(T_{z_i})_h^t\cdot T_{z_i}(x-z_i)}{|T_{z_i}(x-z_i)|},\ \ &|T_{z_i}(x-z_i)|\leq s_{\delta,i},\\
\frac{\hat{q}_{\delta,i}}{\ln s_{\delta,i}}\frac{(T_{z_i})_h^t\cdot T_{z_i}(x-z_i)}{|T_{z_i}(x-z_i)|^2},\ \ &|T_{z_i}(x-z_i)|> s_{\delta,i},
\end{cases}
\end{split}
\end{equation}
where $ (T_{z_i})_h^t $ is the $h$-th row of $ (T_{z_i})^t $.

At the end of this section, %by using \eqref{203},
we give some estimates of approximate solutions $ V_{\delta,Z} $, which will be frequently used in the following sections.
\begin{lemma}\label{lemA-5}
	Let $\gamma\in(0,1)$. There exists a constant $ L > 1 $ such that for $ \varepsilon $ small
	\begin{equation*}
	V_{\delta,Z}-q>0,\ \ \  \ \text{in}\ \   \cup_{j=1}^m\left( T_{z_j}^{-1}B_{\left( 1-L \varepsilon^\gamma \right) s_{\delta,j}}(0)+z_j\right),
	\end{equation*}
	\begin{equation*}
	V_{\delta,Z}-q<0,\ \ \  \ \text{in}\ \  \Omega\backslash\cup_{j=1}^m\left( T_{z_j}^{-1}B_{Ls_{\delta,j}}(0)+z_j\right).
	\end{equation*}
	
\end{lemma}

\begin{proof}
If $ |T_{z_j}(x-z_j)|\leq \left( 1-L \varepsilon^\gamma \right) s_{\delta,j}   $, then by \eqref{203} and $ \phi'(1)<0 $ we have
	\begin{equation*}
	\begin{split}
	V_{\delta,Z}(x)-q(x)=&V_{\delta,z_j,\hat{q}_{\delta,j}}(x)-\hat{q}_{\delta,j}+O\left( \frac{ \varepsilon^\gamma}{|\ln\varepsilon|}\right) \\
	= &\frac{\hat{q}_{\delta,j}}{|\phi'(1)|\ln\frac{1}{s_{\delta,j}}}\phi\left( \frac{|T_{z_j}(x-z_j)|}{s_{\delta,j}}\right) +O\left( \frac{\varepsilon^\gamma}{|\ln\varepsilon|}\right) >0,
	\end{split}
	\end{equation*}
	if $ L $ is sufficiently large.

	On the other hand, if $ \tau>0 $ small and $ |T_{z_j}(x-z_j)|\geq s_{\delta,j}^\tau $ for any $ j=1,\cdots,m $, then by the definition of $ V_{\delta,z_j,\hat{q}_{\delta,j}} $ and Lemma \ref{H estimate}
	\begin{equation*}
	\begin{split}
	V_{\delta,Z}(x)-q(x)=&\sum_{j=1}^m\left( V_{\delta,z_j,\hat{q}_{\delta,j}}(x)+H_{\delta,z_j,\hat{q}_{\delta,j}}(x)\right) -q(x)\\
	\leq &\sum_{j=1}^m\frac{\hat{q}_{\delta,j}\ln s_{\delta,j}^\tau}{\ln s_{\delta,j}}-C\\
	\leq &\tau \sum_{j=1}^m \hat{q}_{\delta,j}-C<0.
	\end{split}
	\end{equation*}
	If $ Ls_{\delta,j}\leq |T_{z_j}(x-z_j)|\leq s_{\delta,j}^\tau $, then by \eqref{q_i choice} for $ L $ sufficiently large
	\begin{equation*}
	\begin{split}
	V_{\delta,Z}&(x)-q(x)\\
	=&V_{\delta,z_j,\hat{q}_{\delta,j}}(x)+H_{\delta,z_j,\hat{q}_{\delta,j}}(x)-q(x)+\sum_{i\neq j}\left(V_{\delta, z_i, \hat{q}_{\delta,i}}(x)+H_{\delta, z_i, \hat{q}_{\delta,i}}(x)\right)\\
	=&V_{\delta, z_j, \hat{q}_{\delta,j}}(x)-q(z_j)-\frac{2\pi\hat{q}_{\delta,j}\sqrt{\det K(z_j)}}{\ln s_{\delta,j}}\bar{S}_K(x, z_j)-\sum_{i\neq j}\frac{2\pi\hat{q}_{\delta,i}\sqrt{\det K(z_i)}}{\ln s_{\delta,i}}G_K(x, z_i)+O(s_{\delta,j}^\tau)\\
	=&V_{\delta,z_j,\hat{q}_{\delta,j}}(x)-q(z_j)-\frac{2\pi\hat{q}_{\delta,j}\sqrt{\det K(z_j)}}{\ln s_{\delta,j}}\bar{S}_K(z_j, z_j)-\sum_{i\neq j}\frac{2\pi\hat{q}_{\delta,i}\sqrt{\det K(z_i)}}{\ln s_{\delta,i}}G_K(z_j, z_i)\\
\,\,&+O\left( \frac{\varepsilon^{\tau\gamma}}{|\ln\varepsilon|}\right) \\
	=&V_{\delta,z_j,\hat{q}_{\delta,j}}(x)-\hat{q}_{\delta,j}+O\left( \frac{\varepsilon^{\tau\gamma}}{|\ln\varepsilon|}\right) \\
	\leq &-\frac{\hat{q}_{\delta,j}\ln L}{\ln\frac{1}{s_{\delta,j}}}+O\left( \frac{\varepsilon^{\tau\gamma}}{|\ln\varepsilon|}\right) \\
<&0.
	\end{split}
	\end{equation*}

\end{proof}
\section{The reduction}
In this section  we consider the solvability of  a linear problem related to the linearized operator $ L_\delta $ for \eqref{111} at the approximate solution $ \sum_{j=1}^mV_{\delta, Z,j} $.
%First we prove that for any $ Z $ satisfying \eqref{admis set}, there exists $ \omega_{\delta,Z} $ such that $ V_{\delta, Z}+\omega_{\delta,Z} $ solves \eqref{eq1} in a co-dimensional $ 2m $ subspace of $ H^1_0 $. In the next section we choose proper $ Z=Z(\delta) $ such that $ V_{\delta, Z}+\omega_\delta $ is a solution.

Let us consider the following equation
\begin{equation}\label{eq5}
-\Delta w=w^p_+,\ \ \text{in}\ \mathbb{R}^2.
\end{equation}
The unique $ C^1 $ solution is
\begin{equation*}
w(x)=\begin{cases}
\phi(x),\ \ &|x|\leq 1,\\
\phi'(1)\ln|x|,\ \ &|x|> 1.
\end{cases}
\end{equation*}
By the classical theory for elliptic equations, $ w\in C^{2,\alpha}(\mathbb{R}^2) $ for any $ \alpha\in(0,1) $. The linearized equation of \eqref{eq5} at $ w $ is
\begin{equation}\label{limit eq}
-\Delta v-pw^{p-1}_+v=0, \ \ v\in L^{\infty}(\mathbb{R}^2).
\end{equation}
Clearly, $ \frac{\partial w}{\partial x_h} $ $(h=1,2) $  are solutions of \eqref{limit eq}. It follows from \cite{DY} (see also \cite{CLW}) that
\begin{proposition}[Non-degeneracy]\label{Non-degenerate}
$ w $ is non-degenerate, i.e., the kernel of the linearized equation \eqref{limit eq} is $$ span\left \{\frac{\partial w}{\partial x_1}, \frac{\partial w}{\partial x_2}\right \}. $$
\end{proposition}

Let $ \eta $ be a smooth  truncation function satisfying
\begin{equation*}
supp(\eta)\subseteq B_1(0),\ \ 0\leq \eta\leq 1\ \text{in}\ B_{1}(0),\   \  \eta\equiv1\ \text{in}\ B_{\frac{1}{2}}(0).
\end{equation*}
Define $ \eta_i(x)=\eta\left ((x-z_i)|\ln\varepsilon|^{M+1}\right ) $. Clearly, $ supp(\eta_i)\subseteq B_{|\ln\varepsilon|^{-M-1}}(z_i) $ and $ supp(\eta_i)\cap supp(\eta_j)=\varnothing $ for $ i\neq j $ and $ \varepsilon $ sufficiently small. Moreover, $ ||\nabla\eta_i||_{L^\infty}\leq C|\ln\varepsilon|^{M+1} $ and $ ||\nabla^2\eta_i||_{L^\infty}\leq C|\ln\varepsilon|^{2M+2} $.

Denote
\begin{equation}\label{204}
F_{\delta,Z}=\left \{u\in L^p(\Omega)\mid \int_{\Omega}u\left( \eta_j\frac{\partial V_{\delta, Z,j}}{\partial x_h}\right) =0,\ \ \forall j=1,\cdots,m,\ h=1,2\right \},
\end{equation}
and
\begin{equation}\label{205}
E_{\delta,Z}=\left \{u\in W^{2,p}\cap H^1_0(\Omega)\mid \int_{\Omega}\left (K(x)\nabla u| \nabla \left( \eta_j\frac{\partial V_{\delta, Z,j}}{\partial x_h}\right)\right )=0,\ \ \forall j=1,\cdots,m,\ h=1,2\right \}.
\end{equation}
So $ F_{\delta,Z}$ and $ E_{\delta,Z} $ are co-dimensional $ 2m $ subspaces of $ L^p $ and $ W^{2,p}\cap H^1_0(\Omega) $, respectively.

For any $ u\in L^p(\Omega) $, we  define the projection operator $ Q_\delta: L^p \to  F_{\delta,Z} $
\begin{equation}\label{206}
Q_\delta u:=u-\sum_{j=1}^m\sum_{h=1}^2b_{j,h}\left( -\delta^2\text{div}\left (K(z_j)\nabla   \frac{\partial V_{\delta, z_j, \hat{q}_{\delta,j}}}{\partial x_h}\right )\right) ,
\end{equation}
where $ b_{j,h} (j=1,\cdots,m,\ h=1,2) $ satisfies
\begin{equation}\label{207}
\sum_{j=1}^m\sum_{h=1}^2b_{j,h}\int_{\Omega}\left( -\delta^2\text{div}\left (K(z_j)\nabla   \frac{\partial V_{\delta, z_j, \hat{q}_{\delta,j}}}{\partial x_h}\right )\right)\left( \eta_i\frac{\partial V_{\delta, Z,i}}{\partial x_\hbar}\right)=\int_{\Omega}u \left( \eta_i\frac{\partial V_{\delta, Z,i}}{\partial x_\hbar}\right)
\end{equation}
for $ i=1,\cdots,m,\ \hbar=1,2. $

We claim that $ Q_\delta $ is a well-defined  linear projection operator from $ L^p $ to $ F_{\delta,Z}$. Indeed, using  \eqref{3-001} and \eqref{200},  for  $ Z\in \Lambda_{\varepsilon, m} $ the coefficient matrix
\begin{equation}\label{coef of C}
\begin{split}
\int_{\Omega}&\left( -\delta^2\text{div}\left (K(z_j)\nabla   \frac{\partial V_{\delta, z_j, \hat{q}_{\delta,j}}}{\partial x_h}\right )\right)\left( \eta_i\frac{\partial V_{\delta, Z,i}}{\partial x_\hbar}\right)\\
&=p\int_{\Omega}\eta_i(V_{\delta, z_j, \hat{q}_{\delta,j}}-\hat{q}_{\delta,j})^{p-1}_+ \frac{\partial V_{\delta, z_j, \hat{q}_{\delta,j}}}{\partial x_h}\frac{\partial V_{\delta, z_i, \hat{q}_{\delta,i}}}{\partial x_\hbar}+O\left( \frac{\varepsilon^\gamma}{|\ln\varepsilon|^{p+1}}\right)\\
&=\delta_{i,j}\frac{(M_{i})_{h,\hbar}}{|\ln\varepsilon|^{p+1}}+O\left( \frac{\varepsilon^\gamma}{|\ln\varepsilon|^{p+1}}\right) ,
\end{split}
\end{equation}
where $ \delta_{i,j}= 1 $ if $ i = j $ and $ \delta_{i,j}= 0 $ otherwise.  $ M_{i} $ are $ m $ positive definite matrices %and there exist positive constants $ \bar{c}_1,\bar{c}_2 $ independent of $ \delta, Z $
such that all eigenvalues of $ M_{i} $ belong to $ (\bar{c}_1,\bar{c}_2) $ for constants $ \bar{c}_1,\bar{c}_2>0 $.
This implies the existence and uniqueness of $ b_{j,h} $.
Note that for $ u\in L^p $, $ Q_{\delta} u \equiv u $ in $ \Omega \backslash\cup_{i=1}^mB_{Ls_{\delta,i}}(z_{i}) $ for some $ L>1 $. Moreover, one can easily get that there exists a constant $ C>0 $ independent of $ \delta $, such that for any $ q\in[1,+\infty) $, $ u\in L^q(\Omega) $ with $ supp(u)\subset  \cup_{j=1}^mB_{Ls_{\delta,j}}(z_j)$,
\begin{equation*}
||Q_\delta u||_{L^q(\Omega)}\leq C||u||_{L^q(\Omega)}.
\end{equation*}

The linearized operator of \eqref{111} at $ V_{\delta, Z} $ is
\begin{equation*}
L_\delta \omega =-\delta^2\text{div}(K(x)\nabla \omega)- p(V_{\delta, Z}-q)^{p-1}_+\omega.
\end{equation*}
The following lemma gives estimates of the linear operator $ Q_\delta L_\delta $.
\begin{lemma}\label{coercive esti}
There exist  $ \rho_0>0, \delta_0>0 $ such that for any $ \delta\in(0,\delta_0],  Z\in \Lambda_{\varepsilon, m} $, if $ u\in E_{\delta,Z} $ satisfying $ Q_\delta L_\delta u=0 $ in $ \Omega\backslash\cup_{j=1}^mB_{Ls_{\delta,j}}(z_j) $ for some $ L>1 $ large, then
\begin{equation*}
||Q_\delta L_\delta u||_{L^p}\geq \frac{\rho_0\varepsilon^{\frac{2}{p}}}{|\ln\varepsilon|^{p-1}}||u||_{L^\infty}.
\end{equation*}
\end{lemma}

\begin{proof}
We argue by contradiction. Suppose that there are $ \delta_N \to 0 $, $ Z_N=(z_{N,1},\cdots,z_{N,m})\to (z_{1},\cdots,z_{m})\in B_{\bar{\rho}}(x_{0})^{(m)} $  and $  u_N \in E_{\delta_N, Z_N} $ with $ Q_{\delta_N} L_{\delta_N} u_N = 0 $ in $ \Omega \backslash\cup_{j=1}^mB_{Ls_{\delta_N,j}}(z_{N,j}) $ for some $ L $ large and $ ||u_N||_{L^\infty}=1 $ such that
\begin{equation*}
||Q_{\delta_N} L_{\delta_N} u_N||_{L^p}\leq \frac{1}{N}\frac{\varepsilon_N^{\frac{2}{p}}}{|\ln\varepsilon_N|^{p-1}}.
\end{equation*}
 Let
\begin{equation}\label{208}
Q_{\delta_N} L_{\delta_N} u_N=  L_{\delta_N} u_N-\sum_{j=1}^m\sum_{h=1}^2b_{j,h,N}\left( -\delta_N^2\text{div}\left (K(z_{N,j})\nabla   \frac{\partial V_{\delta_N, z_{N,j},\hat{q}_{\delta_N,j}}}{\partial x_h}\right )\right) .
\end{equation}
We now estimate $ b_{j,h,N} $. For fixed $ i=1,\cdots,m, \hbar=1,2 $, multiplying \eqref{208} by $ \eta_i\frac{\partial V_{\delta_N, Z_{N},i}}{\partial x_\hbar} $ and integrating on $ \Omega $ we get
\begin{equation*}
\begin{split}
\int_{\Omega}u_N&L_{\delta_N}\left( \eta_i\frac{\partial V_{\delta_N, Z_{N},i}}{\partial x_\hbar}\right) =\int_{\Omega}L_{\delta_N}u_N\left( \eta_i\frac{\partial V_{\delta_N, Z_{N},i}}{\partial x_\hbar}\right) \\
=&\sum_{j=1}^m\sum_{h=1}^2b_{j,h,N}\int_{\Omega}-\delta_N^2\text{div}\left (K(z_{N,j})\nabla   \frac{\partial V_{\delta_N, z_{N,j},\hat{q}_{\delta_N,j}}}{\partial x_h}\right )\left( \eta_i\frac{\partial V_{\delta_N, Z_{N},i}}{\partial x_\hbar}\right) .
\end{split}
\end{equation*}

We estimate $ \int_{\Omega}u_NL_{\delta_N}\left( \eta_i\frac{\partial V_{\delta_N, Z_{N},i}}{\partial x_\hbar}\right). $ Note that
\begin{equation*}
\begin{split}
&\int_{\Omega}u_NL_{\delta_N}\left( \eta_i\frac{\partial V_{\delta_N, Z_{N},i}}{\partial x_\hbar}\right) \\
=&-\int_{\Omega}u_N\delta_N^2\text{div}\left (K(x)\nabla \left( \eta_i\frac{\partial V_{\delta_N, Z_{N},i}}{\partial x_\hbar}\right) \right )- p\int_{\Omega}u_N(V_{\delta_N, Z_N}-q)^{p-1}_+\left( \eta_i\frac{\partial V_{\delta_N, Z_{N},i}}{\partial x_\hbar}\right)  \\
=&-\int_{\Omega}\eta_iu_N\delta_N^2\text{div}\left (K(x)\nabla \frac{\partial V_{\delta_N, Z_{N},i}}{\partial x_\hbar} \right )-2\int_{\Omega}u_N\delta_N^2\left (K(x)\nabla\eta_i|\nabla   \frac{\partial V_{\delta_N, Z_{N},i}}{\partial x_\hbar}  \right )\\
&-\int_{\Omega}u_N\delta_N^2\text{div}\left (K(x)\nabla   \eta_i  \right )\frac{\partial V_{\delta_N, Z_{N},i}}{\partial x_\hbar}- p\int_{\Omega}u_N(V_{\delta_N, Z_N}-q)^{p-1}_+\left( \eta_i\frac{\partial V_{\delta_N, Z_{N},i}}{\partial x_\hbar}\right)
\end{split}
\end{equation*}
\begin{equation}\label{301}
\begin{split}
=&\int_{\Omega}\eta_iu_Np\left( V_{\delta_N, z_{N,i}, \hat{q}_{\delta_N,i}}-\hat{q}_{\delta_N,i}\right)^{p-1}_+\frac{\partial V_{\delta_N, z_{N,i}, \hat{q}_{\delta_N,i}}}{\partial x_\hbar}+\int_{\Omega}\eta_iu_N\delta_N^2\text{div}\left (\frac{\partial K(x)}{\partial x_\hbar}\nabla V_{\delta_N, Z_{N},i}  \right )\\
&-2\int_{\Omega}u_N\delta_N^2\left (K(x)\nabla\eta_i|\nabla   \frac{\partial V_{\delta_N, Z_{N},i}}{\partial x_\hbar}  \right ) -\int_{\Omega}u_N\delta_N^2\text{div}\left (K(x)\nabla   \eta_i  \right )\frac{\partial V_{\delta_N, Z_{N},i}}{\partial x_\hbar}\\
&- p\int_{\Omega}u_N(V_{\delta_N, Z_N}-q)^{p-1}_+\left( \eta_i\frac{\partial V_{\delta_N, Z_{N},i}}{\partial x_\hbar}\right).
%=:&I_1+I_2+I_3+I_4+I_5.
\end{split}
\end{equation}

By \eqref{3-001}, \eqref{203} and Lemma \ref{lemA-5}, one has
\begin{equation}\label{302}
\begin{split}
\int_{\Omega}\eta_i&u_Np\left( V_{\delta_N, z_{N,i}, \hat{q}_{\delta_N,i}}-\hat{q}_{\delta_N,i}\right)^{p-1}_+\frac{\partial V_{\delta_N, z_{N,i}, \hat{q}_{\delta_N,i}}}{\partial x_\hbar}-  p\int_{\Omega}u_N(V_{\delta_N, Z_N}-q)^{p-1}_+\left( \eta_i\frac{\partial V_{\delta_N, Z_{N},i}}{\partial x_\hbar}\right)\\
=&p\int_{\Omega}u_N\left( V_{\delta_N, z_{N,i}, \hat{q}_{\delta_N,i}}-\hat{q}_{\delta_N,i}\right)^{p-1}_+\frac{\partial V_{\delta_N, z_{N,i}, \hat{q}_{\delta_N,i}}}{\partial x_\hbar}\\
&-p\int_{\Omega}u_N\left( V_{\delta_N, z_{N,i}, \hat{q}_{\delta_N,i}}-\hat{q}_{\delta_N,i}+O\left( \frac{\varepsilon_N^\gamma}{|\ln\varepsilon_N|}\right) \right)^{p-1}_+\frac{\partial V_{\delta_N, z_{N,i}, \hat{q}_{\delta_N,i}}}{\partial x_\hbar}+O\left( \frac{\varepsilon_N^{1+\gamma}}{|\ln\varepsilon_N|^p}\right)\\
=&O\left( \frac{\varepsilon_N^{1+\gamma}}{|\ln\varepsilon_N|^p}\right).
\end{split}
\end{equation}
By the choice of $ \eta_i $, we have $ ||\nabla\eta_i||_{L^\infty}\leq C|\ln\varepsilon|^{M+1}, ||\nabla^2\eta_i||_{L^\infty}\leq C|\ln\varepsilon|^{2M+2} $. Thus there holds
\begin{equation}\label{304}
\begin{array}{ll}
\int_{\Omega}\eta_iu_N\delta_N^2\text{div}\left (\frac{\partial K(x)}{\partial x_\hbar}\nabla V_{\delta_N, Z_{N},i}  \right )&\\
\qquad\qquad\,=\int\limits_{|T_{z_{N,i}}(x-z_{N,i})|\leq s_{\delta_N,i}}\eta_iu_N \left( \delta_N^2\text{div}\left( \frac{\partial K(x)}{\partial x_\hbar}\nabla V_{\delta_N, z_{N,i}, \hat{q}_{\delta_N,i}}\right)\right)&\\
\qquad\qquad\quad\,\,\,+\int\limits_{ s_{\delta_N,i}<|T_{z_{N,i}}(x-z_{N,i})|\leq |\ln\varepsilon|^{-M-1}}
\eta_iu_N \left( \delta_N^2\text{div}\left( \frac{\partial K(x)}{\partial x_\hbar}\nabla V_{\delta_N, z_{N,i}, \hat{q}_{\delta_N,i}}\right)\right)&\\
\qquad\qquad\quad\,\,+O\left( \frac{\delta_N^2}{|\ln\varepsilon_N|}\right)  &\\
\qquad\qquad\,=O\left( \frac{\delta_N^2}{|\ln\varepsilon_N|}\right) +O(\delta_N^2)+O\left( \frac{\delta_N^2}{|\ln\varepsilon_N|}\right)&\\
\qquad\qquad\,=O\left( \frac{\varepsilon_N^2}{|\ln\varepsilon_N|^{p-1}}\right),
\end{array}
\end{equation}
\begin{equation}\label{305}
\begin{split}
-2\int_{\Omega}&u_N\delta_N^2\left (K(x)\nabla\eta_i|\nabla   \frac{\partial V_{\delta_N, Z_{N},i}}{\partial x_\hbar}  \right )\\
=&-2\int_{B_{|\ln\varepsilon|^{-M-1}}(z_{N,i})\backslash B_{\frac{|\ln\varepsilon|^{-M-1}}{2}}(z_{N,i})}u_N\delta_N^2\left (K(x)\nabla\eta_i|\nabla   \frac{\partial V_{\delta_N, z_{N,i},\hat{q}_{\delta_N,i}}}{\partial x_\hbar}  \right )\\
\,\,\,&+O\left( \delta_N^2|\ln\varepsilon_N|^{M}\right)\\
=&O\left( \delta_N^2|\ln\varepsilon_N|^{M}\right),
\end{split}
\end{equation}
\begin{equation}\label{306}
\begin{split}
-\int_{\Omega}&u_N\delta_N^2\text{div}\left (K(x)\nabla   \eta_i  \right )\frac{\partial V_{\delta_N, Z_{N},i}}{\partial x_\hbar}\\
=&-\int_{B_{|\ln\varepsilon|^{-M-1}}(z_{N,i})\backslash B_{\frac{|\ln\varepsilon|^{-M-1}}{2}}(z_{N,i})}u_N\delta_N^2\text{div}\left (K(x)\nabla   \eta_i  \right )\frac{\partial V_{\delta_N, z_{N,i},\hat{q}_{\delta_N,i}}}{\partial x_\hbar}\\
\,\,&+O\left( \delta_N^2|\ln\varepsilon_N|^{2M+1}\right)\\
=&O\left( \delta_N^2|\ln\varepsilon_N|^{2M+1}\right),
\end{split}
\end{equation}
where we have used \eqref{200}. Taking \eqref{302}, \eqref{304}, \eqref{305} and \eqref{306} into \eqref{301}, we get
\begin{equation*}
\begin{split}
&\int_{\Omega}u_NL_{\delta_N}\left( \eta_i\frac{\partial V_{\delta_N, Z_{N},i}}{\partial x_\hbar}\right) =O\left( \frac{\varepsilon_N^{1+\gamma}}{|\ln\varepsilon_N|^p}\right).
\end{split}
\end{equation*}
Combining with \eqref{coef of C} we obtain
\begin{equation*}
b_{j,h,N}=O\left (\varepsilon_N^{1+\gamma}|\ln\varepsilon_N|\right ),
\end{equation*}
which implies that
\begin{equation*}
\begin{split}
\sum_{j=1}^m\sum_{h=1}^2&b_{j,h,N}\left( -\delta_N^2\text{div}\left (K(z_{N,j})\nabla   \frac{\partial V_{\delta_N, z_{N,j},\hat{q}_{\delta_N,j}}}{\partial x_h}\right )\right)\\
%=&p\sum_{j=1}^m\sum_{h=1}^2b_{j,h,N}(V_{\delta_N, z_{N,j}, \hat{q}_{\delta_N,j}}-\hat{q}_{\delta_N,j})^{p-1}_+ \frac{\partial V_{\delta_N, z_{N,j},\hat{q}_{\delta_N,j}}}{\partial x_h} \\
&=O\left( \sum_{j=1}^m\sum_{h=1}^2\frac{\varepsilon_N^{\frac{2}{p}-1}|b_{j,h,N}|}{|\ln\varepsilon_N|^{p}}\right)
=O\left( \frac{\varepsilon_N^{\frac{2}{p}+\gamma}}{|\ln\varepsilon_N|^{p-1}}\right),\ \ \ \ \text{in}\ L^p(\Omega).
\end{split}
\end{equation*}
Hence by \eqref{208} we have
\begin{equation}\label{307}
\begin{split}
L_{\delta_N} u_N=&Q_{\delta_N} L_{\delta_N} u_N+\sum_{j=1}^m\sum_{h=1}^2b_{j,h,N}\left( -\delta_N^2\text{div}\left (K(z_{N,j})\nabla   \frac{\partial V_{\delta_N, z_{N,j},\hat{q}_{\delta_N,j}}}{\partial x_h}\right )\right)\\
=&O\left( \frac{1}{N}\frac{\varepsilon_N^{\frac{2}{p}}}{|\ln\varepsilon_N|^{p-1}}\right) +O\left(  \frac{\varepsilon_N^{\frac{2}{p}+\gamma}}{|\ln\varepsilon_N|^{p-1}}\right)
=o\left( \frac{\varepsilon_N^{\frac{2}{p}}}{|\ln\varepsilon_N|^{p-1}}\right),\ \ \ \ \text{in}\ L^p(\Omega).
\end{split}
\end{equation}

For fixed $ i, $ we define the scaled function $ \tilde{u}_{N,i}(y)=u_N(s_{\delta_N, i}y+z_{N,i}) $ for $ y\in \Omega_{N,i}:=\{y\in\mathbb{R}^2\mid s_{\delta_N, i}y+z_{N,i}\in \Omega\} $.
Define
\begin{equation*}
\tilde{L}_{N,i}u=-\text{div}(K(s_{\delta_N, i}y+z_{N,i})\nabla u)-p\frac{s_{\delta_N,i}^2}{\delta_N^2}(V_{\delta_N, Z_N}(s_{\delta_N, i}y+z_{N,i})-q(s_{\delta_N, i}y+z_{N,i}))^{p-1}_+u.
\end{equation*}
Then
\begin{equation*}
\begin{split}
||\tilde{L}_{N,i}\tilde{u}_{N,i}||_{L^p(\Omega_{N,i})}
%=&\left( \int_{\Omega_{N,i}}\left( \tilde{L}_{N,i}\tilde{u}_{N,i} \right)^pdy\right)^{\frac{1}{p}}  \\
%=&\left( \frac{1}{s_{\delta_N, i}^2}\int_{\Omega}\left( -s_{\delta_N, i}^2\text{div}(K(x)\nabla u_N)-p\frac{s_{\delta_N, i}^2}{\delta_N^2}(V_{\delta_N,Z_N}-q)^{p-1}_+u_N \right)^pdx\right)^{\frac{1}{p}}\\
=\frac{s_{\delta_N, i}^2}{s_{\delta_N, i}^{\frac{2}{p}}\delta_N^2}||L_{\delta_N}u_N||_{L^p(\Omega)}.
\end{split}
\end{equation*}
%i.e., $ s_{\delta_N, i}^{\frac{2}{p}}\cdot\frac{\delta_N^2}{s_{\delta_N, i}^2}||\tilde{L}_{N,i}\tilde{u}_{N,i}||_{L^p(\Omega_{N,i})}=||L_{\delta_N}u_N||_{L^p(\Omega)}. $
Note that $ \frac{\delta_N^2}{s_{\delta_N, i}^2}=O(\frac{1}{|\ln\varepsilon_N|^{p-1}}) $ and $ s_{\delta_N, i}=O(\varepsilon_N) $, thus by \eqref{307} we get
\begin{equation*}
\tilde{L}_{N,i}\tilde{u}_{N,i}=o(1)\ \ \ \  \text{in}\ \ L^p(\Omega_{N,i}).
\end{equation*}
Since $ ||\tilde{u}_{N,i}||_{L^\infty(\Omega_{N,i})}=1 $, by the classical regularity theory of  elliptic equations, $ \tilde{u}_{N,i} $ is uniformly bounded in $ W^{2,p}_{loc}(\mathbb{R}^2) $, which implies that
\begin{equation*}
\tilde{u}_{N,i}\to u_i\ \ \ \ \text{in}\ \ C^1_{loc}(\mathbb{R}^2).
\end{equation*}

We claim that $ u_i\equiv 0. $ On the one hand, note that for $ Z\in \Lambda_{\varepsilon, m} $, $ |z_i-z_j|\geq |\ln\varepsilon|^{-M} $. So by \eqref{203}, $ z_{N,i}\to z_i $ as $ N\to\infty $ and the fact that $ \lim_{\varepsilon\to 0}\varepsilon|\ln\varepsilon|^M=0 $, we get
\begin{equation*}
\begin{split}
\frac{s_{\delta_N,i}^2}{\delta_N^2}&(V_{\delta_N, Z_N}(s_{\delta_N, i}y+z_{N,i})-q(s_{\delta_N, i}y+z_{N,i}))^{p-1}_+\\
=&\frac{s_{\delta_N,i}^2}{\delta_N^2}\left( V_{\delta_N, z_{N,i}, \hat{q}_{\delta_N,i}}(s_{\delta_N, i}y+z_{N,i})-\hat{q}_{\delta_N,i}+O\left (\frac{\varepsilon_N^{\gamma}}{|\ln\varepsilon_N|}\right )\right)^{p-1}_+ \\
\to&\phi(T_{z_i}y)^{p-1}_+\ \  \ \ \text{in}\ C^0_{loc}(\mathbb{R}^2)\ \text{as}\ N\to\infty,
\end{split}
\end{equation*}
from which we deduce that $ u_i $ satisfies
\begin{equation*}
-\text{div}(K(z_i)\nabla u_i(x))-p\phi(T_{z_i}x)^{p-1}_+u_i(x)=0,\ \ x\in\mathbb{R}^2.
\end{equation*}
Let $ \hat{u}_i(x)=u_i(T_{z_i}^{-1}x) $. Since $ T_{z_i}^{-1}(T_{z_i}^{-1})^t=K(z_i) $, we have
\begin{equation*}
-\Delta \hat{u}_i(x)=-\text{div}(K(z_i)\nabla u_i)(T_{z_i}^{-1}x)=p\phi(x)^{p-1}_+\hat{u}_i(x),\ \ \ \   \ x\in \mathbb{R}^2.
\end{equation*}
By Proposition \ref{Non-degenerate}, there exist $ c_1,c_2 $ such that
\begin{equation}\label{210}
\hat{u}_i=c_1\frac{\partial \phi}{\partial x_1}+c_2\frac{\partial \phi}{\partial x_2}.
\end{equation}

On the other hand, since $ u_N\in E_{\delta_N,Z_N} $, we get
\begin{equation*}
\int_{\Omega}-\delta_N^2\text{div}\left( K(x)\nabla \left( \eta_i\frac{\partial V_{\delta_N, Z_{N}, i}}{\partial x_\hbar}\right)\right)  u_N=0,\ \ \forall\ i=1,\cdots,m,\  \hbar=1,2,
\end{equation*}
which implies that
\begin{equation}\label{209}
\begin{split}
0%=&-\int_{\Omega}u_N\eta_i\delta_N^2\text{div}\left( K(x)\nabla \frac{\partial V_{\delta_N, z_{N,i}, \hat{q}_{\delta_N,i}}}{\partial x_\hbar}\right) -2\int_{\Omega}u_N\delta_N^2\left (K(x)\nabla\eta_i|\nabla   \frac{\partial V_{\delta_N, z_{N,i},\hat{q}_{\delta_N,i}}}{\partial x_\hbar}  \right )\\
%& -\int_{\Omega}u_N\delta_N^2\text{div}\left (K(x)\nabla   \eta_i  \right )\frac{\partial V_{\delta_N, z_{N,i},\hat{q}_{\delta_N,i}}}{\partial x_\hbar}\\
=&p\int_{\Omega}u_N\left( V_{\delta_N, z_{N,i}, \hat{q}_{\delta_N,i}}-\hat{q}_{\delta_N,i}\right)^{p-1}_+\frac{\partial V_{\delta_N, z_{N,i}, \hat{q}_{\delta_N,i}}}{\partial x_\hbar}+\int_{\Omega}\eta_iu_N\delta_N^2\text{div}\left (\frac{\partial K(x)}{\partial x_\hbar}\nabla V_{\delta_N, Z_{N},i}  \right )\\
&-2\int_{\Omega}u_N\delta_N^2\left (K(x)\nabla\eta_i|\nabla   \frac{\partial V_{\delta_N, Z_{N},i}}{\partial x_\hbar}  \right ) -\int_{\Omega}u_N\delta_N^2\text{div}\left (K(x)\nabla   \eta_i  \right )\frac{\partial V_{\delta_N, Z_{N},i}}{\partial x_\hbar}.
\end{split}
\end{equation}
By \eqref{304}, \eqref{305} and \eqref{306},
\begin{equation}\label{209-1}
\begin{split}
\int_{\Omega}&\eta_iu_N\delta_N^2\text{div}\left (\frac{\partial K(x)}{\partial x_\hbar}\nabla V_{\delta_N, Z_{N},i}  \right )-2\int_{\Omega}u_N\delta_N^2\left (K(x)\nabla\eta_i|\nabla   \frac{\partial V_{\delta_N, Z_{N},i}}{\partial x_\hbar}  \right ) \\
&-\int_{\Omega}u_N\delta_N^2\text{div}\left (K(x)\nabla   \eta_i  \right )\frac{\partial V_{\delta_N, Z_{N},i}}{\partial x_\hbar}
=O\left( \delta_N^2|\ln\varepsilon_N|^{2M+1}\right).
\end{split}
\end{equation}
%and
%\begin{equation}\label{209-2}
%-2\int_{\Omega}u_N\delta_N^2\left (K(x)\nabla\eta_i|\nabla   \frac{\partial V_{\delta_N, z_{N,i},\hat{q}_{\delta_N,i}}}{\partial x_\hbar}  \right ) -\int_{\Omega}u_N\delta_N^2\text{div}\left (K(x)\nabla   \eta_i  \right )\frac{\partial V_{\delta_N, z_{N,i},\hat{q}_{\delta_N,i}}}{\partial x_\hbar}=O\left( \frac{\varepsilon_N^2}{|\ln\varepsilon_N|^{p}}\right).
%\end{equation}
It follows from  \eqref{200} that
\begin{equation}\label{209-3}
\begin{split}
&p\int_{\Omega} u_N(V_{\delta_N, z_{N,i}, \hat{q}_{\delta_N,i}}-\hat{q}_{\delta_N,i})^{p-1}_+ \frac{\partial V_{\delta_N, z_{N,i}, \hat{q}_{\delta_N,i}}}{\partial x_\hbar}\\
=&p\int_{\Omega}\frac{1}{s_{\delta_N,i}}\left( \frac{\delta_N}{s_{\delta_N,i}}\right)^{\frac{2p}{p-1}} \phi\left( \frac{T_{z_{N,i}}(x-z_{N,i})}{s_{\delta_N,i}}\right)^{p-1}_+\phi'\left( \frac{T_{z_{N,i}}(x-z_{N,i})}{s_{\delta_N,i}}\right)\frac{(T_{z_{N,i}})_\hbar^t\cdot T_{z_{N,i}}(x-z_{N,i})}{|T_{z_{N,i}}(x-z_{N,i})|}u_N\\
=&ps_{\delta_N,i}\left( \frac{\delta_N}{s_{\delta_N,i}}\right)^{\frac{2p}{p-1}}\int_{\mathbb{R}^2}\phi(T_{z_{N,i}}y)^{p-1}_+\phi'(T_{z_{N,i}}y)\frac{(T_{z_{N,i}})_\hbar^t\cdot T_{z_{N,i}}y}{|T_{z_{N,i}}y|}\tilde{u}_{N,i}(y)dy.
\end{split}
\end{equation}
Taking \eqref{209-1}  and \eqref{209-3} into \eqref{209}, we have
\begin{equation}\label{209-0}
\begin{split}
0=&ps_{\delta_N,i}\left( \frac{\delta_N}{s_{\delta_N,i}}\right)^{\frac{2p}{p-1}}\int_{\mathbb{R}^2}\phi(T_{z_{N,i}}y)^{p-1}_+\phi'(T_{z_{N,i}}y)\frac{(T_{z_{N,i}})_\hbar^t\cdot T_{z_{N,i}}y}{|T_{z_{N,i}}y|}\tilde{u}_{N,i}(y)dy+O\left( \delta_N^2|\ln\varepsilon_N|^{2M+1}\right).
\end{split}
\end{equation}
Dividing both sides of \eqref{209-0} into $ ps_{\delta_N,i}(\frac{\delta_N}{s_{\delta_N,i}})^{\frac{2p}{p-1}} $ and passing $ N $ to the limit, we get
\begin{equation*}
\begin{split}
0%=&\int_{\mathbb{R}^2}\phi(T_{z_{i}}y)^{p-1}_+\phi'(T_{z_{i}}y)\frac{(T_{z_{i}})_\hbar^t\cdot T_{z_{i}}y}{|T_{z_{i}}y|}u_{i}(y)dy\\
=\int_{\mathbb{R}^2}\phi(x)^{p-1}_+\phi'(x)\frac{(T_{z_{i}})_\hbar^t\cdot x}{|x|}\hat{u}_{i}(x)\sqrt{\det(K(z_i))}dx,\ \  \hbar=1,2,
\end{split}
\end{equation*}
which implies that
\begin{equation}\label{211}
0=\int_{B_1(0)}\phi^{p-1}_+\frac{\partial \phi}{\partial x_h}\hat{u}_i.
\end{equation}
Combining \eqref{210} with \eqref{211}, there holds $ c_1=c_2=0. $ That is, $ u_i\equiv 0. $ So we conclude that $ \tilde{u}_{N,i}\to 0 $ in  $ C^1(B_L(0)) $, which implies that
\begin{equation}\label{212}
||u_N||_{L^\infty(B_{Ls_{\delta_N,i}}(z_{N,i}))}=o(1).
\end{equation}

Since $ Q_{\delta_N} L_{\delta_N} u_N = 0 $ in $ \Omega \backslash\cup_{i=1}^mB_{Ls_{\delta_N,i}}(z_{N,i}) $, we have for $ L $ large
\begin{equation*}
 L_{\delta_N} u_N = 0\  \ \text{in} \ \Omega \backslash\cup_{i=1}^mB_{Ls_{\delta_N,i}}(z_{N,i}).
\end{equation*}
By Lemma \ref{lemA-5}, one has $ (V_{\delta_N,Z_N}-q)_+=0\  \ \text{in} \ \Omega \backslash\cup_{i=1}^mB_{Ls_{\delta_N,i}}(z_{N,i}). $
So $ -\text{div}(K(x)\nabla u_N)=0 $ in $ \Omega \backslash\cup_{i=1}^mB_{Ls_{\delta_N,i}}(z_{N,i}). $
%Since $ u_N=o(1) $ on $ \cup_{i=1}^m\partial B_{Ls_{\delta_N,i}}(z_{N,i}) $ and $ u_N=0 $ on $ \partial \Omega $,
Thus by the maximum principle, we get
\begin{equation*}
||u_N||_{L^\infty(\Omega\backslash \cup_{i=1}^mB_{Ls_{\delta_N,i}}(z_{N,i}))}=o(1),
\end{equation*}
which combined with \eqref{212} we have
\begin{equation*}
||u_N||_{L^\infty(\Omega)}=o(1).
\end{equation*}
This is a contradiction since $ ||u_N||_{L^\infty(\Omega)}=1. $

\end{proof}

A direct consequence of Lemma \ref{coercive esti} is that $ Q_\delta L_\delta $ is indeed a one to one and onto map from $ E_{\delta,Z} $ to $ F_{\delta, Z}. $
\begin{proposition}\label{one to one and onto}
$ Q_\delta L_\delta $ is  a one to one and onto map from $ E_{\delta,Z} $ to $ F_{\delta, Z}. $
\end{proposition}

\begin{proof}
If $ Q_\delta L_\delta u=0 $, by Lemma \ref{coercive esti}, $ u=0 $. So $ Q_\delta L_\delta $ is  one to one.

 Denote
\begin{equation*}
\hat{E}=\left \{u\in H^1_0(\Omega)\mid \int_{\Omega}\left( K(x)\nabla u|\nabla\left(\eta_i \frac{\partial V_{\delta, Z,i}}{\partial x_h}\right)\right)  =0,\ \ \ \ i=1,\cdots,m,\ h=1,2\right \}.
\end{equation*}
Then $ E_{\delta,Z}=\hat{E}\cap W^{2,p}(\Omega) $. For any $ \hat{h}\in F_{\delta,Z}  $, by the Riesz representation theorem there is a unique $ u\in H^1_0(\Omega) $ such that
\begin{equation}\label{213}
\delta^2\int_{\Omega}\left( K(x)\nabla u|\nabla\varphi\right) =\int_{\Omega}\hat{h}\varphi,\ \ \   \ \forall \varphi\in H^1_0(\Omega).
\end{equation}

Since $ \hat{h}\in F_{\delta,Z} $, we have $ u\in \hat{E}. $ Using the classical $ L^p $ theory of elliptic equations, we conclude that $ u\in W^{2,p}(\Omega) $, which implies that $ u\in E_{\delta,Z}. $ Thus $ -\delta^2\text{div}(K(x)\nabla)=Q_{\delta}(-\delta^2\text{div}(K(x)\nabla)) $ is a one to one and onto map from $ E_{\delta,Z} $ to $ F_{\delta, Z}. $

For any $ h\in F_{\delta,Z} $,  $ Q_\delta L_\delta u=h $ is equivalent to
\begin{equation}\label{214}
u=(Q_{\delta}(-\delta^2\text{div}(K(x)\nabla)))^{-1}pQ_\delta(V_{\delta,Z}-q)^{p-1}_+u+(Q_{\delta}(-\delta^2\text{div}(K(x)\nabla)))^{-1}h,\ \ u\in E_{\delta,Z}.
\end{equation}
Note that $  (Q_{\delta}(-\delta^2\text{div}(K(x)\nabla)))^{-1}pQ_\delta(V_{\delta,Z}-q)^{p-1}_+u $ is a compact operator in $ E_{\delta,Z}, $ by the Fredholm alternative, \eqref{214} is solvable if and only if
\begin{equation*}
u=(Q_{\delta}(-\delta^2\text{div}(K(x)\nabla)))^{-1}pQ_\delta(V_{\delta,Z}-q)^{p-1}_+u
\end{equation*}
has only trivial solution, which is true since $ Q_\delta L_\delta  $ is one to one. So $ Q_\delta L_\delta $ is  an onto  map from $ E_{\delta,Z} $ to $ F_{\delta, Z}  $ and the proof is complete.
\end{proof}

\section{Solvability of a nonlinear equation}
%Now we  find solution of \eqref{111} being of the form $$ V_{\delta, Z}+\omega_\delta. $$ First we prove that for any $ Z $ satisfying \eqref{admis set}, there exists $ \omega_{\delta,Z} $ such that $ V_{\delta, Z}+\omega_{\delta,Z} $ solves \eqref{eq1} in a co-dimensional $ 2m $ subspace of $ H^1_0 $. In the next section we choose proper $ Z=Z(\delta) $ such that $ V_{\delta, Z}+\omega_\delta $ is a solution.

%From \eqref{3-03}, to construct solutions of \eqref{111} being the form  $ V_{\delta,Z}+\omega_\delta $, it suffices to solve
%\begin{equation}\label{215}
%\begin{split}
%L_\delta \omega_\delta=l_\delta+R_\delta(\omega_\delta),
%\end{split}
%\end{equation}
%where $ l_\delta=\left (\sum_{j=1}^mV_{\delta,Z,j}-q\right )^{p}_+-\sum_{j=1}^m(V_{\delta, z_j, \hat{q}_j}-\hat{q}_j)^{p}_+  $
%and $ R_\delta(\omega_\delta):=\left (\sum_{j=1}^mV_{\delta,Z,j}+\omega_\delta-q\right )^{p}_+-\left (\sum_{j=1}^mV_{\delta,Z,j}-q\right )^{p}_+-p\left (\sum_{j=1}^mV_{\delta,Z,j}-q\right )^{p-1}_+\omega_\delta. $

In this section, we solve solutions $ \omega\in E_{\delta,Z}  $ of the following nonlinear equation
\begin{equation}\label{216}
Q_\delta L_\delta \omega=Q_\delta l_\delta+Q_\delta R_\delta(\omega),
\end{equation}
or equivalently,
\begin{equation*}
\begin{split}
\omega=T_\delta(\omega):=(Q_\delta L_\delta)^{-1}Q_\delta l_\delta+(Q_\delta L_\delta)^{-1}Q_\delta R_\delta(\omega).
\end{split}
\end{equation*}
%And then we can reduce \eqref{215} to a finite dimensional problem.
We have
\begin{proposition}\label{exist and uniq of w}
There exists $ \delta_0>0, $ such that for any $ \gamma\in (0,1) $, $ 0<\delta<\delta_0 $ and $ Z\in \Lambda_{\varepsilon, m} $, \eqref{216} has the unique solution $ \omega_{\delta,Z}\in E_{\delta,Z} $ with
\begin{equation*}
||\omega_{\delta,Z}||_{L^\infty(\Omega)}=O\left(\frac{\varepsilon^\gamma}{|\ln\varepsilon|}\right).
\end{equation*}
\end{proposition}

\begin{proof}
It follows from  Lemma \ref{lemA-5} that for $ L $ sufficiently large and $ \delta $ small,
\begin{equation*}
(V_{\delta,Z}-q)_+=0,\ \ \ \ \text{in}\ \Omega \backslash\cup_{i=1}^mB_{Ls_{\delta,i}}(z_{i}).
\end{equation*}
Let $ \mathcal{N}= E_{\delta,Z} \cap\{\omega\mid ||\omega||_{L^\infty(\Omega)}\leq \frac{1}{|\ln\varepsilon|^{2-\theta_0}}\}$ for some $ \theta_0\in(0,1). $ Then $ \mathcal{N} $ is complete under $ L^\infty $ norm and $ T_\delta $ is a map from $ E_{\delta,Z} $ to $ E_{\delta,Z} $. We now prove that $ T_\delta $ is a contraction map from $ \mathcal{N} $ to $ \mathcal{N} $.

First, we claim that $ T_\delta $ is a map from $ \mathcal{N} $ to $ \mathcal{N} $. For any $ \omega\in \mathcal{N} $, by Lemma \ref{lemA-5} we get that for $ L>1 $ large and  $ \delta $ small,
\begin{equation*}
(V_{\delta,Z}+\omega-q)_+=0,\ \ \ \ \text{in}\ \Omega \backslash\cup_{i=1}^mB_{Ls_{\delta,i}}(z_{i}).
\end{equation*}
So $ l_\delta=R_\delta(\omega)=0 $ in $ \Omega \backslash\cup_{i=1}^mB_{Ls_{\delta,i}}(z_{i}). $ By the definition of $ Q_\delta $,
\begin{equation*}
Q_\delta l_\delta+Q_\delta R_\delta(\omega)=0,\ \ \ \ \text{in}\ \Omega \backslash\cup_{i=1}^mB_{Ls_{\delta,i}}(z_{i}).
\end{equation*}
Thus by Lemma \ref{coercive esti}, we obtain
\begin{equation*}
||T_\delta(\omega)||_{L^\infty}=||(Q_\delta L_\delta)^{-1}(Q_\delta l_\delta+Q_\delta R_\delta(\omega))||_{L^\infty}\leq C\frac{|\ln\varepsilon|^{p-1}}{\varepsilon^{\frac{2}{p}}}||Q_\delta l_\delta+Q_\delta R_\delta(\omega)||_{L^p}.
\end{equation*}
Note that
\begin{equation*}
||Q_\delta l_\delta+Q_\delta R_\delta(\omega)||_{L^p}\leq C(|| l_\delta||_{L^p}+ ||R_\delta(\omega)||_{L^p}).
\end{equation*}
It follows from \eqref{203}, the definition of $ l_\delta, R_\delta(\omega) $ and Lemma \ref{lemA-5} that
\begin{equation*}
\begin{split}
||l_\delta||_{L^p}=&||(V_{\delta,Z}-q)^p_+-\sum_{j=1}^m(V_{\delta,z_j,\hat{q}_{\delta,j}}-\hat{q}_{\delta,j})^p_+||_{L^p}\leq C\frac{\varepsilon^{\frac{2}{p}+\gamma}}{|\ln\varepsilon|^{p}},
\end{split}
\end{equation*}
and
\begin{equation*}
\begin{split}
||R_\delta(\omega)||_{L^p}=||(V_{\delta,Z}+\omega-q)^p_+-(V_{\delta,Z}-q)^p_+-p(V_{\delta,Z}-q)^{p-1}_+\omega||_{L^p}\leq C\frac{\varepsilon^{\frac{2}{p}}}{|\ln\varepsilon|^{p-2}}||\omega||_{L^\infty}^2.
\end{split}
\end{equation*}
Hence we get
\begin{equation}\label{217}
\begin{split}
||T_\delta(\omega)||_{L^\infty}\leq C\varepsilon^{-\frac{2}{p}}|\ln\varepsilon|^{p-1}\left( \frac{\varepsilon^{\frac{2}{p}+\gamma}}{|\ln\varepsilon|^{p}}+\frac{\varepsilon^{\frac{2}{p}}}{|\ln\varepsilon|^{p-2}}||\omega||_{L^\infty}^2\right)
\leq \frac{1}{|\ln\varepsilon|^{2-\theta_0}}.
\end{split}
\end{equation}
So $ T_\delta $ is a map from $ \mathcal{N}$ to $ \mathcal{N} $.

Then we prove that  $ T_\delta $ is a contraction map. For any $ \omega_1,\omega_2\in \mathcal{N} $,
\begin{equation*}
T_\delta(\omega_1)-T_\delta(\omega_2)=(Q_\delta L_\delta)^{-1}Q_\delta(R_\delta(\omega_1)-R_\delta(\omega_2)).
\end{equation*}
Note that $ R_\delta(\omega_1)=R_\delta(\omega_2)=0 $ in $ \Omega \backslash\cup_{i=1}^mB_{Ls_{\delta,i}}(z_{i}). $ By Lemma \ref{coercive esti} and the definition of $ \mathcal{N} $, for $ \delta $ sufficiently small
\begin{equation*}
\begin{split}
||T_\delta(\omega_1)-T_\delta(\omega_2)||_{L^\infty}\leq &C\varepsilon^{-\frac{2}{p}}|\ln\varepsilon|^{p-1}||R_\delta(\omega_1)-R_\delta(\omega_2)||_{L^p}\\
\leq &C\varepsilon^{-\frac{2}{p}}|\ln\varepsilon|^{p-1}\varepsilon^{\frac{2}{p}}\left(\frac{||\omega_1||_{L^\infty}+||\omega_2||_{L^\infty}}{|\ln\varepsilon|^{p-2}} \right) ||\omega_1-\omega_2||_{L^\infty}\\
\leq& \frac{1}{2}||\omega_1-\omega_2||_{L^\infty}.
\end{split}
\end{equation*}
So $ T_\delta $ is a contraction map.

To conclude,  $ T_\delta $ is a contraction map from $ \mathcal{N} $ to $ \mathcal{N} $ and thus there is a unique $ \omega_{\delta,Z}\in \mathcal{N} $ such that $ \omega_{\delta,Z}=T_\delta(\omega_{\delta,Z}) $. Moreover, by \eqref{217}  we have $ ||\omega_{\delta,Z}||_{L^\infty(\Omega)}=O\left( \frac{\varepsilon^\gamma}{|\ln\varepsilon|}\right).  $

\end{proof}

%\begin{remark}\label{rk1}
%Indeed, since $ K,q \in C^\infty$ and $ p>1 $,  we can also check that $ \omega_{\delta,Z} $ is a $ C^1 $ map about $ Z $, see \cite{CLW,CPY} for example.
%\end{remark}

The result of Proposition \ref{exist and uniq of w} implies that there exists a unique solution $ \omega_{\delta,Z}\in E_{\delta,Z} $ to \eqref{216}. % satisfying $ ||\omega_{\delta,Z}||_{L^\infty(\Omega)}=O\left(\frac{\varepsilon^\gamma}{|\ln\varepsilon|}\right).  $
This implies that for some $ b_{j,h}=b_{j,h}(Z) $
\begin{equation}\label{5-000}
L_\delta \omega_{\delta,Z}= l_\delta+  R_\delta(\omega_{\delta,Z})+\sum_{j=1}^m\sum_{h=1}^2b_{j,h}\left( -\delta^2\text{div}\left (K(z_j)\nabla   \frac{\partial V_{\delta, z_j, \hat{q}_{\delta,j}}}{\partial x_h}\right )\right),
\end{equation}
or equivalently
\begin{equation}\label{5-00}
-\delta^2\text{div}(K(x)\nabla (V_{\delta,Z}+\omega_{\delta,Z}))-(V_{\delta,Z}+\omega_{\delta,Z}-q)^p_+=\sum_{j=1}^m\sum_{h=1}^2b_{j,h}\left( -\delta^2\text{div}\left (K(z_j)\nabla   \frac{\partial V_{\delta, z_j, \hat{q}_{\delta,j}}}{\partial x_h}\right )\right).
\end{equation}
At the end of this section, we give some properties of the differentiability of $ \omega_{\delta,Z}  $ with respect to the variable $ Z $, which will be used in the next section. Using the similar method as that in \cite{CPY1,CPY,WYZ}, we  can estimate the $ L^\infty $ norm of  $ \frac{\partial \omega_{\delta,Z}}{\partial z_{i,h}} $ and show that $ \omega_{\delta,Z} $ is a $ C^1 $ map of $ Z $ in $ H^1_0 (\Omega) $.

\begin{proposition}\label{differ of w}
Let $ \omega_{\delta,Z} $ be the function obtained in Proposition \ref{exist and uniq of w}. Then $ \omega_{\delta,Z} $ is a $ C^1 $ map of $ Z $ in the
norm of $ H^1_0 (\Omega) $, and for any $ \gamma\in (0,1) $, $ l=1,\cdots,m $, $ \bar{h}=1,2 $
\begin{equation*}
\bigg|\bigg|\frac{\partial \omega_{\delta,Z}}{\partial z_{l,\bar{h}}}\bigg|\bigg|_{L^\infty(\Omega)}=O\left (\frac{1}{\varepsilon^{1-\gamma}|\ln\varepsilon|}\right ).
\end{equation*}

\end{proposition}

\begin{proof}
Note that from Lemma \ref{Green expansion2}, the regular part of Green's function $ \bar{S}_K(x,x)\in C^1(\Omega) $. Thus taking $ \frac{\partial}{\partial z_{l,\bar{h}}} $ in  \eqref{q_i choice}, we get
\begin{equation}\label{5-01}
\begin{split}
\frac{\partial \hat{q}_{\delta,i}}{\partial z_{l,\bar{h}}}=&\frac{\partial q}{\partial x_{\bar{h}}}(z_i)\delta_{i,l}+\sum_{j\neq l}\frac{2\pi\hat{q}_{\delta,j}\sqrt{\det K(z_j)}}{\ln s_{\delta,j}}\frac{\partial G_K(z_l,z_j)}{\partial z_{l,\bar{h}}}+o(1)\sum_{j=1}^m|\frac{\partial \hat{q}_{\delta,i}}{\partial z_{l,\bar{h}}}|+o(1)\\
=&O\left (|\ln\varepsilon|^{M}\right ),
\end{split}
\end{equation}
where we have used $ |\nabla_{z_l} G(z_l, z_j)|\leq C\frac{1}{|z_l-z_j|}\leq C|\ln\varepsilon|^M $ for $ Z\in \Lambda_{\varepsilon, m}. $ By the definition of $ V_{\delta, z_j, \hat{q}_{\delta,j}} $ and \eqref{5-01}, we get
\begin{equation}\label{5-02}
\bigg|\bigg|\frac{\partial V_{\delta, z_j, \hat{q}_{\delta,j}}}{\partial z_{l,\bar{h}}}\bigg|\bigg|_{L^\infty(\Omega)}=O\left (\frac{1}{\varepsilon|\ln\varepsilon|}\right )+O\left (|\ln\varepsilon|^{M}\right )=O\left (\frac{1}{\varepsilon|\ln\varepsilon|}\right ).
\end{equation}
Using the definition of $ H_{\delta, z_j, \hat{q}_{\delta,j}} $ in \eqref{eq3} and the $ L^p $-theory of elliptic equations, one has
\begin{equation}\label{5-03}
\bigg|\bigg|\frac{\partial H_{\delta, z_j, \hat{q}_{\delta,j}}}{\partial z_{l,\bar{h}}}\bigg|\bigg|_{W^{1,p}(\Omega)}\leq
\begin{cases}
\frac{C}{\varepsilon^{1-\frac{2}{p}}|\ln\varepsilon|},\ \ &p>2,\\
C,\ \ &p=2,\\
\frac{C}{|\ln\varepsilon|},\ \ &1\leq p<2.
\end{cases}
\end{equation}
Combining \eqref{5-02} and \eqref{5-03}, we have
\begin{equation}\label{5-04}
\bigg|\bigg|\frac{\partial V_{\delta, Z, j}}{\partial z_{l,\bar{h}}}\bigg|\bigg|_{L^\infty(\Omega)}=O\left (\frac{1}{\varepsilon|\ln\varepsilon|}\right ).
\end{equation}

Now we calculate the $ L^\infty $ norm of $ \frac{\partial \omega_{\delta,Z}}{\partial z_{l,\bar{h}}}. $ Note that from \eqref{5-00}, $ b_{j,h} $ is determined by
\begin{equation}\label{5-05}
\begin{split}
\sum_{j=1}^m\sum_{h=1}^2&b_{j,h}\int_\Omega\left( -\delta^2\text{div}\left (K(z_j)\nabla   \frac{\partial V_{\delta, z_j, \hat{q}_{\delta,j}}}{\partial x_h}\right )\right)\left( \eta_i\frac{\partial V_{\delta, Z,i}}{\partial x_\hbar}\right)\\
=&\int_\Omega\delta^2\left (K(x)\nabla(V_{\delta,Z}+\omega_{\delta,Z})|\nabla\left( \eta_i\frac{\partial V_{\delta, Z,i}}{\partial x_\hbar}\right)\right )-\int_\Omega(V_{\delta,Z}+\omega_{\delta,Z}-q)^p_+\left( \eta_i\frac{\partial V_{\delta, Z,i}}{\partial x_\hbar}\right)\\
=&\int_\Omega\delta^2\left (K(x)\nabla V_{\delta,Z}|\nabla\left( \eta_i\frac{\partial V_{\delta, Z,i}}{\partial x_\hbar}\right)\right )-\int_\Omega(V_{\delta,Z}+\omega_{\delta,Z}-q)^p_+\left( \eta_i\frac{\partial V_{\delta, Z,i}}{\partial x_\hbar}\right)\\
=&\int_\Omega\left( \sum_{j=1}^m(V_{\delta, z_j, \hat{q}_{\delta,j}}-\hat{q}_{\delta,j})^{p}_+-\left (V_{\delta,Z}+\omega_{\delta,Z}-q\right )^{p}_+\right) \left( \eta_i\frac{\partial V_{\delta, Z,i}}{\partial x_\hbar}\right),
\end{split}
\end{equation}
where we have used $ \omega_{\delta,Z}\in E_{\delta,Z} $. By Lemma \ref{lemA-5},
\begin{equation*}
\begin{split}
\int_\Omega&\left( \sum_{j=1}^m(V_{\delta, z_j, \hat{q}_{\delta,j}}-\hat{q}_{\delta,j})^{p}_+-\left (V_{\delta,Z}+\omega_{\delta,Z}-q\right )^{p}_+\right) \left( \eta_i\frac{\partial V_{\delta, Z,i}}{\partial x_\hbar}\right)\\
&=O\left (|\ln\varepsilon|^{-p+1}\left (\frac{\varepsilon^\gamma}{|\ln\varepsilon|}+||\omega_{\delta,Z}||_{L^\infty}\right )\int_{\cup_{j=1}^mB_{L\varepsilon}(z_j)}\bigg|\frac{\partial V_{\delta, Z,i}}{\partial x_\hbar} \bigg|\right )=O\left (\frac{\varepsilon^{1+\gamma}}{|\ln\varepsilon|^{p+1}}\right ).
\end{split}
\end{equation*}
Thus combining this with \eqref{coef of C} and \eqref{5-05}, we get
\begin{equation}\label{5-06}
b_{j,h}=O\left ( \varepsilon^{1+\gamma} \right ).
\end{equation}

Taking $ \frac{\partial}{\partial z_{l,\bar{h}}} $ in both sides of \eqref{5-05}, we obtain
\begin{equation}\label{5-07}
\begin{split}
\sum_{j=1}^m&\sum_{h=1}^2\frac{\partial b_{j,h}}{\partial z_{l,\bar{h}}}\int_\Omega\left( -\delta^2\text{div}\left (K(z_j)\nabla   \frac{\partial V_{\delta, z_j, \hat{q}_{\delta,j}}}{\partial x_h}\right )\right)\left( \eta_i\frac{\partial V_{\delta, Z,i}}{\partial x_\hbar}\right)\\
=&-\sum_{j=1}^m\sum_{h=1}^2b_{j,h} \frac{\partial }{\partial z_{l,\bar{h}}}\left\{  \int_\Omega\left( -\delta^2\text{div}\left (K(z_j)\nabla   \frac{\partial V_{\delta, z_j, \hat{q}_{\delta,j}}}{\partial x_h}\right )\right)\left( \eta_i\frac{\partial V_{\delta, Z,i}}{\partial x_\hbar}\right)\right\} \\
&+\frac{\partial }{\partial z_{l,\bar{h}}}\left\{\int_\Omega\left( \sum_{j=1}^m(V_{\delta, z_j, \hat{q}_{\delta,j}}-\hat{q}_{\delta,j})^{p}_+-\left (V_{\delta,Z}+\omega_{\delta,Z}-q\right )^{p}_+\right) \left( \eta_i\frac{\partial V_{\delta, Z,i}}{\partial x_\hbar}\right)\right\}.
\end{split}
\end{equation}
Note that from \eqref{5-02} and \eqref{5-03},
\begin{equation*}
 \frac{\partial }{\partial z_{l,\bar{h}}}\left\{  \int_\Omega\left( -\delta^2\text{div}\left (K(z_j)\nabla   \frac{\partial V_{\delta, z_j, \hat{q}_{\delta,j}}}{\partial x_h}\right )\right)\left( \eta_i\frac{\partial V_{\delta, Z,i}}{\partial x_\hbar}\right)\right\}=O\left( \frac{1}{\varepsilon|\ln\varepsilon|^{p+1}}\right)
\end{equation*}
and
\begin{equation*}
\begin{split}
 \int_\Omega&\left( \sum_{j=1}^m(V_{\delta, z_j, \hat{q}_{\delta,j}}-\hat{q}_{\delta,j})^{p}_+-\left (V_{\delta,Z}+\omega_{\delta,Z}-q\right )^{p}_+\right)  \frac{\partial }{\partial z_{l,\bar{h}}}\left( \eta_i\frac{\partial V_{\delta, Z,i}}{\partial x_\hbar}\right) \\
 &=O\left (\frac{1}{ |\ln\varepsilon|^{p-1}}\left (\frac{\varepsilon^\gamma}{|\ln\varepsilon|}+||\omega_{\delta,Z}||_{L^\infty}\right )\int_{\cup_{j=1}^mB_{L\varepsilon}(z_j)}\bigg|\frac{\partial^2 V_{\delta, Z,i}}{\partial z_{l,\bar{h}}\partial x_\hbar} \bigg|\right )=O\left (\frac{\varepsilon^{\gamma}}{|\ln\varepsilon|^{p+1}}\right ).
\end{split}
\end{equation*}
Inserting these into \eqref{5-07}, we obtain
\begin{equation}\label{5-08}
\begin{split}
 \frac{\partial b_{j,h}}{\partial z_{l,\bar{h}}}\cdot \frac{1}{|\ln\varepsilon|^{p+1}}=&O\left (\frac{ |b_{j,h}|}{\varepsilon|\ln\varepsilon|^{p+1}}+\frac{\varepsilon^{\gamma}}{|\ln\varepsilon|^{p+1}}\right )\\
 &+O\left(  \int_\Omega\frac{\partial }{\partial z_{l,\bar{h}}}\left( \sum_{j=1}^m(V_{\delta, z_j, \hat{q}_{\delta,j}}-\hat{q}_{\delta,j})^{p}_+-\left (V_{\delta,Z}+\omega_{\delta,Z}-q\right )^{p}_+\right) \left( \eta_i\frac{\partial V_{\delta, Z,i}}{\partial x_\hbar}\right) \right) .
\end{split}
\end{equation}
Using \eqref{203}, \eqref{5-01}, \eqref{5-02} and \eqref{5-03}, we have
\begin{equation*}\label{5-09}
\begin{split}
\frac{\partial }{\partial z_{l,\bar{h}}}&\left( \sum_{j=1}^m(V_{\delta, z_j, \hat{q}_{\delta,j}}-\hat{q}_{\delta,j})^{p}_+-\left (V_{\delta,Z}+\omega_{\delta,Z}-q\right )^{p}_+\right)\\
=& p\sum_{j=1}^m\left (V_{\delta, z_j, \hat{q}_{\delta,j}}-\hat{q}_{\delta,j}\right )^{p-1}_+\frac{\partial }{\partial z_{l,\bar{h}}}\left (V_{\delta, z_j, \hat{q}_{\delta,j}}-\hat{q}_{\delta,j}\right )-p\left (V_{\delta,Z}+\omega_{\delta,Z}-q\right )^{p-1}_+\frac{\partial }{\partial z_{l,\bar{h}}}\left(V_{\delta,Z}+\omega_{\delta,Z} \right)  \\
=&-p\left (V_{\delta,Z}+\omega_{\delta,Z}-q\right )^{p-1}_+\frac{\partial \omega_{\delta,Z} }{\partial z_{l,\bar{h}}}+p\sum_{j=1}^m\left( \left (V_{\delta, z_j, \hat{q}_{\delta,j}}-\hat{q}_{\delta,j}\right )^{p-1}_+-\left (V_{\delta,Z}+\omega_{\delta,Z}-q\right )^{p-1}_+\right)\frac{\partial V_{\delta, z_j, \hat{q}_{\delta,j}}}{\partial z_{l,\bar{h}}} \\
&-p\sum_{j=1}^m\left (V_{\delta,Z}+\omega_{\delta,Z}-q\right )^{p-1}_+\frac{\partial H_{\delta, z_j, \hat{q}_{\delta,j}}}{\partial z_{l,\bar{h}}}-p\sum_{j=1}^m\left (V_{\delta, z_j, \hat{q}_{\delta,j}}-\hat{q}_{\delta,j}\right )^{p-1}_+\frac{\partial \hat{q}_{\delta,j} }{\partial z_{l,\bar{h}}} \\
=&-p\left (V_{\delta,Z}+\omega_{\delta,Z}-q\right )^{p-1}_+\frac{\partial \omega_{\delta,Z} }{\partial z_{l,\bar{h}}}+O\left( \frac{1}{\varepsilon^{1-\gamma}|\ln\varepsilon|^p}\right),
\end{split}
\end{equation*}
from which we deduce,
\begin{equation}\label{5-10}
\begin{split}
\int_\Omega&\frac{\partial }{\partial z_{l,\bar{h}}}\left( \sum_{j=1}^m(V_{\delta, z_j, \hat{q}_{\delta,j}}-\hat{q}_{\delta,j})^{p}_+-\left (V_{\delta,Z}+\omega_{\delta,Z}-q\right )^{p}_+\right) \left( \eta_i\frac{\partial V_{\delta, Z,i}}{\partial x_\hbar}\right)\\
&=-p\int_\Omega \left (V_{\delta,Z}+\omega_{\delta,Z}-q\right )^{p-1}_+\frac{\partial \omega_{\delta,Z} }{\partial z_{l,\bar{h}}}\left( \eta_i\frac{\partial V_{\delta, Z,i}}{\partial x_\hbar}\right)+O\left( \frac{\varepsilon^\gamma}{ |\ln\varepsilon|^{p+1}}\right).
\end{split}
\end{equation}
Inserting \eqref{5-10} into \eqref{5-08}, we obtain
\begin{equation}\label{5-11}
\begin{split}
\frac{\partial b_{j,h}}{\partial z_{l,\bar{h}}}\cdot \frac{1}{|\ln\varepsilon|^{p+1}}=&O\left( \int_\Omega \left (V_{\delta,Z}+\omega_{\delta,Z}-q\right )^{p-1}_+\frac{\partial \omega_{\delta,Z} }{\partial z_{l,\bar{h}}}\left( \eta_i\frac{\partial V_{\delta, Z,i}}{\partial x_\hbar}\right)\right) +O\left( \frac{\varepsilon^\gamma}{ |\ln\varepsilon|^{p+1}}\right).
\end{split}
\end{equation}

Taking $ \frac{\partial}{\partial z_{l,\bar{h}}} $ in $ \int_{\Omega}\delta^2\left (K(x) \nabla\omega_{\delta,Z}| \nabla \left( \eta_i\frac{\partial V_{\delta, Z,i}}{\partial x_\hbar}\right)\right )=0$, one has
\begin{equation}\label{5-12}
\int_{\Omega}\delta^2 \left (K(x) \nabla \frac{\partial \omega_{\delta,Z} }{\partial z_{l,\bar{h}}}| \nabla \left( \eta_i\frac{\partial V_{\delta, Z,i}}{\partial x_\hbar}\right)\right )=-\int_{\Omega} \delta^2 \left (K(x) \nabla\omega_{\delta,Z}| \nabla \frac{\partial  }{\partial z_{l,\bar{h}}}\left( \eta_i\frac{\partial V_{\delta, Z,i}}{\partial x_\hbar}\right)\right ) .
\end{equation}
On the one hand, direct computation shows  that

\begin{equation}\label{5-13}
\begin{split}
-\int_{\Omega} \delta^2 \left (K(x) \nabla\omega_{\delta,Z}| \nabla \frac{\partial  }{\partial z_{l,\bar{h}}}\left( \eta_i\frac{\partial V_{\delta, Z,i}}{\partial x_\hbar}\right)\right )=\int_{\Omega} \delta^2 \text{div} \left (K(x) \nabla\omega_{\delta,Z}  \right )\frac{\partial  }{\partial z_{l,\bar{h}}}\left( \eta_i\frac{\partial V_{\delta, Z,i}}{\partial x_\hbar}\right).
\end{split}
\end{equation}
Note that $ ||l_\delta+R_\delta(\omega_{\delta,Z})||_{L^p}\leq C\frac{\varepsilon^{\frac{2}{p}+\gamma}}{|\ln\varepsilon|^{p}} $. By \eqref{5-000}  and \eqref{5-06},
\begin{equation*}
L_\delta \omega_{\delta,Z}= l_\delta+  R_\delta(\omega_{\delta,Z})+\sum_{j=1}^m\sum_{h=1}^2b_{j,h}\left( -\delta^2\text{div}\left (K(z_j)\nabla   \frac{\partial V_{\delta, z_j, \hat{q}_{\delta,j}}}{\partial x_h}\right )\right)=O\left( \frac{\varepsilon^{\frac{2}{p}+\gamma}}{|\ln\varepsilon|^{p}}\right) \  \ \ \text{in}\ L^p(\Omega),
\end{equation*}
which implies that
\begin{equation*}
\begin{split}
-\delta^2 \text{div}\left (K(x) \nabla\omega_{\delta,Z}  \right )=L_\delta \omega_{\delta,Z}+p\left( V_{\delta,Z}-q\right)^{p-1}_+\omega_{\delta,Z}= O\left( \frac{\varepsilon^{\frac{2}{p}+\gamma}}{|\ln\varepsilon|^{p}}\right) \  \ \ \text{in}\ L^p(\Omega).
\end{split}
\end{equation*}
Taking this into \eqref{5-13} and using \eqref{5-03}, one has
\begin{equation}\label{5-14}
\begin{split}
-\int_{\Omega} \delta^2 \left (K(x) \nabla\omega_{\delta,Z}| \nabla \frac{\partial  }{\partial z_{l,\bar{h}}}\left( \eta_i\frac{\partial V_{\delta, Z,i}}{\partial x_\hbar}\right)\right )=O\left( \frac{\varepsilon^\gamma}{ |\ln\varepsilon|^{p+1}}\right).
\end{split}
\end{equation}
On the other hand,
\begin{equation}\label{5-15}
\begin{split}
\int_{\Omega}\delta^2& \left (K(x) \nabla \frac{\partial \omega_{\delta,Z} }{\partial z_{l,\bar{h}}}| \nabla \left( \eta_i\frac{\partial V_{\delta, Z,i}}{\partial x_\hbar}\right)\right )=\int_{\Omega}\frac{\partial \omega_{\delta,Z} }{\partial z_{l,\bar{h}}}-\delta^2  \text{div}\left (K(x) \nabla \left( \eta_i\frac{\partial V_{\delta, Z,i}}{\partial x_\hbar}\right)\right )\\
=&\int_{\Omega}\frac{\partial \omega_{\delta,Z} }{\partial z_{l,\bar{h}}}\bigg\{\eta_ip\left (V_{\delta, z_i, \hat{q}_{\delta,i}}-\hat{q}_{\delta,i}\right )^{p-1}_+\frac{\partial V_{\delta, z_i, \hat{q}_{\delta,i}}}{\partial x_\hbar}+\eta_i\delta^2\text{div}\left (\frac{\partial K(x)}{\partial x_\hbar}\nabla V_{\delta, Z,i}  \right ) \\
& -2 \delta^2\left (K(x)\nabla\eta_i|\nabla   \frac{\partial V_{\delta, Z,i}}{\partial x_\hbar}  \right ) -\delta^2\text{div}\left (K(x)\nabla   \eta_i  \right )\frac{\partial V_{\delta, Z,i}}{\partial x_\hbar} \bigg\}\\
=& p\int_{\Omega}\frac{\partial \omega_{\delta,Z} }{\partial z_{l,\bar{h}}} \left (V_{\delta, z_i, \hat{q}_{\delta,i}}-\hat{q}_{\delta,i}\right )^{p-1}_+\frac{\partial V_{\delta, z_i, \hat{q}_{\delta,i}}}{\partial x_\hbar}+O\left( \delta^2|\ln\varepsilon|^{2M+1}\bigg|\bigg|\frac{\partial \omega_{\delta,Z} }{\partial z_{l,\bar{h}}}\bigg|\bigg|_{L^\infty}\right).
\end{split}
\end{equation}
Combining \eqref{5-12} with \eqref{5-14} and \eqref{5-15}, we obtain
\begin{equation}\label{5-16}
\begin{split}
\int_{\Omega}\frac{\partial \omega_{\delta,Z} }{\partial z_{l,\bar{h}}} \left (V_{\delta, z_i, \hat{q}_{\delta,i}}-\hat{q}_{\delta,i}\right )^{p-1}_+\frac{\partial V_{\delta, z_i, \hat{q}_{\delta,i}}}{\partial x_\hbar}=O\left( \frac{\varepsilon^\gamma}{ |\ln\varepsilon|^{p+1}}+\delta^2|\ln\varepsilon|^{2M+1}\bigg|\bigg|\frac{\partial \omega_{\delta,Z} }{\partial z_{l,\bar{h}}}\bigg|\bigg|_{L^\infty}\right).
\end{split}
\end{equation}
%\begin{equation*}
%\begin{split}
%-\int_{\Omega}&\omega_{\delta,Z}\frac{\partial  }{\partial z_{l,\bar{h}}}\left( -\delta^2\nabla\cdot\left (K(x)  \nabla \left( \eta_i\frac{\partial V_{\delta, Z,i}}{\partial x_\hbar}\right)\right ) \right)\\
%=&-\int_{\Omega} \omega_{\delta,Z}\frac{\partial  }{\partial z_{l,\bar{h}}}\bigg\{ \eta_ip\left (V_{\delta, z_j, \hat{q}_{\delta,j}}-\hat{q}_{\delta,j}\right )^{p-1}_+\frac{\partial V_{\delta, z_j, \hat{q}_{\delta,j}}}{\partial x_h}+\eta_i\delta^2\nabla\cdot\left (\frac{\partial K(x)}{\partial x_\hbar}  \nabla   V_{\delta, Z,i} \right )\\
%& -\delta^2\left( K(x)\nabla\eta_i|\nabla \frac{\partial V_{\delta, Z,i}}{\partial x_\hbar}\right) \bigg\}
%\end{split}
%\end{equation*}
Taking \eqref{5-16} into \eqref{5-11} and using \eqref{203} and Proposition \ref{exist and uniq of w}, we conclude that
\begin{equation}\label{5-17}
\begin{split}
\frac{\partial b_{j,h}}{\partial z_{l,\bar{h}}}\cdot \frac{1}{|\ln\varepsilon|^{p+1}}=& O\left( \frac{\varepsilon^\gamma}{ |\ln\varepsilon|^{p+1}}+ \frac{\varepsilon^{1+\gamma}}{|\ln\varepsilon|^p}\bigg|\bigg|\frac{\partial \omega_{\delta,Z} }{\partial z_{l,\bar{h}}}\bigg|\bigg|_{L^\infty}\right).
\end{split}
\end{equation}

Now taking $ \frac{\partial}{\partial z_{l,\bar{h}}} $ in both sides of \eqref{5-00}, we get
%\begin{equation}
%\begin{split}
%L_\delta \omega_{\delta,Z}= l_\delta+  R_\delta(\omega_{\delta,Z})+\sum_{j=1}^m\sum_{h=1}^2b_{j,h}\left( -\delta^2\text{div}\left (K(z_j)\nabla   \frac{\partial V_{\delta, z_j, \hat{q}_{\delta,j}}}{\partial x_h}\right )\right),
%\end{split}
%\end{equation}
\begin{equation}\label{5-18}
\begin{split}
-\delta^2\text{div}&\left( K(x)\nabla \frac{\partial\omega_{\delta,Z}}{\partial z_{l,\bar{h}}}\right)-p\left( V_{\delta,Z}+\omega_{\delta,Z}-q\right)^{p-1}_+\frac{\partial\omega_{\delta,Z}}{\partial z_{l,\bar{h}}}\\
=&\delta^2\text{div}\left( K(x)\nabla \frac{\partial V_{\delta,Z}}{\partial z_{l,\bar{h}}}\right)+p\left( V_{\delta,Z}+\omega_{\delta,Z}-q\right)^{p-1}_+\frac{\partial V_{\delta,Z}}{\partial z_{l,\bar{h}}}\\
% L_\delta \frac{\partial\omega_{\delta,Z}}{\partial z_{l,\bar{h}}}
%=& p(p-1)\left( V_{\delta,Z}-q\right)^{p-2}_+\frac{\partial V_{\delta,Z}}{\partial z_{l,\bar{h}}}\omega_{\delta,Z}+\frac{\partial l_\delta}{\partial z_{l,\bar{h}}}+  \frac{\partial R_\delta(\omega_{\delta,Z})}{\partial z_{l,\bar{h}}}\\
&+\sum_{j=1}^m\sum_{h=1}^2\frac{\partial b_{j,h}}{\partial z_{l,\bar{h}}}p\left(   V_{\delta, z_j, \hat{q}_{\delta,j}}-\hat{q}_{\delta,j} \right )^{p-1}_+\frac{\partial V_{\delta, z_j, \hat{q}_{\delta,j}}}{\partial x_h}\\
&+\sum_{j=1}^m\sum_{h=1}^2b_{j,h}\frac{\partial }{\partial z_{l,\bar{h}}}\left( p\left(   V_{\delta, z_j, \hat{q}_{\delta,j}}-\hat{q}_{\delta,j} \right )^{p-1}_+\frac{\partial V_{\delta, z_j, \hat{q}_{\delta,j}}}{\partial x_h}\right) .
\end{split}
\end{equation}
Note that the function $ \frac{\partial\omega_{\delta,Z}}{\partial z_{l,\bar{h}}} $ may not be in $ E_{\delta,Z} $. We make the following decomposition:
\begin{equation}\label{5-19}
\frac{\partial\omega_{\delta,Z}}{\partial z_{l,\bar{h}}}=\omega_{\delta}^*+\sum_{j=1}^m\sum_{h=1}^2C_{j,h}\zeta_j\frac{\partial V_{\delta, z_j, \hat{q}_{\delta,j}}}{\partial x_h},
\end{equation}
where $ \omega_{\delta}^*\in E_{\delta,Z} $,  $\zeta_j(x):=\eta\left (\frac{|T_{z_j}\left( x-z_j\right) |}{s_{\delta,j}}\right )$ and $ C_{j,h} $ is determined by
\begin{equation*}
\begin{split}
\sum_{j=1}^m\sum_{h=1}^2&C_{j,h}\int_{\Omega}\zeta_j\frac{\partial V_{\delta, z_j, \hat{q}_{\delta,j}}}{\partial x_h}-\delta^2\text{div}\left( K(x)\nabla\left(\eta_i\frac{\partial V_{\delta,Z,i}}{\partial x_\hbar} \right) \right)\\
=& \int_{\Omega}\frac{\partial\omega_{\delta,Z}}{\partial z_{l,\bar{h}}}-\delta^2\text{div}\left( K(x)\nabla\left(\eta_i\frac{\partial V_{\delta,Z,i}}{\partial x_\hbar} \right)\right) \ \ \ \ i=1,\cdots,m,\ \hbar=1,2.
\end{split}
\end{equation*}
Direct computation shows that
\begin{equation*}
\int_{\Omega}\zeta_j\frac{\partial V_{\delta, z_j, \hat{q}_{\delta,j}}}{\partial x_h}-\delta^2\text{div}\left( K(x)\nabla\left(\eta_i\frac{\partial V_{\delta,Z,i}}{\partial x_\hbar} \right) \right)=\left( (\tilde{M}_i)_{h,\hbar}\delta_{i,j}+o(1)\right) \frac{1}{|\ln\varepsilon|^{p+1}},
\end{equation*}
where $ \tilde{M}_{i} $ are $ m $ positive definite matrices. Combining this with \eqref{5-12} and \eqref{5-14}, we obtain
\begin{equation}\label{5-20}
C_{j,h}=O\left( \frac{\varepsilon^\gamma}{ |\ln\varepsilon|^{p+1}}\right)\cdot|\ln\varepsilon|^{p+1}=O\left( \varepsilon^\gamma \right).
\end{equation}
Inserting \eqref{5-19} in \eqref{5-18}, we get
\begin{equation}\label{5-21}
\begin{split}
-\delta^2\text{div}&\left( K(x)\nabla  \omega_{\delta}^*\right)-p\left( V_{\delta,Z} -q\right)^{p-1}_+\omega_{\delta}^*\\
=&\sum_{j=1}^m\sum_{h=1}^2C_{j,h} \delta^2\text{div}\left( K(x)\nabla \left( \zeta_j\frac{\partial V_{\delta, z_j, \hat{q}_{\delta,j}}}{\partial x_h}\right) \right)+ \sum_{j=1}^m\sum_{h=1}^2C_{j,h}p\left( V_{\delta,Z} -q\right)^{p-1}_+\zeta_j\frac{\partial V_{\delta, z_j, \hat{q}_{\delta,j}}}{\partial x_h}\\
&+p\left( \left( V_{\delta,Z}+\omega_{\delta,Z}-q\right)^{p-1}_+-\left( V_{\delta,Z} -q\right)^{p-1}_+\right) \frac{\partial\omega_{\delta,Z}}{\partial z_{l,\bar{h}}}\\
&-p\sum_{j=1}^m\left (V_{\delta, z_j, \hat{q}_{\delta,j}}-\hat{q}_{\delta,j}\right )^{p-1}_+\frac{\partial }{\partial z_{l,\bar{h}}}\left (V_{\delta, z_j, \hat{q}_{\delta,j}}-\hat{q}_{\delta,j}\right )+p\left (V_{\delta,Z}+\omega_{\delta,Z}-q\right )^{p-1}_+\frac{\partial V_{\delta,Z}}{\partial z_{l,\bar{h}}}   \\
% L_\delta \frac{\partial\omega_{\delta,Z}}{\partial z_{l,\bar{h}}}
%=& p(p-1)\left( V_{\delta,Z}-q\right)^{p-2}_+\frac{\partial V_{\delta,Z}}{\partial z_{l,\bar{h}}}\omega_{\delta,Z}+\frac{\partial l_\delta}{\partial z_{l,\bar{h}}}+  \frac{\partial R_\delta(\omega_{\delta,Z})}{\partial z_{l,\bar{h}}}\\
&+\sum_{j=1}^m\sum_{h=1}^2\frac{\partial b_{j,h}}{\partial z_{l,\bar{h}}}p\left(   V_{\delta, z_j, \hat{q}_{\delta,j}}-\hat{q}_{\delta,j} \right )^{p-1}_+\frac{\partial V_{\delta, z_j, \hat{q}_{\delta,j}}}{\partial x_h}\\
&+\sum_{j=1}^m\sum_{h=1}^2b_{j,h}\frac{\partial }{\partial z_{l,\bar{h}}}\left( p\left(   V_{\delta, z_j, \hat{q}_{\delta,j}}-\hat{q}_{\delta,j} \right )^{p-1}_+\frac{\partial V_{\delta, z_j, \hat{q}_{\delta,j}}}{\partial x_h}\right) .
\end{split}
\end{equation}
By \eqref{5-01}, \eqref{5-06}, \eqref{5-17} and \eqref{5-20}, one computes directly that
\begin{equation*}
\begin{split}
\sum_{j=1}^m&\sum_{h=1}^2C_{j,h} \delta^2\text{div}\left( K(x)\nabla \left( \zeta_j\frac{\partial V_{\delta, z_j, \hat{q}_{\delta,j}}}{\partial x_h}\right) \right)+ \sum_{j=1}^m\sum_{h=1}^2C_{j,h}p\left( V_{\delta,Z} -q\right)^{p-1}_+\zeta_j\frac{\partial V_{\delta, z_j, \hat{q}_{\delta,j}}}{\partial x_h} \\
=&\left( |C_{j,h}|\frac{\varepsilon^{\frac{2}{p}-1}}{|\ln\varepsilon|^p}\right)=O\left(  \frac{\varepsilon^{\gamma+\frac{2}{p}-1}}{|\ln\varepsilon|^p}\right)\ \ \text{in}\ L^p(\Omega),
\end{split}
\end{equation*}
\begin{equation*}
\begin{split}
p\left( \left( V_{\delta,Z}+\omega_{\delta,Z}-q\right)^{p-1}_+-\left( V_{\delta,Z} -q\right)^{p-1}_+\right) \frac{\partial\omega_{\delta,Z}}{\partial z_{l,\bar{h}}}=O\left(   \frac{\varepsilon^{\gamma+\frac{2}{p}}}{|\ln\varepsilon|^{p-1}}\bigg|\bigg|\frac{\partial \omega_{\delta,Z} }{\partial z_{l,\bar{h}}}\bigg|\bigg|_{L^\infty}\right)\ \ \text{in}\ L^p(\Omega),
\end{split}
\end{equation*}
\begin{equation*}
\begin{split}
-p&\sum_{j=1}^m\left (V_{\delta, z_j, \hat{q}_{\delta,j}}-\hat{q}_{\delta,j}\right )^{p-1}_+\frac{\partial }{\partial z_{l,\bar{h}}}\left (V_{\delta, z_j, \hat{q}_{\delta,j}}-\hat{q}_{\delta,j}\right )+p\left (V_{\delta,Z}+\omega_{\delta,Z}-q\right )^{p-1}_+\frac{\partial V_{\delta,Z}}{\partial z_{l,\bar{h}}}\\
=&O\left( \frac{\varepsilon^{\frac{2}{p}}}{|\ln\varepsilon|^{p-1}}|\ln\varepsilon|^{M+1}+ \frac{\varepsilon^{\gamma+\frac{2}{p}-1}}{|\ln\varepsilon|^p}\right) =O\left(  \frac{\varepsilon^{\gamma+\frac{2}{p}-1}}{|\ln\varepsilon|^p}\right)\ \ \text{in}\ L^p(\Omega),
\end{split}
\end{equation*}
\begin{equation*}
\begin{split}
&\sum_{j=1}^m\sum_{h=1}^2\frac{\partial b_{j,h}}{\partial z_{l,\bar{h}}}p\left(   V_{\delta, z_j, \hat{q}_{\delta,j}}-\hat{q}_{\delta,j} \right )^{p-1}_+\frac{\partial V_{\delta, z_j, \hat{q}_{\delta,j}}}{\partial x_h}\\
&=O\left(   \frac{\varepsilon^{ \frac{2}{p}}}{\varepsilon|\ln\varepsilon|^{p}}\bigg|\frac{\partial b_{j,h}}{\partial z_{l,\bar{h}}}\bigg|\right)=O\left(   \frac{\varepsilon^{ \frac{2}{p}-1}}{|\ln\varepsilon|^{p}}\left(   \varepsilon^\gamma +  \varepsilon^{1+\gamma} |\ln\varepsilon| \bigg|\bigg|\frac{\partial \omega_{\delta,Z} }{\partial z_{l,\bar{h}}}\bigg|\bigg|_{L^\infty}\right) \right)\ \ \text{in}\ L^p(\Omega),
\end{split}
\end{equation*}
\begin{equation*}
\begin{split}
&\sum_{j=1}^m\sum_{h=1}^2b_{j,h}\frac{\partial }{\partial z_{l,\bar{h}}}\left( p\left(   V_{\delta, z_j, \hat{q}_{\delta,j}}-\hat{q}_{\delta,j} \right )^{p-1}_+\frac{\partial V_{\delta, z_j, \hat{q}_{\delta,j}}}{\partial x_h}\right)
=O\left(  \frac{\varepsilon^{\gamma+\frac{2}{p}-1}}{|\ln\varepsilon|^p}\right)\ \ \text{in}\ L^p(\Omega).
\end{split}
\end{equation*}
Combining these with Lemma \ref{coercive esti} and \eqref{5-21}, we are led to
\begin{equation}\label{5-22}
||\omega_{\delta}^*||_{L^\infty }\leq C\varepsilon^{-\frac{2}{p}}|\ln\varepsilon|^{p-1} \left(\frac{\varepsilon^{\gamma+\frac{2}{p}-1}}{|\ln\varepsilon|^p}+ \frac{\varepsilon^{\gamma+\frac{2}{p}}}{|\ln\varepsilon|^{p-1}}\bigg|\bigg|\frac{\partial \omega_{\delta,Z} }{\partial z_{l,\bar{h}}}\bigg|\bigg|_{L^\infty} \right).
\end{equation}
From the decomposition \eqref{5-19}, we have
\begin{equation*}
\bigg|\bigg|\frac{\partial \omega_{\delta,Z} }{\partial z_{l,\bar{h}}}\bigg|\bigg|_{L^\infty}\leq ||\omega_{\delta}^*||_{L^\infty }+O\left( \frac{1}{\varepsilon^{1-\gamma}|\ln\varepsilon|}\right) .
\end{equation*}
Taking this in \eqref{5-22}, we obtain
\begin{equation*}
||\omega_{\delta}^*||_{L^\infty }\leq C  \frac{1}{\varepsilon^{1-\gamma}|\ln\varepsilon|},
\end{equation*}
from which we deduce, $ \big|\big|\frac{\partial \omega_{\delta,Z} }{\partial z_{l,\bar{h}}}\big|\big|_{L^\infty(\Omega)}=O\left( \frac{1}{\varepsilon^{1-\gamma}|\ln\varepsilon|} \right)  $.

Finally, we prove that $ \omega_{\delta,Z} $ is a $ C^1 $ map of $ Z\in \Lambda_{\varepsilon, m} $ in $ H^1(\Omega) $. To prove the continuity of $ \omega_{\delta,Z} $ of $ Z $, let $ Z_j\to Z_0 $. By Proposition \ref{exist and uniq of w}, $ \omega_{\delta,Z_j}  $ is uniformly bounded in $ L^\infty(\Omega) $. Thus  using \eqref{5-000} and \eqref{5-06}, we conclude that $ ||\omega_{\delta,Z_j}||_{H^1_0(\Omega)} $ is bounded by a constant $ C $  which is independent of $ j $. Then there is a subsequence (still denoted by $ Z_j $) such that
\begin{equation*}
\omega_{\delta,Z_j}\rightarrow \omega^{**}\ \ \ \  \text{weakly in } H^1_0(\Omega)
\end{equation*}
and
\begin{equation*}
\omega_{\delta,Z_j}\rightarrow \omega^{**}\ \ \ \  \text{strongly in } L^2(\Omega).
\end{equation*}
Using the equation again, we can get that
\begin{equation*}
\omega_{\delta,Z_j}\rightarrow \omega^{**}\ \ \ \  \text{strongly in } H^1_0(\Omega),
\end{equation*}
from which  we deduce that $ \omega^{**}\in E_{\delta,Z_0} $ and $ \omega^{**} $ satisfies \eqref{216} with $ Z_j $ replaced by $ Z_0 $.  By the uniqueness, we get $ \omega^{**}=\omega_{\delta,Z_0} $ and hence  $ \omega_{\delta,Z} $ is continuous in $Z$ in the norm of $ H^1_0(\Omega) $. Moreover, using similar method as  Proposition 3.7 in  \cite{CPY1}, we can get $ \frac{\partial \omega_{\delta,Z}}{\partial z_{l,h}} $ is continuous of $ Z $ in $ H^1(\Omega) $. The proof is thus complete.
\end{proof}

\section{Finite-dimensional energy expansion}

In view of Proposition \ref{exist and uniq of w}, given any $ \delta $ small and $ Z\in \Lambda_{\varepsilon, m} $, there exists a unique $ \omega_{\delta,Z}\in E_{\delta,Z} $ satisfying $ Q_\delta L_\delta \omega_{\delta,Z}=Q_\delta l_\delta+Q_\delta R_\delta(\omega_{\delta,Z}), $
i.e., for some $ b_{j,h}=b_{j,h}(Z) $
\begin{equation*}
 L_\delta \omega_{\delta,Z}= l_\delta+  R_\delta(\omega_{\delta,Z})+\sum_{j=1}^m\sum_{h=1}^2b_{j,h}\left( -\delta^2\text{div}\left (K(z_j)\nabla   \frac{\partial V_{\delta, z_j, \hat{q}_{\delta,j}}}{\partial x_h}\right )\right).
\end{equation*}
Thus, it suffices to find $Z$ solving the following finite dimensional problem
\begin{equation*}
 b_{j,h}(Z)=0,\ \ \forall\ j=1,\cdots,m,\ h=1,2.
\end{equation*}

%In the following we find proper $ Z=Z(\delta) $ such that $ V_{\delta,Z}+\omega_{\delta,Z} $ is a solution of \eqref{111}. Note that
Define
\begin{equation}\label{functional I}
I_\delta(u)=\frac{\delta^2}{2}\int_{\Omega}\left( K(x)\nabla u|\nabla u\right)-\frac{1}{p+1}\int_{\Omega}(u-q)^{p+1}_+,
\end{equation}
and
\begin{equation*}
K_\delta(Z)=I_\delta(V_{\delta,Z}+\omega_{\delta,Z}).
\end{equation*}
%Then by Remark \ref{rk1}, we know that $ P_\delta(Z) $ is a $ C^1 $ function.
%By the regularity of $ K $ and $ q $, it is not hard to check that if $ Z $ is a critical point of $ P_\delta(Z)  $, then $ V_{\delta,Z}+\omega_{\delta,Z} $ is a critical point of $ I_\delta $, i.e., a solution of \eqref{111}. We now give estimates of  $ P_\delta(Z) $.
It follows from Proposition \ref{differ of w} that $ K_\delta(Z) $ is a $ C^1 $ function of $ Z $.

The following lemma shows that, to find solutions of \eqref{111}, it suffices to prove the existence of critical points of  $ K_\delta(Z) $.
\begin{lemma}\label{choice of Z}
If $ Z\in \Lambda_{\varepsilon, m}$ is a critical point of $ K_\delta(Z) $, then
$ V_{\delta,Z}+\omega_{\delta,Z} $  is a solution to \eqref{111}.
\end{lemma}
\begin{proof}
It follows from Proposition \ref{exist and uniq of w} that
\begin{equation}\label{611}
\begin{split}
\langle I'(V_{\delta,Z}+\omega_{\delta,Z}), \phi\rangle=\sum_{j=1}^m\sum_{h=1}^2b_{j,h}\int_{\Omega} -\delta^2\text{div}\left (K(z_j)\nabla   \frac{\partial V_{\delta, z_j, \hat{q}_{\delta,j}}}{\partial x_h}\right )\phi,\ \ \forall \phi\in H^1_0(\Omega).
\end{split}
\end{equation}
We only need to choose $ Z $, such that the corresponding constants $ b_{j,h} $  are all zero.
Suppose that $ Z $ is a critical point of $ K_\delta(Z) $. Then from \eqref{611} and Proposition \ref{differ of w}, for $ i=1,\cdots,m, \hbar=1,2 $
\begin{equation}\label{603}
\begin{split}
0=&\frac{\partial K_\delta(Z)}{\partial z_{i,\hbar}}=\left \langle I'(V_{\delta,Z}+\omega_{\delta,Z}), \frac{\partial (V_{\delta,Z}+\omega_{\delta,Z})}{\partial z_{i,\hbar}}\right \rangle\\
=&\sum_{j=1}^m\sum_{h=1}^2b_{j,h}\int_{\Omega} -\delta^2\text{div}\left (K(z_j)\nabla   \frac{\partial V_{\delta, z_j, \hat{q}_{\delta,j}}}{\partial x_h}\right )\frac{\partial (V_{\delta,Z}+\omega_{\delta,Z})}{\partial z_{i,\hbar}}\\
=&\sum_{j=1}^m\sum_{h=1}^2b_{j,h}((M_i)_{h,\hbar}\delta_{i,j}+o(\varepsilon^\gamma))\frac{1}{|\ln\varepsilon|^{p+1}}+O\left (\frac{\varepsilon}{|\ln\varepsilon|^p}\sum_{j=1}^m\sum_{h=1}^2b_{j,h}\bigg|\bigg|\frac{\partial \omega_{\delta,Z}}{\partial z_{j,h}}\bigg|\bigg|_{L^\infty}\right )\\
=&\sum_{j=1}^m\sum_{h=1}^2b_{j,h}((M_i)_{h,\hbar}\delta_{i,j}+o(\varepsilon^\gamma))\frac{1}{|\ln\varepsilon|^{p+1}}+O\left (\frac{\varepsilon^\gamma}{|\ln\varepsilon|^{p+1}}\sum_{j=1}^m\sum_{h=1}^2b_{j,h}\right ),
\end{split}
\end{equation}
from which we deduce that  $ b_{j,h}(Z)=0. $

\end{proof}

Now we give the energy expansion of the functional $ K_\delta(Z) $. We prove
the following result.
\begin{proposition}\label{pro401}
There holds
\begin{equation*}
K_{\delta}(Z)=I_\delta(V_{\delta,Z})+O\left( \frac{\varepsilon^{2+\gamma}}{|\ln\varepsilon|^{p+1}}\right).
\end{equation*}
\end{proposition}

\begin{proof}
Note  that
\begin{equation*}
\begin{split}
K_{\delta}(Z)=&I_\delta(V_{\delta,Z})+\delta^2\int_{\Omega}\left( K(x)\nabla V_{\delta,Z}|\nabla \omega_{\delta,Z}\right) +\frac{\delta^2}{2}\int_{\Omega}\left( K(x)\nabla \omega_{\delta,Z}|\nabla \omega_{\delta,Z}\right)\\
&-\frac{1}{p+1}\left( \int_{\Omega}(V_{\delta,Z}+\omega_{\delta,Z}-q)^{p+1}_+-\int_{\Omega}(V_{\delta,Z}-q)^{p+1}_+\right).
\end{split}
\end{equation*}
It follows from Proposition \ref{exist and uniq of w} that
\begin{equation*}
\begin{split}
\int_{\Omega}&(V_{\delta,Z}+\omega_{\delta,Z}-q)^{p+1}_+-\int_{\Omega}(V_{\delta,Z}-q)^{p+1}_+\\
&=(p+1)\sum_{j=1}^m\int_{B_{Ls_{\delta,j}}(z_j)}(V_{\delta,Z}-q)^{p}_+\omega_{\delta,Z}
+O\left( \sum_{j=1}^m\int_{B_{Ls_{\delta,j}}(z_j)}(V_{\delta,Z}-q)^{p-1}_+\omega_{\delta,Z}^2\right)\\
%&=O\left( \sum_{j=1}^m\frac{s_{\delta,j}^2||\omega_{\delta,Z}||_{L^\infty}}{|\ln\varepsilon|^p}\right) +O\left( \sum_{j=1}^m\frac{s_{\delta,j}^2||\omega_{\delta,Z}||_{L^\infty}^2}{|\ln\varepsilon|^{p-1}}\right) \\
&=O\left( \frac{\varepsilon^{2+\gamma}}{|\ln\varepsilon|^{p+1}}\right).
\end{split}
\end{equation*}
Since $  -\delta^2\text{div}(K(x)\nabla V_{\delta, Z,j})=(V_{\delta,z_j,\hat{q}_{\delta,j}}-\hat{q}_{\delta,j})^p_+, $
we get
\begin{equation*}
\begin{split}
\delta^2\int_{\Omega}\left( K(x)\nabla V_{\delta,Z}|\nabla \omega_{\delta,Z}\right)
=&\sum_{j=1}^m\int_{B_{Ls_{\delta,j}}(z_j)}(V_{\delta,z_j,\hat{q}_{\delta,j}}-\hat{q}_{\delta,j})^p_+\omega_{\delta,Z}
%=&O\left( \sum_{j=1}^m\frac{s_{\delta,j}^2||\omega_{\delta,Z}||_{L^\infty}}{|\ln\varepsilon|^p}\right)\\
=O\left( \frac{\varepsilon^{2+\gamma}}{|\ln\varepsilon|^{p+1}}\right).
\end{split}
\end{equation*}
As for $ \frac{\delta^2}{2}\int_{\Omega}\left( K(x)\nabla \omega_{\delta,Z}|\nabla \omega_{\delta,Z}\right), $ since  $ \omega_{\delta,Z}\in E_{\delta,Z} $, we get $ -\delta^2\text{div}(K(x)\nabla \omega_{\delta,Z})\in F_{\delta,Z} $. So
\begin{equation*}
\begin{split}
Q_\delta L_\delta \omega_{\delta,Z}%=&Q_\delta(-\delta^2\text{div}(K(x)\nabla \omega_{\delta,Z}))-Q_\delta(p(V_{\delta,Z}-q)^{p-1}_+\omega_{\delta,Z})\\
=&-\delta^2\text{div}(K(x)\nabla \omega_{\delta,Z})-Q_\delta(p(V_{\delta,Z}-q)^{p-1}_+\omega_{\delta,Z}),
\end{split}
\end{equation*}
which combined with $ Q_\delta L_\delta \omega_{\delta,Z}=Q_\delta l_\delta+Q_\delta R_\delta(\omega_{\delta,Z})  $ yields
\begin{equation*}
-\delta^2\text{div}(K(x)\nabla \omega_{\delta,Z})=Q_\delta(p(V_{\delta,Z}-q)^{p-1}_+\omega_{\delta,Z})+Q_\delta l_\delta+Q_\delta R_\delta(\omega_{\delta,Z}).
\end{equation*}
Hence by Lemma \ref{lemA-5} and Proposition \ref{exist and uniq of w}, one has
\begin{equation*}
\begin{split}
\delta^2\int_{\Omega}&\left( K(x)\nabla \omega_{\delta,Z}|\nabla \omega_{\delta,Z}\right)\\
=&\int_{\Omega}Q_\delta(p(V_{\delta,Z}-q)^{p-1}_+\omega_{\delta,Z})\omega_{\delta,Z}+\int_{\Omega}Q_\delta l_\delta\omega_{\delta,Z}+\int_{\Omega}Q_\delta R_\delta(\omega_{\delta,Z})\omega_{\delta,Z}\\
%\leq &||Q_\delta(p(V_{\delta,Z}-q)^{p-1}_+\omega_{\delta,Z})||_{L^1}||\omega_{\delta,Z}||_{L^\infty}+||Q_\delta l_\delta||_{L^1}||\omega_{\delta,Z}||_{L^\infty}+||Q_\delta R_\delta(\omega_{\delta,Z})||_{L^1}||\omega_{\delta,Z}||_{L^\infty}\\
=&O\left (\left( ||(p(V_{\delta,Z}-q)^{p-1}_+\omega_{\delta,Z})||_{L^1}+|| l_\delta||_{L^1}+|| R_\delta(\omega_{\delta,Z})||_{L^1}||\right) \omega_{\delta,Z}||_{L^\infty}\right )
%= &O\left( \frac{\varepsilon^2||\omega_{\delta,Z}||_{L^\infty}^2}{|\ln\varepsilon|^{p-1}}\right) +O\left( \frac{\varepsilon^{2+\gamma}||\omega_{\delta,Z}||_{L^\infty}}{|\ln\varepsilon|^{p}}\right) +O\left( \frac{\varepsilon^2||\omega_{\delta,Z}||_{L^\infty}^3}{|\ln\varepsilon|^{p-2}}\right) \\
=O\left( \frac{\varepsilon^{2+\gamma}}{|\ln\varepsilon|^{p+1}}\right).
\end{split}
\end{equation*}
To conclude, we get $ K_{\delta}(Z)=I_\delta(V_{\delta,Z})+O\left( \frac{\varepsilon^{2+\gamma}}{|\ln\varepsilon|^{p+1}}\right). $

\end{proof}

\begin{proposition}\label{order of main term}
There holds
\begin{equation}\label{344}
\begin{split}
I_\delta(V_{\delta,Z})=&\sum_{j=1}^m\frac{\pi\delta^2}{\ln\frac{1}{\varepsilon}}q(z_j)^2\sqrt{\det K(z_j)}+\sum_{j=1}^m\frac{(p-1)\pi\delta^2}{4\left(  \ln\frac{1}{\varepsilon}\right)^2}q(z_j)^2\sqrt{\det K(z_j)}\\
&-\sum_{j=1}^m\frac{2\pi^2\delta^2 q(z_j)^2  \det K(z_j)}{\left( \ln\frac{1}{\varepsilon}\right)^2}\bar{S}_K(z_j,z_j)\\
&-\sum_{1\leq i\neq j\leq m}\frac{2\pi^2\delta^2q(z_i)q(z_j)\sqrt{\det K(z_i)}\sqrt{\det K(z_j)}}{\left( \ln\frac{1}{\varepsilon}\right) ^2}G_K(z_i,z_j)+O\left( \frac{\delta^2\left( \ln|\ln\varepsilon|\right) ^2}{|\ln\varepsilon|^3}\right).
\end{split}
\end{equation}
\end{proposition}
\begin{proof}
Note that
\begin{equation}\label{218}
\begin{split}
I_\delta(V_{\delta,Z})=&\frac{1}{2}\int_{\Omega}-\delta^2\text{div}\left( K(x)\nabla V_{\delta,Z}  \right)V_{\delta,Z}-\frac{1}{p+1}\int_{\Omega}(V_{\delta,Z}-q)^{p+1}_+\\
%=&\frac{1}{2}\sum_{j=1}^m\int_{\Omega}(V_{\delta,z_j,\hat{q}_{\delta,j}}-\hat{q}_{\delta,j})^p_+V_{\delta,Z}-\frac{1}{p+1}\int_{\Omega}(V_{\delta,Z}-q)^{p+1}_+\\
=&\frac{1}{2}\sum_{j=1}^m\int_{\Omega}(V_{\delta,z_j,\hat{q}_{\delta,j}}-\hat{q}_{\delta,j})^p_+V_{\delta,Z,j} +\frac{1}{2}\sum_{1\leq i\neq j\leq m}\int_{\Omega}(V_{\delta,z_j,\hat{q}_{\delta,j}}-\hat{q}_{\delta,j})^p_+V_{\delta,Z,i}\\
&-\frac{1}{p+1}\int_{\Omega}(V_{\delta,Z}-q)^{p+1}_+.
\end{split}
\end{equation}
By   the definition of $ V_{\delta,Z,j} $, we have
\begin{equation*}
\begin{split}
\int_{\Omega}(&V_{\delta,z_j,\hat{q}_{\delta,j}}-\hat{q}_{\delta,j})^p_+V_{\delta,Z,j}\\
%=&\int_{\Omega}(V_{\delta,z_j,\hat{q}_{\delta,j}}-\hat{q}_{\delta,j})^p_+ (V_{\delta,z_j,\hat{q}_{\delta,j}}+H_{\delta,z_j,\hat{q}_{\delta,j}})\\
=&\hat{q}_{\delta,j}\int_{\Omega}(V_{\delta,z_j,\hat{q}_{\delta,j}}-\hat{q}_{\delta,j})^p_++\int_{\Omega}(V_{\delta,z_j,\hat{q}_{\delta,j}}-\hat{q}_{\delta,j})^{p+1}_++\int_{\Omega}(V_{\delta,z_j,\hat{q}_{\delta,j}}-\hat{q}_{\delta,j})^p_+H_{\delta,z_j,\hat{q}_{\delta,j}}.
\end{split}
\end{equation*}
By the definition of $ V_{\delta,z_j,\hat{q}_{\delta,j}} $, the fact that $ T_{z_j}^{-1}(T_{z_j}^{-1})^t=K(z_j) $ and \eqref{PI}, we get
\begin{equation*}
\begin{split}
\hat{q}_{\delta,j}&\int_{\Omega}(V_{\delta,z_j,\hat{q}_{\delta,j}}-\hat{q}_{\delta,j})^p_+\\
=&\hat{q}_{\delta,j} s_{\delta,j}^2(\frac{\delta}{s_{\delta,j}})^{\frac{2p}{p-1}}\int_{|T_{z_j}x|\leq 1}\phi(T_{z_j}x)^pdx\\
%=&\hat{q}_{\delta,j} s_{\delta,j}^2(\frac{\delta}{s_{\delta,j}})^{\frac{2p}{p-1}}\sqrt{\det(K(z_j))}\cdot 2\pi|\phi'(1)|\\
=&\hat{q}_{\delta,j}\delta^2|\phi'(1)|^{p-1}\left( \frac{\ln\frac{1}{s_{\delta,j}}}{\hat{q}_{\delta,j}}\right)^{p-1}|\phi'(1)|^{-p}\left( \frac{\ln\frac{1}{s_{\delta,j}}}{\hat{q}_{\delta,j}}\right)^{-p}\sqrt{\det(K(z_j))}\cdot 2\pi|\phi'(1)|\\
=&\frac{2\pi\delta^2}{\ln\frac{1}{s_{\delta,j}}}\hat{q}_{\delta,j}^2\sqrt{\det(K(z_j))}.
\end{split}
\end{equation*}
Similarly,
\begin{equation*}
\begin{split}
\int_{\Omega}(V_{\delta,z_j,\hat{q}_{\delta,j}}-\hat{q}_{\delta,j})^{p+1}_+
=& s_{\delta,j}^2\left( \frac{\delta}{s_{\delta,j}}\right)^{\frac{2(p+1)}{p-1}}\sqrt{\det(K(z_j))}\cdot \frac{(p+1)\pi}{2}|\phi'(1)|^2\\
%=&\delta^2|\phi'(1)|^{p-1}\left( \frac{\ln\frac{1}{s_{\delta,j}}}{\hat{q}_{\delta,j}}\right)^{p-1}|\phi'(1)|^{-(p+1)}\left( \frac{\ln\frac{1}{s_{\delta,j}}}{\hat{q}_{\delta,j}}\right)^{-(p+1)}\sqrt{det(K(z_j))}\cdot \frac{(p+1)\pi}{2}|\phi'(1)|^2\\
=&\frac{(p+1)\pi\delta^2}{2\left(  \ln\frac{1}{s_{\delta,j}}\right)^2}\hat{q}_{\delta,j}^2\sqrt{\det(K(z_j))}.
\end{split}
\end{equation*}
By Lemma \ref{H estimate}, we have
\begin{equation*}
\begin{split}
\int_{\Omega}(&V_{\delta,z_j,\hat{q}_{\delta,j}}-\hat{q}_{\delta,j})^p_+H_{\delta,z_j,\hat{q}_{\delta,j}}\\
=&\frac{2\pi\hat{q}_{\delta,j}\sqrt{\det K(z_j)}}{\ln \frac{1}{s_{\delta,j}}}\int_{\Omega}(V_{\delta,z_j,\hat{q}_{\delta,j}}-\hat{q}_{\delta,j})^p_+\bar{S}_K(x,z_j)dx+O\left(  \frac{\varepsilon^\gamma}{|\ln\varepsilon|}\int_{\Omega}(V_{\delta, z_j, \hat{q}_{\delta,j}}-\hat{q}_{\delta,j})^{p}_+\right) \\
=&\frac{4\pi^2\delta^2 \bar{S}_K(z_j,z_j)}{\left( \ln\frac{1}{s_{\delta,j}}\right)^2}\hat{q}_{\delta,j}^2\cdot \det(K(z_j))+O\left(\frac{\varepsilon^{2+\gamma}}{|\ln\varepsilon|^{p+1}}\right),
\end{split}
\end{equation*}
from which we deduce
\begin{equation}\label{219}
\begin{split}
\int_{\Omega}(V_{\delta,z_j,\hat{q}_{\delta,j}}-\hat{q}_{\delta,j})^p_+V_{\delta,Z,j}
=&\frac{2\pi\delta^2}{\ln\frac{1}{s_{\delta,j}}}\hat{q}_{\delta,j}^2\sqrt{\det(K(z_j))}+\frac{(p+1)\pi\delta^2}{2\left(  \ln\frac{1}{s_{\delta,j}}\right)^2}\hat{q}_{\delta,j}^2\sqrt{\det(K(z_j))}\\
&+\frac{4\pi^2\delta^2 \bar{S}_K(z_j,z_j)}{\left( \ln\frac{1}{s_{\delta,j}}\right)^2}\hat{q}_{\delta,j}^2  \det(K(z_j))+O\left(\frac{\varepsilon^{2+\gamma}}{|\ln\varepsilon|^{p+1}}\right).
\end{split}
\end{equation}
Similarly by Lemmas \ref{Green expansion}, \ref{H estimate}, the definition of $ \Lambda_{\varepsilon, m} $ and the fact that $ \lim_{\varepsilon\to 0}\varepsilon|\ln\varepsilon|=0 $, for $ 1\leq i\neq j\leq m $
\begin{equation}\label{220}
\begin{split}
\int_{\Omega}(&V_{\delta,z_j,\hat{q}_{\delta,j}}-\hat{q}_{\delta,j})^p_+V_{\delta,Z,i}\\
%=&\int_{\Omega}(V_{\delta,z_j,\hat{q}_{\delta,j}}-\hat{q}_{\delta,j})^p_+\left( \frac{\hat{q}_{\delta,i}}{\ln s_{\delta,i}}\ln |T_{z_i}(x-z_i)|-\frac{2\pi\hat{q}_{\delta,i}\sqrt{\det K(z_i)}}{\ln  s_{\delta,i}} \bar{S}_K(x,z_i) \right)dx+O\left(\frac{\varepsilon^{2+\gamma}}{|\ln\varepsilon|^{p+1}}\right) \\
=&\frac{2\pi\hat{q}_{\delta,i}\sqrt{\det K(z_i)}}{\ln\frac{1}{s_{\delta,i}}}\int_{\Omega}(V_{\delta,z_j,\hat{q}_{\delta,j}}-\hat{q}_{\delta,j})^p_+G_K(x,z_i)dx+O\left(\frac{\varepsilon^{2+\gamma}}{|\ln\varepsilon|^{p+1}}\right)\\
=&\frac{4\pi^2\delta^2G_K(z_j,z_i)}{\ln\frac{1}{s_{\delta,i}}\ln\frac{1}{s_{\delta,j}}}\hat{q}_{\delta,i}\hat{q}_{\delta,j}\sqrt{\det(K(z_i))}\sqrt{\det(K(z_j))}+O\left(\frac{\varepsilon^{2+\gamma}}{|\ln\varepsilon|^{p+1}}\right).
\end{split}
\end{equation}

Finally by \eqref{203},
\begin{equation}\label{223}
\begin{split}
\int_{\Omega}(V_{\delta,Z}-q)^{p+1}_+
=&\sum_{j=1}^m\int_{B_{Ls_{\delta,j}}(z_j)}\left( V_{\delta, z_j, \hat{q}_{\delta,j}}-\hat{q}_{\delta,j}+O\left( \frac{\varepsilon^\gamma}{|\ln\varepsilon|}\right) \right)^{p+1}_+\\
=&\sum_{j=1}^m\int_{\Omega}(V_{\delta, z_j, \hat{q}_{\delta,j}}-\hat{q}_{\delta,j})^{p+1}_++O\left( \frac{\varepsilon^\gamma}{|\ln\varepsilon|}\sum_{j=1}^m\int_{\Omega}(V_{\delta, z_j, \hat{q}_{\delta,j}}-\hat{q}_{\delta,j})^{p}_+\right) \\
=&\sum_{j=1}^m\frac{(p+1)\pi\delta^2}{2\left(  \ln\frac{1}{s_{\delta,j}}\right)^2}\hat{q}_{\delta,j}^2\sqrt{\det(K(z_j))}+O\left(\frac{\varepsilon^{2+\gamma}}{|\ln\varepsilon|^{p+1}}\right).
\end{split}
\end{equation}
Taking \eqref{219}, \eqref{220} and \eqref{223} into \eqref{218}, one has
\begin{equation}\label{225}
\begin{split}
I_\delta(V_{\delta,Z})=&\sum_{j=1}^m\frac{\pi\delta^2}{\ln\frac{1}{s_{\delta,j}}}\hat{q}_{\delta,j}^2\sqrt{\det(K(z_j))}+\sum_{j=1}^m\frac{(p+1)\pi\delta^2}{4\left(  \ln\frac{1}{s_{\delta,j}}\right)^2}\hat{q}_{\delta,j}^2\sqrt{\det(K(z_j))}\\
&+\sum_{j=1}^m\frac{2\pi^2\delta^2 \bar{S}_K(z_j,z_j)}{\left( \ln\frac{1}{s_{\delta,j}}\right)^2}\hat{q}_{\delta,j}^2  \det(K(z_j))\\
&+\sum_{1\leq i\neq j\leq m}\frac{2\pi^2\delta^2G_K(z_j,z_i)}{\ln\frac{1}{s_{\delta,i}}\ln\frac{1}{s_{\delta,j}}}\hat{q}_{\delta,i}\hat{q}_{\delta,j}\sqrt{\det(K(z_i))}\sqrt{\det(K(z_j))}\\
&-\sum_{j=1}^m\frac{\pi\delta^2}{2\left(  \ln\frac{1}{s_{\delta,j}}\right)^2}\hat{q}_{\delta,j}^2\sqrt{\det(K(z_j))}+O\left(\frac{\varepsilon^{2+\gamma}}{|\ln\varepsilon|^{p+1}}\right).
\end{split}
\end{equation}
Taking \eqref{3-003}, \eqref{q_i choice} and \eqref{2000} into \eqref{225}, we get
%\begin{equation*}
%\begin{split}
%I_\delta(V_{\delta,Z})=&\sum_{j=1}^m\frac{\pi\delta^2}{\ln\frac{1}{\varepsilon}}q(z_j)^2\sqrt{\det K(z_j)}+\sum_{j=1}^m\frac{(p-1)\pi\delta^2}{4\left(  \ln\frac{1}{\varepsilon}\right)^2}q(z_j)^2\sqrt{\det K(z_j)}\\
%&-\sum_{j=1}^m\frac{2\pi^2\delta^2 q(z_j)^2  \det K(z_j)}{\left( \ln\frac{1}{\varepsilon}\right)^2}\bar{S}_K(z_j,z_j)\\
%&-\sum_{1\leq i\neq j\leq m}\frac{2\pi^2\delta^2q(z_i)q(z_j)\sqrt{\det K(z_i)}\sqrt{\det K(z_j)}}{\left( \ln\frac{1}{\varepsilon}\right) ^2}G_K(z_i,z_j)+O\left( \frac{\left( \ln|\ln\varepsilon|\right) ^2}{|\ln\varepsilon|^3}\right).
%\end{split}
%\end{equation*}
\eqref{344}. %The proof is thus finished.

\end{proof}
\section{Proof of Theorem \ref{thm1}}
Let $ x_{0} $ be a strict local maximum point  of $ q^2\sqrt{\det(K)} $ in $ \Omega $, i.e., there exists $ \bar{\rho}>0 $ sufficiently small such that $ B_{\bar{\rho}}(x_{0})\Subset\Omega   $ and
\begin{equation}\label{6-01}
q^2\sqrt{\det(K)}(y)< q^2\sqrt{\det(K)}(x_0)\ \ \ \ \forall y\in B_{\bar{\rho}}(x_{0})\backslash\{x_0\}.
\end{equation}

Now we prove the existence of maximizers of $ K_\delta(Z) $ in $  \Lambda_{\varepsilon, m}$.
Note that by Propositions \ref{pro401} and \ref{order of main term},
\begin{equation}\label{6-02}
\begin{split}
K_\delta(Z)=&\sum_{j=1}^m\frac{\pi\delta^2}{\ln\frac{1}{\varepsilon}}q^2\sqrt{\det K}(z_j)%-\sum_{j=1}^m\frac{2\pi^2\delta^2 q(z_j)^2  \det K(z_j)}{\left( \ln\frac{1}{\varepsilon}\right)^2}\bar{S}_K(z_j,z_j)
-\sum_{1\leq i\neq j\leq m}\frac{2\pi^2\delta^2q(z_i)q(z_j)\sqrt{\det K(z_i)}\sqrt{\det K(z_j)}}{\left( \ln\frac{1}{\varepsilon}\right) ^2}G_K(z_i,z_j)\\
&+O\left( \frac{\delta^2 }{|\ln\varepsilon|^2}\right).
\end{split}
\end{equation}
We have
\begin{lemma}\label{finite dimen solu}
For any $ \delta $ sufficiently small, the following maximization problem
\begin{equation*}
\max\limits_{Z\in \overline{\Lambda_{\varepsilon, m}}}K_\delta(Z)
\end{equation*}
has a solution $ Z_\delta\in\Lambda_{\varepsilon, m}. $
\end{lemma}

\begin{proof}
Clearly $ K_\delta $ has a maximizer in $ \overline{\Lambda_{\varepsilon, m}} $.
Let $ Z_\delta=(z_{1,\delta}, \cdots, z_{m,\delta})\in\overline{\Lambda_{\varepsilon, m}} $ be a maximizer. It suffices to prove that $ Z_\delta\in\Lambda_{\varepsilon, m}. $ We choose a test function
\begin{equation*}
z_j^0=x_0+\frac{1}{\sqrt{|\ln\varepsilon|}}\hat{z}_j^0,
\end{equation*}
where $ \hat{z}_j^0=\left (\cos\frac{(j-1)\pi}{m}, \sin\frac{(j-1)\pi}{m}\right ) $, $  j=1,\cdots,m $, form a $ m $-regular polygon with radius 1 in $ \mathbb{R}^2 $. Then it is easy to see that $ (z_1^0,\cdots, z_m^0)\in\Lambda_{\varepsilon, m} $ since $ |z_j^0-z_i^0|\geq C|\ln\varepsilon|^{-\frac{1}{2}}\geq |\ln\varepsilon|^{-M} $. Using Lemma \ref{Green expansion}, \eqref{6-01} and \eqref{6-02}, one computes directly that
\begin{equation}\label{6-03}
\begin{split}
\max\limits_{Z\in \Lambda_{\varepsilon, m}}K_\delta(Z)\geq& K_\delta((z_1^0,\cdots, z_m^0))\\
\geq & \frac{m\pi\delta^2q^2\sqrt{\det K}(x_0) }{|\ln \varepsilon|}
-\sum_{1\leq i\neq j\leq m}\frac{\pi\delta^2q^2\sqrt{\det K}(z_j^0)}{  |\ln \varepsilon|  ^2}\ln\frac{1}{|z_i^0-z_j^0|}
+O\left( \frac{\delta^2 }{|\ln\varepsilon|^2}\right)\\
\geq & \frac{m\pi\delta^2q^2\sqrt{\det K}(x_0) }{|\ln \varepsilon|}
-\frac{m(m-1)\pi\delta^2q^2\sqrt{\det K}(x_0)}{ 2 |\ln \varepsilon|  ^2}\ln|\ln\varepsilon|
+O\left( \frac{\delta^2 }{|\ln\varepsilon|^2}\right).
\end{split}
\end{equation}

We assume that $ (z_{1,\delta}, \cdots, z_{m,\delta})\in \partial\Lambda_{\varepsilon, m} $. There are two possibilities: either there exists a $ j_0 $ such that $ z_{j_0,\delta}\in \partial B_{\bar{\rho}}(x_{0}),  $ in which case, $ q^2\sqrt{\det K}(z_{j_0,\delta})\leq  q^2\sqrt{\det K}(x_0)-\sigma_0  $ for some $ \sigma_0>0 $; or there exists $ i_0\neq j_0 $ such that $ |z_{i_0,\delta}-z_{j_0,\delta}|=|\ln\varepsilon|^{-M}. $

In the first case, we have
\begin{equation}\label{6-04}
\begin{split}
\max\limits_{ \Lambda_{\varepsilon, m}}K_\delta \leq & \frac{\pi\delta^2\left( mq^2\sqrt{\det K}(x_0)-\sigma_0 \right) }{|\ln \varepsilon|}
+O\left( \frac{\delta^2 \ln|\ln\varepsilon|}{|\ln\varepsilon|^2}\right),
\end{split}
\end{equation}
which contradicts \eqref{6-03} for $ \varepsilon $ sufficiently small. This also shows that $ \lim_{\varepsilon\to 0} q^2\sqrt{\det K}(z_{j_0,\delta})=q^2\sqrt{\det K}(x_0)  $.  By assumptions, we have $ \lim_{\varepsilon\to 0}  z_{j_0,\delta} =x_0. $

In the second case, by \eqref{6-02} we have
\begin{equation}\label{6-05}
\begin{split}
\max\limits_{ \Lambda_{\varepsilon, m}}K_\delta \leq&  \frac{m\pi\delta^2q^2\sqrt{\det K}(x_0) }{|\ln \varepsilon|}
- \frac{\pi\delta^2q^2\sqrt{\det K}(z_{j_0,\delta})}{  |\ln \varepsilon|  ^2}\ln\frac{1}{|z_{j_0,\delta}-z_{j_0,\delta}|}
+O\left( \frac{\delta^2 }{|\ln\varepsilon|^2}\right)\\
\leq &  \frac{m\pi\delta^2q^2\sqrt{\det K}(x_0) }{|\ln \varepsilon|}
- \frac{M\pi\delta^2q^2\sqrt{\det K}(z_{j_0,\delta})}{  |\ln \varepsilon|  ^2}\ln|\ln\varepsilon|
+O\left( \frac{\delta^2 }{|\ln\varepsilon|^2}\right).
\end{split}
\end{equation}
Combining \eqref{6-03} with \eqref{6-05}, we get
\begin{equation*}
\frac{M\pi\delta^2q^2\sqrt{\det K}(z_{j_0,\delta})}{  |\ln \varepsilon|  ^2}\ln|\ln\varepsilon|\leq \frac{m(m-1)\pi\delta^2q^2\sqrt{\det K}(x_0)}{ 2 |\ln \varepsilon|  ^2}\ln|\ln\varepsilon|.
\end{equation*}
This clearly contradicts with the choice of $ M=m^2+1 $ for $ \varepsilon $ sufficiently small. Thus we get  $ Z_\delta\in\Lambda_{\varepsilon, m} $.

\end{proof}

{\bf Proof of Theorem \ref{thm1}:}
From Lemma \ref{finite dimen solu}, we know that for $ \delta>0 $ sufficiently small, there exists   $ Z_\delta=(z_{1,\delta}, \cdots, z_{m,\delta}) $ being a critical point of $ K_\delta(Z) $ in $ \Lambda_{\varepsilon, m} $ and as $ \delta\to 0 $,
\begin{equation*}
(z_{1,\delta}, \cdots, z_{m,\delta})\to (x_{0},\cdots,x_{0}).
\end{equation*}
Lemma \ref{choice of Z} then guarantees that $ v_\delta= \sum_{j=1}^mV_{\delta, Z,j}+\omega_{\delta,Z } $ is a clustered solution to \eqref{111}.

Let $ u_\varepsilon=|\ln\varepsilon|v_\delta $ and $ \delta=\varepsilon|\ln\varepsilon|^{-\frac{p-1}{2}} $, then $ u_\varepsilon $ is a  solution to \eqref{eq1-1}. Define $ A_{\varepsilon,i}=\{u_\varepsilon>q\ln\frac{1}{\varepsilon}\}\cap B_{|\ln\varepsilon|^{-M-1}}(z_{i,\delta})  $. From Lemma \ref{lemA-5}, there exist $ R_1, R_2>0 $ such that
\begin{equation*}
B_{R_1\varepsilon}(z_{i,\delta})\subseteq A_{\varepsilon,i}\subseteq B_{R_2\varepsilon}(z_{i,\delta}) .
\end{equation*}
It remains to  calculate the limiting value of $\frac{1}{\varepsilon^2}\int_{\Omega}\left (u_\varepsilon-q\ln\frac{1}{\varepsilon}\right )^{p}_+dx. $ %Define $ \kappa_i(u_\varepsilon)=\frac{1}{\varepsilon^2}\int_{B_{\bar{\rho}}(x_{0,i})}(u_\varepsilon-q\ln\frac{1}{\varepsilon})^{p}_+dx. $
We have
\begin{lemma}\label{circulation}
There holds for $ i=1,\cdots,m $
\begin{equation*}
\lim_{\varepsilon\to 0}\frac{1}{\varepsilon^2}\int_{B_{|\ln\varepsilon|^{-M-1}}(z_{i,\delta})}\left (u_\varepsilon-q\ln\frac{1}{\varepsilon}\right )^{p}_+dx=2\pi q\sqrt{\det K}(x_{0}).
\end{equation*}
As a consequence,
\begin{equation*}
\lim_{\varepsilon\to 0}\frac{1}{\varepsilon^2}\int_{\Omega}\left (u_\varepsilon-q\ln\frac{1}{\varepsilon}\right )^{p}_+dx=2\pi m q\sqrt{\det K}(x_{0}).
\end{equation*}
\end{lemma}
\begin{proof}
	
It follows from  \eqref{q_i choice}, \eqref{203} and Proposition \ref{exist and uniq of w} that
\begin{equation*}
\begin{split}
\frac{1}{\varepsilon^2}&\int_{B_{|\ln\varepsilon|^{-M-1}}(z_{i,\delta})}\left (u_\varepsilon-q\ln\frac{1}{\varepsilon}\right )^{p}_+dx
=	\frac{|\ln\varepsilon|^p}{\varepsilon^2}\int_{B_{|\ln\varepsilon|^{-M-1}}(z_{i,\delta})}(w_\delta-q)^{p}_+dx\\
&=\frac{|\ln\varepsilon|^p}{\varepsilon^2}\int_{B_{Ls_{\delta,i}}(z_{i,\delta})}\left(V_{\delta, z_{i,\delta}, \hat{q}_{\delta,i}}(x)-\hat{q}_{\delta,i}+O\left( \frac{\varepsilon^\gamma}{|\ln\varepsilon|}\right)\right)^{p}_+dx\\
%&=\frac{|\ln\varepsilon|^p}{\varepsilon^2}s_{\delta,i}^2\left( \frac{\delta}{s_{\delta,i}}\right)^{\frac{2p}{p-1}}\int_{|T_{z_{i,\delta}}x|\leq 1}\phi(T_{z_{i,\delta}}x)^pdx+o(1)\\
&=\frac{|\ln\varepsilon|}{\delta^2}\delta^2|\phi'(1)|^{p-1}\left( \frac{\ln\frac{1}{s_{\delta,i}}}{\hat{q}_{\delta,i}}\right)^{p-1}|\phi'(1)|^{-p}\left( \frac{\ln\frac{1}{s_{\delta,i}}}{\hat{q}_{\delta,i}}\right)^{-p}\sqrt{\det K(z_{i,\delta})}\cdot 2\pi|\phi'(1)|+o(1)\\
%&= \frac{2\pi\hat{q}_{\delta,i} |\ln\varepsilon|}{\ln\frac{1}{s_{\delta,i}}}\cdot \sqrt{det(K(z_{i,\delta}))}+o(1) \\
&\to 2\pi q \sqrt{\det K }(x_{0})\ \ \ \  \text{as}\ \delta\to0.
\end{split}
\end{equation*}

\end{proof}
The rest of properties of $ u_\varepsilon $ can be easily deduced from the decomposition of $ v_\delta $ in \eqref{solu config} and we finish the proof of Theorem \ref{thm1}.

\section{Proof of Theorem \ref{thm2}}
It suffices to consider solutions to the problem
\begin{equation}\label{eq01}
\begin{cases}
-\varepsilon^2\text{div}(K_H(x)\nabla u)=  \left(u-\left( \frac{\alpha|x|^2}{2}+\beta\right)\ln\frac{1}{\varepsilon}\right)^{p}_+,\ \ &x\in B_{R^*}(0),\\
u=0,\ \ &x\in\partial  B_{R^*}(0).
\end{cases}
\end{equation}
%where $ \alpha,\beta  $ are any given constants satisfying $ \min_{x\in B_{R^*}(0)}\frac{\alpha|x|^2}{2}+\beta>0 $.
Let $ v=u\backslash|\ln\varepsilon|  $ and $ \delta=\varepsilon|\ln\varepsilon|^{-\frac{p-1}{2}} $, then
\begin{equation}\label{eq02}
\begin{cases}
-\delta^2\text{div}(K_H(x)\nabla v)=  \left( v-\left( \frac{\alpha|x|^2}{2}+\beta\right) \right)^{p}_+,\ \ &x\in B_{R^*}(0),\\
v=0,\ \ &x\in\partial  B_{R^*}(0).
\end{cases}
\end{equation}
Note that \eqref{eq02} coincides with \eqref{111} with $ q=\frac{\alpha|x|^2}{2}+\beta $, $ K=K_H $ and $ \Omega=B_{R^*}(0) $. However,  since $ q^2\sqrt{det(K_H)} $ is a radial function and the set of  extreme points is rotational-invariant, results of Theorem \ref{thm2} can not be deduced directly from those of Theorem \ref{thm1}. %In this case one can also use the reduction procedure to construct solutions of \eqref{eq02}, by using the rotational symmetry of $ K_H, q $ and the domain $ B_{R^*}(0). $

%Let $ h(r)=h(|x|)=q^2\sqrt{det(K_H)}(x) $ for any $ x\in B_{R^*}(0) $. We call $ z^* $ is a  $ strict $ $ local $ $ maximum $ $ (minimum) $ $ point $ of  $q^2\sqrt{det(K_H)}$  $ up $ $ to $ $ ratation $ in $ B_{R^*}(0) $, if $ |z^*| $ is a strict local maximum (minimum) point of $ h $ in $ (0,{R^*}) $.

Let $ x_0  $ be a strict local maximizer of $ q^2\sqrt{\det K_H} $ up to a rotation. %Define $ \mathcal{N}=B_{\bar{\rho}}(z_1) $. Then we can  construct solutions of \eqref{eq02} being of the form $ v_\delta=PV_{\delta,z,\hat{q}}+\omega_{\delta,z} $, where $ z $ is near $ z_1. $
By Lemma \ref{coercive esti} and Proposition \ref{exist and uniq of w}, for any $ Z\in \Lambda_{\varepsilon, m} $ there exists a unique $ \omega_{\delta,Z}\in E_{\delta,Z} $ such that $ Q_\delta L_\delta \omega_{\delta,Z}=Q_\delta l_\delta+Q_\delta R_\delta(\omega_{\delta,Z}). $ So it remains to prove the existence of  maximizers $ Z=Z_{\delta} $ of $ K_\delta  $  near $ x_0 $.  Indeed,  by the rotational symmetry of $ q $ and  $ \det K_H $, one computes directly that
\begin{equation}\label{501}
\begin{array}{ll}
K_\delta(Z)=\sum_{j=1}^m\frac{\pi\delta^2}{\ln\frac{1}{\varepsilon}}q^2\sqrt{\det K}(z_j)-\sum_{1\leq i\neq j\leq m}\frac{2\pi^2\delta^2q(z_i)q(z_j)\sqrt{\det K(z_i)}\sqrt{\det K(z_j)}}{\left( \ln\frac{1}{\varepsilon}\right) ^2}G_K(z_i,z_j)&\\
\qquad\qquad\,+\tilde{N}_\delta(Z),&
\end{array}
\end{equation}
where $ \tilde{N}_\delta(Z) $ is a $ O\left (\frac{\delta^2}{|\ln\varepsilon|^2}\right ) -$perturbation term which is invariant under a rotation. %In fact, we can choose $ T_x^{-1}=\begin{pmatrix}
%\cos\theta_x & -\sin\theta_x  \\
%\sin\theta_x &\cos\theta_x
%\end{pmatrix}\begin{pmatrix}
%\frac{k}{\sqrt{k^2+|x|^2}} & 0  \\
%0 &1
%\end{pmatrix}\begin{pmatrix}
%\cos\theta_x &  \sin\theta_x  \\
%-\sin\theta_x &\cos\theta_x
%\end{pmatrix} $ in \eqref{T_z choice}, where $ (|x|,\theta_x) $ is the polar coordinate of $ x $. By the rotational symmetry of $ q $ and  the domain $ B_R(0) $, one can  prove that for every $ \theta\in[0,2\pi] $,  if we define $ \bar{z}=\bar{R}_\theta(z) $, then
%\begin{equation*}
%V_{\delta,\bar{z}}(\bar{R}_\theta(x))=V_{\delta,z}(x)\ \text{and} \ \omega_{\delta,\bar{z}}(\bar{R}_\theta(x))=\omega_{\delta,z}(x)    \ \ \text{for any}\  x\in B_{R^*}(0).
%\end{equation*}
%Hence one computes directly  that $ P_\delta(\bar{z})=P_\delta(z) $, i.e., $  P_\delta $ is a radially symmetric function. Note that $ q^2\sqrt{det(K_H)}(x)=\left( \frac{\alpha|x|^2}{2}+\beta\right)^2\cdot \frac{k}{\sqrt{k^2+|x|^2}} $ is also radially symmetric. Thus by  Proposition \ref{pro401} and \ref{order of main term},  we have \eqref{501}.
%Since $ z_1 $ is a strict local maximum (minimum) point of $ q^2\sqrt{det(K_H)} $ up to rotation,
So it is not hard to prove  the existence of $ Z_{\delta} $ near $ x_0 $ being a maximizer of $ K_\delta  $, which yields a solution $ v_\delta $ of \eqref{eq02}. Let $ u_\varepsilon=v_\delta|\ln\varepsilon| $, then $ u_\varepsilon $ is a solution of \eqref{eq01}. Moreover, one has
\begin{equation*}
\lim_{\varepsilon\to 0}\frac{1}{\varepsilon^2}\int_{\Omega}\left( u_\varepsilon-q\ln\frac{1}{\varepsilon}\right)^{p}_+dx=2\pi m q \sqrt{\det(K_H)}(x_0)=\frac{km\pi(\alpha|x_0|^2+2\beta)}{\sqrt{k^2+|x_0|^2}}.
\end{equation*}
\textbf{Proof of  Corollary \ref{cor3}}: We choose $ \alpha $ and $ \beta $ such that $ \alpha<0 $ and $ \min_{x\in B_{R^*}(0)}\left( \frac{\alpha|x|^2}{2}+\beta\right)>0 $ in Theorem \ref{thm2}. Then $ (0,0) $ is the unique strict local maximizer of $ \left( \frac{\alpha|x|^2}{2}+\beta\right)^2\sqrt{\det K_H} $ up to a rotation. Thus by Theorem \ref{thm2}, for any $ m\in \mathbb{N}^* $ there exist clustered helical rotational-invariant vorticity fields $ \textbf{w}_\varepsilon $ to \eqref{Euler eq2}  with angular velocity $ \alpha|\ln\varepsilon| $, whose support sets are $ m $ helical tubes and collapse into $ x_3- $axis as $\varepsilon\to 0$. Moreover, the circulations satisfy as $\varepsilon\to 0$
\begin{equation*}
\int_{B_{R^*}(0)} \omega_\varepsilon dx\to 2\pi m \beta.
\end{equation*}

\subsection*{Acknowledgments:}

\par
D. Cao was partially supported by National Key R \& D
Program ( 2022YFA1005602 ). J. Wan was supported by NNSF of China (grant No. 12101045 and 12271539).

\subsection*{Conflict of interest statement} On behalf of all authors, the corresponding author states that there is no conflict of interest.

\subsection*{Data availability statement} All data generated or analysed during this study are included in this published article  and its supplementary information files.


\begin{thebibliography}{99}

\bibitem{AS}
H. Abidi and S. Sakrani, Global well-posedness of helicoidal Euler equations, \textit{J. Funct. Anal.}, 271(8)(2016), 2177--2214.

\bibitem{ALW}
W. Ao, Y. Liu and J. Wei,  Clustered travelling vortex rings to the axisymmetric three-dimensional incompressible Euler flows, \textit{Phys. D}, 434(2022), Paper No. 133258.

\bibitem{Ben}
M. Benvenutti, Nonlinear stability for stationary helical vortices, \textit{NoDEA Nonlinear Differential Equations Appl.}, 27(2020), no. 2, Paper No. 15.

\bibitem{BLN}
A.C. Bronzi, M.C. Lopes Filho and H.J. Nussenzveig Lopes,  Global existence of a weak solution of the incompressible Euler equations with helical symmetry and $ L^p $ vorticity, \textit{Indiana Univ. Math. J.}, 64(1)(2015), 309--341.


%\bibitem{Bur}
%G.R. Burton, Rearrangements of functions, maximization of convex functionals, and
%vortex rings, Math. Ann. 276(2)(1987) 225--253.


%\bibitem{CF}
%L.A. Caffarelli and A. Friedman, Asymptotic estimates for the plasma problem, Duke Math. J. 47 (1980) 705--742.

\bibitem{CGPY}
D. Cao, Y. Guo, S. Peng and S. Yan, Local uniqueness for vortex patch problem in incompressible planar steady flow, \textit{J. Math. Pures Appl.}, 131(2019),
251--289.



\bibitem{CLW}
D. Cao, Z. Liu and J. Wei, Regularization of point vortices  for the
Euler equation in dimension two, \textit{Arch. Ration. Mech. Anal.},
212(2014), 179--217.

\bibitem{CPY1}
D. Cao, S. Peng and S. Yan, Multiplicity of solutions for the plasma problem
in two dimensions,  \textit{Adv. Math.}, 225(2010), 2741--2785.

\bibitem{CPY}
D. Cao, S. Peng and S. Yan, Planar vortex patch problem in incompressible steady flow,  \textit{Adv. Math.}, 270(2015), 263--301.

\bibitem{CW}
D. Cao and J. Wan, Helical vortices with small cross-section for 3D incompressible Euler equation, \textit{J. Funct. Anal.}, 284(2023), Paper No. 109836.

\bibitem{CW1}
D. Cao and J. Wan, Structure of Green's function of elliptic equations and helical vortex patches for 3D incompressible Euler equations, \textit{Math. Ann.},
https://doi.org/10.1007/s00208-023-02589-8.

\bibitem{CW2}
D. Cao and J. Wan, Desingularization of 3D steady Euler equations with helical symmetry, \textit{Calc. Var. \& Partial Differential Equations}, to appear, arXiv: 2206.00196.






\bibitem{DY}
E.N. Dancer and S. Yan, The Lazer-McKenna conjecture and a free boundary problem
in two dimensions, \textit{J. Lond. Math. Soc.}, 78(2008), 639--662.

%\bibitem{DR}
%L. S. Da Rios, Sul moto dun liquido indefinito con un filetto vorticoso di forma qualunque, Rendiconti del Circolo Matematico di Palermo (1884-1940) 22(1)(1906) 117--135.

%\bibitem{DR2}
%L.S. Da Rios, Sul moto dei filetti vorticosi di forma qualunque, Rend. R. Acc. Lincei 18(1909)75--79.

\bibitem{DDMW}
J. D$\acute{\text{a}}$vila, M. del Pino, M. Musso and J. Wei, Gluing methods for vortex dynamics in Euler flows, \textit{Arch. Ration. Mech. Anal.}, 235(3)(2020), 1467--1530.

\bibitem{DDMW2}
J. D$\acute{\text{a}}$vila, M. del Pino, M. Musso and J. Wei, Travelling helices and the vortex filament conjecture
in the incompressible Euler equations, \textit{Calc. Var. \& Partial Differential Equations}, 61(2022), art.119.


\bibitem{De2}
J. Dekeyser and J. Van Schaftingen,  Vortex motion for the lake equations, \textit{Comm. Math. Phys.}, 375(2020), 1459--1501.
%\bibitem{De2}
%J. Dekeyser, Asymptotic of steady vortex pair in the lake equation, SIAM J. Math. Anal. 51(2)(2019)1209-1237.

\bibitem{DV}
S. de Valeriola and J. Van Schaftingen, Desingularization of vortex rings and shallow water vortices by semilinear elliptic problem, \textit{Arch. Ration. Mech. Anal.}, 210(2)(2013), 409--450.




\bibitem{Du}
A. Dutrifoy, Existence globale en temps de solutions h$\acute{\text{e}}$lico$\ddot{\text{i}}$dales des $\acute{\text{e}}$quations d'Euler, \textit{C. R. Acad. Sci. Paris S$\acute{\text{e}}$r. I Math.}, 329(7)(1999), 653--656.



\bibitem{ET}
B. Ettinger and E.S. Titi, Global existence and uniqueness of weak solutions of three-dimensional Euler equations with helical symmetry in the absence of vorticity stretching, \textit{SIAM J. Math. Anal.}, 41(1)(2009), 269--296.


\bibitem{Fr1}
L.E. Fraenkel, On steady vortex rings of small cross-section in an ideal fluid, \textit{Proc. R. Soc. Lond. A.}, 316(1970), 29--62.

\bibitem{FB}
L.E. Fraenkel and M.S. Berger, A global theory of steady vortex rings in an ideal fluid, \textit{Acta Math.}, 132(1974), 13--51.



\bibitem{GT}
D. Gilbarg and N.S. Trudinger,  Elliptic Partial Differential Equations of Second Order, Classics in Mathematics, Springer, Berlin, 2001.

%\bibitem{G}
%W.H. Greub, Linear Algebra, Third edition, Die Grundlehren der Mathematischen Wissenschaften, Band 97 Springer-Verlag New York, Inc., New York 1967, xiii+434 pp.

\bibitem{GM}
I. Guerra and M. Musso, Cluster of vortex helices in the incompressible 3d Euler equations, arXiv:2304.14025.

\bibitem{He}
H. Helmholtz, On integrals of the hydrodynamics equations which express vortex motion, \textit{J. Reine Angew. Math.}, 55(1858), 25--55.


%\bibitem{JS}
%R.L. Jerrard and C. Seis, On the vortex filament conjecture for Euler flows, Arch. Ration. Mech. Anal. 224(1)(2017)135--172.

\bibitem{JS2}
R.L. Jerrard and D. Smets, On the motion of a curve by its binormal curvature, \textit{J. Eur. 	Math. Soc.}, (JEMS) 17(6)(2015), 1487--1515.

\bibitem{Jiu}
Q. Jiu, J. Li and D. Niu,  Global existence of weak solutions to the three-dimensional Euler equations with helical symmetry, \textit{J. Differential Equations}, 262(10)(2017), 5179--5205.

\bibitem{Ku}
W. Kulpa, The Poincar$\acute{\text{e}}$--Miranda theorem, \textit{Amer. Math. Mon.},  104(1997), 545--550.


%\bibitem{Lamb}
%H. Lamb, Hydrodynamics Cambridge Mathematical Library, 6th edition. Cambridge University Press, Cambridge, (1932).

%\bibitem{LC}
%T. Levi-Civita, Sullattrazione esercitata da una linea materiale in punti prossimi alla linea stessa, Rend. R. Acc. Lincei, 17 (1908)3--15.

%\bibitem{LC3}
%T. Levi-Civita,  Sulla gravitazione di un tubo sottile con applicazione all'anello di Satumo, Rend. Circ. Mat. Palermo, 33 (1912) 354--374.

%\bibitem{LC2}
%T. Levi-Civita, Attrazione newtoniana dei tubi sottili e vortici filiformi, Annali della Scuola Normale Superiore di Pisa - Classe di Scienze Ser. 2, 1(3) (1932) 229--250.

%\bibitem{LYY}
%G. Li, S. Yan and J. Yang, An elliptic problem related to planar vortex pairs, SIAM J. Math. Anal. 36(2005) 1444--1460.

%\bibitem{Li}
%C.C. Lin, On the motion of vortices in two dimension - I. Existence of the Kirchhoff-Routh function, Proc. Natl. Acad. Sci. USA 27(1941) 570--575.

\bibitem{MB}
A. Majda and A. Bertozzi, Vorticity and incompressible flow, Cambridge University Press, Cambridge, 2002.

\bibitem{MR}
Y. Martel and P. Rapha$\ddot{\text{e}}$l, Strongly interacting blow up bubbles for the mass critical NLS, arXiv: 1512.00900.



%\bibitem{MP}
%C. Marchioro and M. Pulvirenti, Mathematical Theory of Incompressible Nonviscous Fluids, Springer-Verlag, 1994.

%\bibitem{Ric}
%R.L. Ricca, Rediscovery of Da Rios equations, Nature 352 (1991) 561--562.


%\bibitem{Ric2}
%R.L. Ricca, The contributions of Da Rios and Levi-Civita to asymptotic potential theory and vortex filament dynamics, Fluid Dyn. Res. 18(5)(1996) 245--268.



\bibitem{SV}
D. Smets and J. Van Schaftingen, Desingularization of vortices for the Euler equation, \textit{Arch. Ration. Mech. Anal.}, 198(3)(2010), 869--925.

\bibitem{WYZ}
J. Wei, D. Ye and F. Zhou, Bubbling solutions for an anisotropic Emden–Fowler equation, \textit{Calc. Var. Partial Differential Equations}, 28(2007), 217--247.

\bibitem{Y}
V.I. Yudovich, Non-stationary flow of an ideal incompressible fluid,
\textit{USSR Comp. Math. Math. Phys.}, 3(1963), 1407--1456.



\end{thebibliography}
\end{document}